%% file: OverconvergentSubanalytic.tex
\documentclass{amsart}
\usepackage{amsmath,amsthm,amssymb}
\usepackage{enumerate}
\usepackage[all]{xy}
\usepackage{amsfonts}
\usepackage{mathrsfs}
\usepackage{dsfont,fancybox}
\usepackage{xcolor,graphics}
\usepackage{float}
\usepackage{cite}
\usepackage{comment}
\usepackage{hyperref}

\newtheorem*{theo_nn}{Theorem}
\newtheorem*{prop_nn}{Proposition}
\newtheorem{theo}{Theorem}[section]
\newtheorem{prop}[theo]{Proposition}
\newtheorem{lemme}[theo]{Lemma}
\newtheorem{cor}[theo]{Corollary}
\newtheorem{rem}[theo]{Remark}
\newtheorem{exem}[theo]{Example}

\theoremstyle{definition}

\newtheorem{defi}[theo]{Definition}

\numberwithin{equation}{section}

\input{macros}

\date{\today}

\begin{document}

\title{Overconvergent subanalytic subsets in the framework of Berkovich spaces}

\author{Florent Martin}
%\address{Florent Martin:  Fakult\"{a}t f\"{u}r Mathematik, Universit\"{a}tstr. 31, 93040 Regensburg, Germany; \email{Florent.Martin@mathematik.uni-regensburg.de}
%}
\address{Florent Martin, Fakult\"{a}t f\"{u}r Mathematik, Universit\"{a}t Regensburg, 93040 Regensburg, Germany}
\email{florent.martin@mathematik.uni-regensburg.de}
\urladdr{http://homepages.uni-regensburg.de/$\sim$maf55605/}

%%%%%%%%

\begin{abstract}
We study the class of overconvergent subanalytic subsets 
of a $k$-affinoid space $X$ when $k$ is a non-archimedean field. 
These are the images along the projection 
$X \times \B^n \to X$ of subsets defined with inequalities between functions of 
$X\times \B^n$ which are overconvergent in the variables of 
$\B^n$.
In particular, we study the local nature, with respect to $X$, of overconvergent subanalytic subsets.
We show that they behave well with respect to the Berkovich topology, but not to the $G$-topology. 
This gives counter-examples to previous results on the subject, and a way to correct them.
Moreover, we study the case dim$(X)=2$, for which a simpler characterisation of 
overconvergent subanalytic subsets is proven. 
%% Keywords are optional
\end{abstract}


\maketitle

\tableofcontents
\section*{Introduction}

\addtocontents{toc}{\protect\setcounter{tocdepth}{1}}
%\subsection{This subsection is numbered but not shown in the toc}

\subsection*{Motivations}
Let us consider a complete non-trivially valued non-archimedean field $k$ (assumed to be algebraically closed in this introduction for simplicity).
Since non-archimedean fields are totally disconnected, one can not 
define the notion of analytic spaces over $k$ as easily as in the case of 
$\mathbb{R}$ or $\mathbb{C}$. Tate \cite{Ta} developed such a theory, and called his spaces 
rigid spaces, whose building blocks are affinoid spaces. However, these spaces are not 
endowed with a classical topology, but with a Grothendieck topology (the $G$-topology). 
Afterwards, V. Berkovich developed 
another viewpoint for $k$-analytic geometry \cite{Berko90,Berko93}. 
His spaces, called $k$-analytic spaces, or Berkovich spaces, have more points than the corresponding rigid spaces and are equipped with a topology which is locally arcwise connected. 
Moreover, in this theory, affinoid spaces are compact. 
R. Huber also developed another viewpoint, 
in the setting of adic spaces \cite{HU96}, and there also exists an approach, initiated by M. Raynaud, using formal geometry (see \cite{BL1} for instance). \par 

If $X, Y$ are $k$-analytic spaces, and $\varphi : Y \to X$ is an analytic map, 
it is natural to wonder what is the shape of $\varphi(Y)$.
By analogy with Chevalley's theorem and the Tarski-Seidenberg theorem, one would like to be able to describe such images 
$\varphi(Y)$ using only functions of  $X$.\par
Without assumption on $\varphi$, this is impossible: one needs some kind 
of \emph{compactness} at some point. 
One reasonable restriction is 
to consider analytic maps 
$\varphi : Y \to X$ where $X$ and $Y$ are affinoid spaces. \par 
In this context the first natural approach it to define a 
\emph{semianalytic set} of a $k$-affinoid space as a finite boolean combination of sets defined by inequalities $\{ |f| \leq |g| \}$  
between analytic functions. 
But the class of semianalytic sets is not big enough: there exist some morphisms of affinoid spaces $\varphi : Y \to X$ such that 
$\varphi (Y)$ is not semianalytic. \par 

To overcome this problem, one has to consider more functions on an affinoid space $X$ than the analytic ones. 
In the framework of $\Zp^n$, Jan Denef and Lou Van den Dries have given \cite{DVdd} a good description of images of analytic maps $\varphi : \Zp^m \to \Zp^n$.  
Their main idea is to allow division of functions. 
In the framework of rigid geometry, where $\Qp$ has to be replaced by some non-archimedean algebraically closed field $k$, 
this idea of allowing divisions has been developed in two ways. \par 

The first one is due to Leonard Lipshitz \cite{LR93,LR_asterisque,Lip_iso,LR_plane} and rests on the introduction of an algebra $S_{m,n}$ of restricted analytic functions on products of closed and open balls. 
This allows L. Lipshitz to define for each affinoid space $X$ the class of subanalytic sets of $X(k)$ (in terms of analytic functions of $X$, division and composition with $S_{m,n}$), and to prove that subanalytic sets are stable under analytic maps between affinoid spaces. \par 
A second approach has been developed by Hans Schoutens in \cite{Sch_sub}. 
This leads to the definition of overconvergent subanalytic sets of $X(k)$. 
Namely, overconvergent subanalytic sets of $X(k)$ form a subclass of the subanalytic sets as defined by L. Lipshitz. 
Overconvergent subanalytic sets are only stable under overconvergent analytic maps between affinoid spaces. 
For instance, if $\varphi: \B^n \to X$ is an analytic map which can be analytically extended to a polydisc of radius $r>1$, then $\varphi(\B^n)$ is an overconvergent subanalytic set of $X$. 

\subsection*{Overconvergent subanalytic sets}
Hans Schoutens used the language of rigid geometry. 
We now summarize his results. 
First let $D: k^2 \to k$ be defined by
\[ 
D(x,y)= 
\begin{cases} 
\frac{x}{y} & \text{if} \  |x| \leq |y| \neq 0 \\
0 & \text{otherwise}
\end{cases} \]
Let $\A$ be an affinoid algebra and $X$ its affinoid space. The algebra $\ADS$ is defined as the smallest $k$-algebra of functions 
$f : X(k) \to k$ such that 
\begin{itemize}
\item $\ADS$ contains the functions induced by $\A$. 
\item If $f, g \in \ADS$, then $D(f,g) \in \ADS$.
\item If $f \in \mathcal{A}\langle Y_1,\ldots,Y_n \rangle$ is overconvergent in the variables $Y_i$, and $g_1,\ldots, g_n \in \ADS$  
satisfy 
$|g_i|_{\text{sup}} \leq 1$, then $f(g_1,\ldots,g_n) \in \ADS$.
\end{itemize}
Stability under overconvergent maps is contained in the following result (we denote by $\B$ the closed unit disc).
\begin{theo_nn} \cite{Sch_sub}
\label{theo_schoutens}
For a subset $S\subset X(k)$ the following are equivalent: 
\begin{itemize}
\item there exists $n\in \N$, a semianalytic set 
$T$ of $X\times \B^n(k)$ defined by inequalities $\{|f| \leq |g|\}$ where $f$ and $g$ are overconvergent 
with respect to the variables of $\B^n$ such that $S= \pi (T)$ where $\pi : X\times \B^n(k) \to X(k)$ is the 
first projection. We call such sets overconvergent subanalytic sets.
\item $S$ is  defined by a boolean combination of inequalities 
$\{ |f| \leq |g| \}$ where $f,g\in \ADS$.
\end{itemize}
\end{theo_nn}  
For instance, if $\varphi :  \B^n \to X$ is an overconvergent map 
(in the sense that it can be extended to a polydisc of radius greater than one), 
then $\varphi( \B^n)$ is overconvergent subanalytic (take for $T$ the graph of $\varphi$).
\subsection*{Results of this text}
In this article we explain how Berkovich spaces are well suited to study overconvergent subanalytic sets. 
Indeed, the definitions that we have given above (semianalytic, overconvergent subanalytic) can be given in the framework of Berkovich spaces. 
For instance if we consider $X = \B^2$ with coordinate functions $T_1,T_2$, 
the inequality $\{ |T_1| \leq |T_2| \}$ naturally defines two sets 
\begin{align*}
S_{\rig} = \{ (t_1,t_2) \in (k^\circ)^2 \st |t_1| \leq |t_2| \} \\
S_{\Berko} = \{ x \in \mathcal{M}(k\{T_1,T_2\}) \st |T_1(x)| \leq |T_2(x)| \} 
\end{align*}
Of course $S_\rig \subset S_\Berko$. More precisely, $S_\rig$ is the set of rigid points of $S_\Berko$. \par 
This new approach with Berkovich spaces allows us to simplify the proof of the theorem of \cite{Sch_sub} mentioned above. 
The part 2 of \cite{Sch_sub} \emph{A combinatorial lemma} is replaced by a simple compactness argument in Berkovich spaces. \par 
If $X$ is an affinoid space, recall that affinoid domains of $X$ are some subsets $S \subset X$ satisfying some universal property with respect to morphism of affinoid spaces 
$f : Y \to X$ such that $\im(f) \subset S$. 
See \cite[7.2.2.2]{BGR} or \cite[2.2.1]{Berko90} for a precise definition.
Weierstrass  (resp. rational) domains are examples of affinoid domains which are defined by inequalities of the form $|f|\leq 1$ (resp. $|f|\leq |g|$) where $f$ and $g$ are analytic functions on $X$. 
Then we consider the local behaviour of overconvergent subanalytic sets. 
If $X$ is an affinoid space there are two ways to consider local behaviour on $X$.

\begin{enumerate}
\item The $G$-topology, where a covering of $X$ is a finite covering $\{X_i\}$ by affinoid domains.  
\item The Berkovich topology \cite{Berko90,Berko93} on $X$ seen as a Berkovich space, which is a real topology.
\end{enumerate}
If $S$ is an overconvergent subanalytic set of $X$ and $U$ an affinoid domain of $X$, it is easy to see that $S\cap U$ is 
an overconvergent subanalytic set of $U$.
It is then natural to wonder if overconvergent subanalytic sets fit well with one of these topologies. We give the following answers.

\begin{prop_nn} \textbf{\ref{CE1}}
There exists an affinoid space $X$, some subset $S \subset X$, and a finite covering $\{X_i\}$ of $X$ by affinoid domains such that 
for all $i$, 
$S\cap X_i$ is overconvergent subanalytic in $X_i$, but $S$ is not overconvergent subanalytic in $X$.
\end{prop_nn}
In other words, being overconvergent subanalytic is not local with respect to the $G$-topology. This contradicts some results of \cite{Sch_sub}, for instance \cite[QE theorem p. 270, Proposition 4.2, Theorem 5.2]{Sch_sub}. \par 
We prove however that the Berkovich topology corrects this. 
If $X$ is an affinoid space seen as a Berkovich space, and $x \in X$, we say that $V$ is an affinoid neighbourhood of $x$ if $V$ is an affinoid 
domain of $X$ and in addition $V$ is a neighbourhood of $X$ w.r.t. the Berkovich topology\footnote{When $x$ is a rigid point, any affinoid domain containing $x$ is an affinoid neighbourhood. 
But this is not true in general.
For instance in the unit disc with coordinate $T$, the rational domain defined by $|T|=1$ is an affinoid domain which contains the Gauss point, but it is not a neighbourhood of the Gauss point.}.
\begin{theo_nn} \textbf{\ref{local}}
A subset $S\subset X$ is overconvergent subanalytic if and only if for every $x\in X$ (seen as a Berkovich space), there 
exists  $V$ an affinoid neighbourhood of $x$ such that $S\cap V$ is overconvergent subanalytic in $V$.
\end{theo_nn}
In other words, being overconvergent subanalytic is a local property, but with respect to the Berkovich topology. \par 
The mistake in \cite{Sch_sub} which we  point out in proposition \ref{CE1} led to other mistakes in further work of H. Schoutens 
\cite{Sch_unif,Sch_plane}.
In particular \cite{Sch_plane} which relies on the false results of \cite{Sch_sub} claims that if  
$k$ is algebraically closed of characteristic $0$, 
then a subset of the unit bidisc  is overconvergent subanalytic if and only if it is rigid-semianalytic 
(i.e. semianalytic locally for the $G$-topology). 
But the counterexample we give in proposition \ref{CE1} proves that this equivalence does not hold. 
Anyway, the proofs of \cite{Sch_plane} rely on some false equivalences of \cite{Sch_sub}. \par 
We show that the Berkovich topology allows one to correct the results of \cite{Sch_plane}.
A $k$-analytic space is said to be good \cite[1.2.16]{Berko93} if any point has an affinoid neighbourhood. 
For instance affinoid spaces are good $k$-analytic spaces.
A $k$-analytic space $X$ is said to be quasi-smooth\footnote{The terminology rig-smooth is also used by some other authors}  if $X$ is geometrically regular \cite[section 5]{Duc_flatness}.
When $k$ is algebraically closed, this is equivalent to say that for all  $x\in X$, the local ring  $\mathcal{O}_{X,x}$ is regular. 
When $k$ is algebraically closed and $X$ is a strictly $k$-analytic space, this is even equivalent to say that for all rigid point  $x\in X$, the local ring  $\mathcal{O}_{X,x}$ is regular (this follows for instance 
from \cite[2.3.4]{Berko90}). 

\begin{theo_nn} \textbf{\ref{theo_dim2}}
Let us assume that $k$ is algebraically closed.
Let $X$ be a good quasi-smooth strictly $k$-analytic space 
of dimension 2. Then a subset $S$ of $X$ is overconvergent subanalytic if and only if it is 
locally semianalytic.
\end{theo_nn}
Here, we say that $S$ is locally semianalytic if for every $x\in X$, there is an affinoid neighbourhood $V$ of $x$ such that $S\cap V$ is semianalytic in $V$.

\subsection*{Ideas behind the proofs}
We want to point out that the two equivalent characterizations of overconvergent subanalytic sets which 
were given in \cite{Sch_sub} and which we have recalled on page \pageref{theo_schoutens} are not very manageable. 
In particular it is hard to prove that some set is not overconvergent subanalytic 
using these characterizations, 
whereas we have much more tools to say that a subset is semianalytic or not. 
In order to overcome this difficulty, we have introduced a third characterization of overconvergent subanalytic sets which is more geometric. 
We remark that the quotient of two analytic functions $f$ and $g$ is not analytic any more, 
but becomes analytic if one blows up $(f,g)$. 
With this in mind, in order to describe a subset of $X$ defined by inequalities $\{ |f| \leq |g| \}$ with $f,g \in \ADS$ 
we can consider some 
finite sequences of blow ups  $\tilde{X} \to X$ and  project some semianalytic sets of $\tilde{X}$ outside the exceptional locus (with some 
extra condition for the overconvergence condition). We call such subsets overconvergent constructible 
(see \ref{defi_over}  for a precise definition).
The idea of looking at analytic functions above some blowup of $X$ 
had already appeared in \cite[2.3 (iv)]{LR_mod}. \par 
With this in mind we would like to restate more precisely the results of this paper. \par 
First, we prove theorem \ref{theo_eq} which asserts that if $X$ is an affinoid space, $S \subset X$ is overconvergent subanalytic if and only if it is 
overconvergent constructible, using at some point the compactness of the Berkovich space $X$.\par 
Then, according to the definition of an overconvergent constructible set, it is easy to prove that 
overconvergent subanalytic sets are local for the Berkovich topology (proposition \ref{local}). \par 
To justify our counterexample in proposition \ref{CE1}, we use the more geometric approach of overconvergent constructible sets which allows one to use results on semianalytic sets. 
Ultimately, our argument relies on the study of some Gauss point in an embedded curve in the polydisc, which strengthens our feeling that Berkovich spaces are well suited to study overconvergent subanalytic sets.  \par 
Finally, we want to mention one more benefit of overconvergent constructible sets. 
In the author's thesis it is proved (proposition 2.4.1) that if $k$ is algebraically closed, $S$ a locally closed overconvergent subanalytic set of the compact $k$-analytic space $X$, and if we consider a prime number 
$\ell \neq \cha(\tilde{k})$, then the \'{e}tale cohomology groups with compact support 
of the germ $(S,X)$ (see \cite[3.4,5.1]{Berko93}
\[ H_c^i( (S,X) , \Ql) \]
are finite dimensional $\Ql$-vector spaces. 
Here again the idea is that (thanks to the presentation of $S$ as an overconvergent constructible set)  
we can reduce to the case where $S$ is semianalytic, and in that case, the finiteness result is proved in \cite[proposition 2.2.3]{Mar_these} 
(which ultimately relies on a finiteness result for affinoid spaces proved by V. Berkovich). \par

\subsection*{Organisation of the paper} $ $ \\
\noindent In section \ref{section1}, we define \emph{constructible data} of $X$, in order to define overconvergent constructible subsets. 
Note that we do not assume that $k$ is algebraically closed contrary to \cite{Sch_sub}. 
In section \ref{section1.2}  we introduce overconvergent subanalytic subsets. 
In section \ref{section1.3} we carefully treat Weierstrass division, trying to be as general as possible 
(namely our results hold for an arbitrary ultrametric Banach algebra, 
and an arbitrary radius of convergence).
In section \ref{section1.4} we prove that overconvergent constructible and overconvergent subanalytic subsets are the same. 
The proof of this result which appears in \cite{Sch_sub}, is here simplified by the use of Berkovich spaces: in particular, the quite technical section 2 of \cite{Sch_sub} \emph{A combinatorial lemma} is replaced by a simple compactness argument (see theorem \ref{theo_eq}). 
In \ref{section1.5} we try to handle the following problem: 
how to pass from a definition that works only for $k$-affinoid spaces 
to a more local definition, with the hope that in the affinoid case 
the local and the global definitions would coincide. 
As we said earlier, trying to do this with the $G$-topology 
will not work. 
If however we 
do this with the Berkovich topology, the definitions will be compatible. 
In section \ref{nonstrict}, we explain how these results can be extended to 
$k$-affinoid spaces (as opposed to strictly $k$-affinoid spaces).
In addition, in that case, we can allow the field $k$ to be trivially valued.
\par 
In section \ref{section2}, we give some counter-examples to erroneous statements of 
\cite{Sch_sub}. Precisely, 
in \cite{Sch_sub} five classes of subsets were defined:  
globally strongly subanalytic, globally strongly $D$-semianalytic, 
strongly subanalytic, locally strongly subanalytic and strongly $D$-semianalytic subsets. 
The three last classes were defined from  
the first two ones by adding "$G$-local" at 
some point. 
In \cite{Sch_sub} it was claimed that these five classes agree.
We explain that this is not the case, namely from these
five classes, 
the first two  ones indeed agree, but not the last three, which are larger 
(see figure \ref{inclusions}, p. \pageref{inclusions}). 
The main idea is that if one replaces 
"$G$-locally" by "locally for the Berkovich topology", the results of \cite{Sch_sub}, 
for instance the theorem \cite[Quantifier elimination Theorem p. 270]{Sch_sub}, become true. 
Let us give one of the counter-examples that we study:
\begin{exem}
\label{exemple_base}
Let $X = \B^2$ be the the closed bidisc, $0<r<1$ with $r\in \val$, $f\in k\{r^{-1}x\}$ an analytic function whose radius of convergence is exactly $r$ and such that $\|f\| <1$. 
We then define 
\[S=\{(x,y) \in \B^2 \ \big| \ |x| < r \ \text{and} \ y=f(x) \}.\]
\end{exem}
Then (see proposition \ref{CE1}) $S$ is rigid-semianalytic, but not overconvergent subanalytic. 
The Berkovich approach is here helpful since to prove this, 
we use a point $\eta$ of the Berkovich bidisc which is not a rigid point, and some properties of its local ring $\mathcal{O}_{X,\eta}$. \par
Finally, in section \ref{section3} we correct 
the proof of \cite{Sch_plane} (which rested on 
the erroneous results of \cite{Sch_sub}, and \cite{Sch_unif}) and restrict the 
hypothesis of it. Namely, we prove that 
when $k$ is algebraically closed, and $X$ is a good quasi-smooth 
strictly $k$-analytic space of dimension $2$, then 
overconvergent subanalytic subsets are in fact locally semianalytic. 
Not only do we give a correct proof of this theorem, but moreover, this 
result is more general than the result of \cite{Sch_plane}, 
where $X$ was the bidisc and where it was assumed that  
the characteristic of $k$ was $0$.

\subsection*{Contribution of this text}
We want to stress the fact that section $1$ is highly inspired by the work of H. Schoutens.
In particular, the definition we give of a constructible datum, and the resulting definition of an overconvergent constructible subset, is a \emph{geometric} formulation 
of what is done in \cite{Sch_sub} concerning $D$-strongly semianalytic subsets.
In particular, the proof of theorem \ref{theo_eq} is very close to the proof  
of \cite[Th 5.2]{Sch_sub}. We have however decided to include a proof of 
theorem \ref{theo_eq} for three reasons.
First, the compactness argument that we use 
in theorem \ref{theo_eq} seems to us enlightening, and a way to see that 
Berkovich spaces are relevant in this context\footnote{However, 
it has to be noted that we could have 
written this proof in the context of adic spaces, and used a similar argument of quasi-compactness.}. 
Secondly, we have the feeling that replacing the strongly $D$-semianalytic subsets 
of \cite{Sch_sub} by our overconvergent constructible 
subsets is more geometric and 
gives a better understanding of the situation.
Finally, the mistakes in \cite{Sch_sub}, that we explain in section 2, have the following consequences:  some of the statements of \cite{Sch_sub} are false. 
For instance, 
\cite[Theorem 5.2]{Sch_sub} is false as we prove in section 2.
In this context it seemed to us relevant to 
write section 1.\par
The same remarks hold for section 3. A statement analogous to 
theorem \ref{theo_dim2} was claimed in 
\cite{Sch_plane}. However, in this article, it was assumed, and 
used in the proofs, that the five 
classes of subsets introduced in \cite{Sch_sub} were the same; 
but since we prove that this is not the case, the proofs of \cite{Sch_plane} are erroneous.\par
Finally, let us mention that another proof of theorem \ref{theo_eq} has also been given in  
\cite[4.4.10]{Clu_Lip}.
\subsection*{Acknowledgements}
I would like to express my deep gratitude to Jean-Fran\c{c}ois Dat and Antoine Ducros for numerous comments on this work. 
I also would like to thank the referee for numerous comments which have greatly improved this text.
The research leading to these results has received funding from 
the European Research Council under the European Community's Seventh 
Framework Programme (FP7/2007-2013) / ERC Grant Agreement n$^o$ 246903 "NMNAG,  from the
Labex CEMPI (ANR-11-LABX-0007-01) and from  SFB 1085 Higher invariants funded by the Deutsche Forschungsgesellschaft (DFG).

\addtocontents{toc}{\protect\setcounter{tocdepth}{2}}

\section{Overconvergent constructible subsets}
\label{section1}
With a few exceptions that will be specified, $k$ will be a non-trivially valued non-archimedean field,
$\A$ will be a strictly $k$-affinoid algebra, and  
$X$ the strictly $k$-affinoid space \affin{A}.

\subsection{Constructible data}
\label{section1.1}

\begin{defi}
Let $X$ be a $k$-affinoid space whose $k$-affinoid algebra is $\mathcal{A}$. A subset 
$S$ of $X$ is called \emph{ semianalytic} if it is a finite boolean combination of sets of the form 
$\{ x\in X \ \big| \ |f(x)| \leq |g(x)| \}$ where $f$ and $g \in \mathcal{A}$ (by finite boolean combination, 
we mean finitely many use of the set-theoretical operators $\cap$, $\cup$ and $^c$). 
A subset of the form 
$\{x \in X \ \big| \ |f_i(x)| \Diamond_i |g_i(x)| \ \forall i=1\ldots n \}$ with $f_i$ and $g_i \in \mathcal{A}$, 
and $\Diamond_i \in \{\leq, <\}$ will be called \emph{basic semianalytic}. 
\end{defi}
\begin{rem}
\label{rembasicsa}
With a repeated use of the rule 
$(A\cup B)\cap C = (A\cap C) \cup (B\cap C)$ one can show that 
any  semianalytic subset 
of $X$ is a finite union of basic semianalytic subsets.
\end{rem}

\begin{defi}
\label{def_dce}
Let $(X,S)$ be a \emph{k-germ} in the sense of \cite[3.4]{Berko93}; 
this just means that $S$ is a subset of $X$. 
An elementary constructible datum of $(X,S)$, is the following datum.
Let $f,g  \in \mathcal{A}$. 
Let also $r$ and $s$ be two real numbers such that  
$r>s>0$ and $s\in \sqrt{|k^*|}$. Let 
\[Y= \mathcal{M}( \mathcal{A} \{r^{-1}t\}/ (f-tg) ) \xrightarrow{\varphi} \affin{A} = X\]
and let $R \subseteq Y$ be a \GSA \ of $Y$. 
Let us set 
\[T:=\varphi^{-1}(S) \cap \{y\in R \ \big| \ g(y)\neq 0 \ \text{and} \ |f(y)|\leq s|g(y)| \}.\]
Then $(Y,T) \xrightarrow{\varphi } (X,S) $ is an \ECD.
If $ \psi : (Y',T') \simeq (Y,T)$ is an isomorphism of 
\emph{k-germs}, and \dce \ is an \ECD, if we set 
$\varphi'=\varphi \circ \psi$, then we will also say that 
$(Y',T') \xrightarrow{\varphi'} (X,S)$ is an \ECD.
\end{defi}

\begin{rem}
\label{rem1}
If \dce \ is an \ECD, then $\varphi(T) \subset S$, 
and $\varphi$ realizes a homeomorphism between $T$ and its image 
$\varphi(T)$.
Moreover 
\[ \{y \in Y \ \big| \  |f(y)|\leq s|g(y)|\neq 0 \} \]
is an analytic domain of $Y$, and can be identified through $\varphi$ with the analytic domain of $X$
\[ \{x \in X \ \big| \  |f(x)|\leq s|g(x)|\neq 0 \} .\]
\end{rem}

\begin{defi}
\label{defidc}
Let $(X,S)$ be a $k$-germ. A \emph{constructible datum} is a sequence
\[ (Y,T)=(X_n,S_n) \xrightarrow{\varphi_n} (X_{n-1} ,S_{n-1}) \rightarrow \cdots \rightarrow (X_1,S_1) \xrightarrow{\varphi_1} (X_0,S_0) = (X,S)\]
where for $i=1\ldots n$, 
$(X_i,S_i) \xrightarrow{\varphi_i} (X_{i-1} , S_{i-1} ) $ is an \ECD.
Let  $\varphi=\varphi_1\circ\ldots \circ \varphi_n $. 
Then we will denote 
this constructible datum by 
\[ \dc. \]
We will say that the complexity of $\varphi$ is $n$.
\end{defi}
In the particular case $S=X$, i.e. $(X,S) = (X,X)$, we will denote the constructible datum by: 
\[(Y,T) \stackrel{\varphi}{\dashrightarrow}  X ,\]
and we will call it a constructible datum of $X$.
This is actually the case that will mainly interest us, but  for technical reasons we have chosen to use \emph{k-germs}.
\begin{rem}
\label{remcdsa}
If $(Y,T) \stackrel{\varphi}{\dashrightarrow} X$ is a constructible datum, it follows easily from 
the above definitions that $T$ is a semianalytic subset of $Y$.
\end{rem}
Remark \ref{rem1} implies that if \dc \ is a constructible datum, 
$\varphi_{|T} : T \xrightarrow{\varphi_{|T}} S$ induces a homeomorphism between 
$T$ and $\varphi(T)$.
It is also clear that if 
$(Z,U) \stackrel{\psi}{\dashrightarrow} (Y,T)$
is a constructible datum \ and \dc \ is another one, then
$(Z,U) \stackrel{\varphi \circ \psi}{\dashrightarrow} (X,S)$ 
is also a constructible datum. \par
We want to point out that in the definition of a constructible datum, $n$ cannot be recovered from $\varphi$ alone.

\begin{defi}
Let \recouv, $i=1\ldots m$ be $m$ constructible data of 
the $k$-germ $(X,S)$. We will say it forms a
\emph{constructible covering} of $(X,S)$ if 
$\bigcup\limits_{i=1}^m \varphi_i(S_i) = S$.
\end{defi}

\begin{defi} 
\label{defi_over}
Let $X$ be a $k$-affinoid space. 
A subset $C$ of $X$ is said to be an \emph{overconvergent constructible subset} 
of $X$ if there exist 
$m$ constructible data $(X_i,S_i) \stackrel{\varphi_i}{\dashrightarrow} X$ 
 for $i=1\ldots m$ such that 
$\bigcup\limits_{i=1}^m \varphi_i(S_i) = C$. 
\end{defi}

\begin{rem}
Using the notation of definition \ref{def_dce}, when \dce \ is an elementary constructible datum, 
with $Y= \mathcal{M}( \mathcal{A} \{r^{-1}t\}/ (f-tg) )$,
then $T$ (and hence $\varphi(T)$) are defined with the function 
$t$ which mimics the function 
$\dfrac{f}{g}$, when it has a sense and its norm is $\leq s$. In addition the condition 
$r>s$ is here to make sure that the new functions of $\mathcal{B}$ 
are overconvergent in $t=\dfrac{f}{g}$, that we see as a function on 
the analytic domain  $\{ x\in X \ \big| \ |f(x)| \leq s |g(x)|\neq 0 \}$.
\end{rem}

The following three results are formal consequences of the previous definitions.

\begin{lemme}
\label{lemme1prod}
If \dce \ is an \ECD \ and 
$(Z,U) \xrightarrow{\psi } (X,S)$ 
is a morphism of $k$-germs, let us consider the Cartesian 
product of $k$-germs:
$$\xymatrix{
(Y,T)  \ar[r]^{\varphi} & (X,S)  \\
(Y,T)\times_{(X,S)} (Z,U)  \ar[u]^{\psi'} \ar[r]^-{\varphi'} & (Z,U) \ar[u]^{\psi}
}$$
Then,
$ (Y,T)\times_{(X,S)} (Z,U) \xrightarrow{\varphi'} (Z,U)$ 
is an \ECD. Moreover if we set  
$$(Y,T)\times_{(X,S)} (Z,U) = (Y',T')$$
then, $ ( \varphi \circ \psi')(T') = \varphi(T)\cap \psi(U)$. 
\end{lemme}

\begin{cor}
\label{cor1prod}
Let \dc \ be a constructible datum 
\[ (Y,T)=(X_n,S_n)\xrightarrow{\varphi_n} \cdots \xrightarrow{\varphi_1} (X_0,S_0) =(X,S)\]
and let 
$(X',S')  \xrightarrow{\psi} (X,S)$  be a morphism of $k$-germs. 
Let us consider the Cartesian product :
$$\xymatrix{
(Y,T)  \ar@{-->}[r]^{\varphi} & (X,S)\\
(Y',T') \ar[u]^{\psi'} \ar@{-->}[r]^{\varphi'} & (X',S')  \ar[u]^{\psi} 
}$$
Then 
$(Y',T') \stackrel{\varphi'}{\dashrightarrow} (X',S') $ 
is a constructible datum \ and  $(  \psi \circ \varphi')(T') = \varphi(T) \cap \psi(S')$.

\end{cor}

\begin{cor}
\label{lemmeintersection}
Let  $(X_1,T_1)  \stackrel{\varphi}{\dashrightarrow} (X,S)$  and 
$(X_2, T_2) \stackrel{\psi}{\dashrightarrow} (X ,S) $  
be two constructible data (with the same target). 
Let us consider the fibered product 
\[\xymatrix{
(X_1,T_1) \ar@{-->}[r]^{\varphi} & (X,S)\\
(Z,U) \ar@{-->}[u]^{\psi'} \ar@{-->}[r]^{\varphi'} & (X_2,T_2) \ar@{-->}[u]^{\psi} 
}\]
Then
$(Z,U) \stackrel{\psi'}{\dashrightarrow} (X_1,T_1)$
and 
$(Z,U)\stackrel{\varphi'}{\dashrightarrow} (Y_2,T_2)$ 
are constructible data. Moreover
$(\varphi \circ \psi') (U) = (\psi \circ \varphi') (U) = \varphi(T_1)\cap \psi(T_2)$.
\end{cor}

\begin{proof}
Lemma \ref{lemme1prod} is a direct consequence of definition \ref{def_dce}. 
Corollary \ref{cor1prod} is then proved by induction on the complexity of $\varphi$ using lemma \ref{lemme1prod}. Similarly, corollary \ref{lemmeintersection} is proved by induction 
on the complexity of $\psi$ using corollary \ref{cor1prod}.
\end{proof}

\begin{prop}
\label{prop_gen} 
\begin{enumerate}
\item If $T \subseteq X$ is a  semianalytic subset of $X$ then 
$T$ is an overconvergent constructible subset of $X$.
\item Let $C \subseteq T$ be an overconvergent constructible subset of 
$Y$ and let
$(Y,T) \stackrel{\varphi}{\dashrightarrow} X$ be a constructible datum. Then 
$\varphi(C)$ is on overconvergent constructible subset of $X$.
\item The class of overconvergent constructible subsets of $X$ is stable under 
finite boolean combinations.
\end{enumerate}
\end{prop}

\begin{proof}
\noindent 
\begin{enumerate}
\item Consider the elementary constructible datum 
$(X,T) \xrightarrow{id} X$.
\item By definition, there exist some constructible data 
$(Y_i,T_i) \stackrel{\varphi_i}{\dashrightarrow} Y$, for $i=1\ldots m$, such that 
$\displaystyle C= \bigcup\limits_{i=1}^m \varphi_i(T_i)$. 
Now if we define 
$\psi_i := \varphi\circ \varphi_i$, then 
$(Y_i,T_i) \stackrel{\psi_i}{\dashrightarrow} X$ are $m$ constructible data, 
and $\displaystyle \varphi(C) = \varphi( \bigcup_{i=1}^m \varphi_i(Y_i)) = \bigcup_{i=1}^m \psi_i(T_i)$, 
so it is an overconvergent constructible subset of $(X,S)$.
\item Stability under finite union is a direct consequence of the definition \ref{defi_over}, as for intersection, it is a consequence of corollary \ref{lemmeintersection}. 
And if $C \subseteq X$ is an overconvergent constructible subset of $X$, let us show that 
$X \setminus C$ is also overconvergent constructible. 
By definition, 
$\displaystyle C= \bigcup_{i=1}^m \varphi_i(S_i)$ where
$(X_i,S_i) \stackrel{\varphi_i}{\dashrightarrow} X$ are some constructible data. 
We do it by induction 
on $c$, the maximum of the complexity of the $\varphi_i$'s. \par
If $c=0$, then $C$ is a semianalytic subset of $X$ so 
$X \setminus C$ is semianalytic, hence overconvergent constructible. \par
If $c>0$ and we assume the result for $c'<c$, then 
\[X\setminus C = X\setminus (\bigcup_{i=1}^m \varphi_i(S_i)) =
\bigcap_{i=1}^m (X \setminus \varphi_i(S_i)) \]
so we can assume that $m=1$, that is to say, we can assume that 
$C = \varphi(T)$ where $(Y,T) \stackrel{\varphi}{\dashrightarrow} X$ 
is a constructible datum 
of complexity $c$. 
Then 
\[\varphi = \psi \circ \varphi' : 
(Y,T) \stackrel{\varphi'}{\dashrightarrow} (Y',T') \xrightarrow{\psi} X\] 
where the complexity of $\varphi'$ is $c-1$ and $\psi$ is an elementary constructible datum. 
Now 
\[ X \setminus \varphi(T) = \psi(T' \setminus \varphi'(T)) \cup (X \setminus \psi(T') ) \] because $\varphi'_{|T}$ and $\psi_{|T'}$ are injective maps. 
By induction hypothesis, 
\[T' \setminus \varphi'(T) = T' \cap (Y' \setminus \varphi'(T) )\] 
is an overconvergent constructible subset of 
$Y'$, thus according to $(1)$, so is $\psi(T' \setminus \varphi'(T) )$.\par
Finally, if the elementary constructible datum $\psi$ is associated with $f,g,r$ and $s$, 
by definition, 
\[T' =\{ y\in R \ \big| \ |f(y)|\leq s |g(y)|\neq 0 \} \] 
for some  semianalytic subset $R$ of $Y'$. 
And if we define 
\[\tilde{T} = \{ y\in Y' \setminus R  \ \big| \ |f(y)|\leq s |g(y)|\neq 0 \}, \] 
then 
\[X \setminus \psi(T') = \psi(\tilde{T}) \cup \{ y \in X \ \big| \ |f(y)| > s |g(y)| \} \cup 
\{ y\in X \ \big| \ g(y)=0 \}. \]
Thus, it is also overconvergent constructible in $X$.
\end{enumerate}
\end{proof}
 
Let $x\in X$, and $U$ be an affinoid neighbourhood of $x$. 
Shrinking $U$ if necessary, 
we can assume \cite[2.5.15]{Berko90} that $U$ is a rational domain of the form
$X(\underline{r}^{-1} \frac{f}{g} ) = 
\{p \in X \ \big| \  |f_i(x)| \leq r_i|g(x)| \}$
such that 
$X\left( \left( \frac{\underline{r}}{2} \right)^{-1} \frac{f}{g} \right)$ still contains $x$.
For each $i$, we pick a real number $s_i$ such that 
$\frac{r_i}{2} < s_i < r_i$ and $s_i \in \val$.  
For each $i$, we consider the \ECD \ 
$(X_i,S_i) \xrightarrow{\varphi_i} X$ defined by 
$X_i = \mathcal{A}\{r_i^{-1}t_i\}/(f_i-t_ig)$,
and $S_i = \{p\in X_i \ \big| \ |f_i(p)|\leq s_i |g(p)| \ \text{and} \ g(p) \neq 0 \}$. One checks that 
$\varphi_i(S_i)$ is a neighbourhood of $x$.
Now if we take the fibered product of all these elementary constructible data, 
we obtain (using corollary \ref{lemmeintersection}) the following constructible datum:
\[ \left( X\left(\underline{r}^{-1} \frac{f}{g} \right),X \left( \underline{s}^{-1} \frac{f}{g} \right) \right) \stackrel{\varphi}{\dashrightarrow} X\]
Here $\varphi$ just corresponds to the embedding of the affinoid domain 
$X(\underline{r}^{-1} \frac{f}{g} )$. 
Moreover $ \varphi \left( X( \underline{s}^{-1} \frac{f}{g} ) \right)$, 
that we might identify with  $X\left( \underline{s}^{-1} \frac{f}{g} \right)$, 
is a neighbourhood of $x$. We can sum up this in the following lemma:

\begin{lemme}
\label{rem_local}
Let $X$ be a strictly $k$-affinoid space.
 Let $x\in X$ and $U$ be an affinoid neighbourhood of $x$.
Then there exists a constructible datum 
$(Y,T) \stackrel{\varphi}{\dashrightarrow} X$ 
such that $T$ is an affinoid domain of $Y$,  
 $\varphi$ is the embedding of an affinoid domain
$Y \to X$ such that $Y$ is in fact an affinoid subdomain of $U$, 
and $\varphi(T)$ is an affinoid neighbourhood of $x$. 
\end{lemme}

\begin{cor}
\label{corlocalover}
Let $X$ be a strictly $k$-affinoid space. Being overconvergent constructible in $X$ is a local property. 
\end{cor}
\begin{proof}
First, if $S\subset X$ is overconvergent constructible, and $U$ is an affinoid domain of $X$, then $S\cap U$ is overconvergent constructible. \par 
On the other hand, let us assume that locally for the Berkovich topology, $S$ is overconvergent constructible, that is to say, let us assume that for all $x\in X$ there exists an affinoid neighbourhood 
$U$ of $x$ such that 
$S\cap U$ is overconvergent constructible.
Then according to lemma \ref{rem_local}, there exists 
a constructible datum $(Y,T) \stackrel{\varphi}{\dashrightarrow} X$  such that 
$Y \xrightarrow[]{\varphi} X$ is the embedding of an affinoid domain, $Y \subset U$, and 
$T$ is an affinoid neighbourhood of $x$. Then, since $T \subset U$, 
$\varphi^{-1}(S)\cap T$ is overconvergent constructible in $T$, and then 
$\varphi(T) \cap S$ is overconvergent constructible in $X$ (see proposition \ref{prop_gen} (2)). 
But since $\varphi(T)$ is an affinoid neighbourhood of $x$, by compactness of $X$ we conclude that $S$ is overconvergent constructible.
\end{proof}
\subsection{Overconvergent subanalytic subsets}
\label{section1.2}
We will denote by 
$\B$ (resp. $\B_r$ for $r>0$) the closed disc of radius $1$ (resp. $r$), 
and if $n$ is an integer, $\B^n$ and $\B_r^n$ will denote the corresponding closed polydiscs.\par
More generally, if $\underline{r} = (r_1,\ldots , r_n) \in {\mathbb{R}^*_+}^n$ is a polyradius, 
we will denote by
\[ \Bur = \mathcal{M}(k\{\underline{r}^{-1}T\}) =
\mathcal{M}(k\{ r_1^{-1}T_1, \ldots,r_n^{-1}T_n\})\]
 the polydisc of radius $\underline{r}$, and 
$\overset{\circ}{\B}(\underline{r})$ the corresponding open polydisc. When the number 
$n$ will be clear from the context, we will write $\underline{1}$ for 
$(1,\ldots, 1) \in \mathbb{R}^n$, and $\underline{0}$ or $0$ for $(0,\ldots, 0)\in \mathbb{R}^n$.
Finally, $\underline{\rho} > \underline{r}$ will mean that 
$\rho_i > r_i$ for $i=1\ldots n$. 

\begin{defi}
\label{defi_suba}
Let $X$ be a strictly $k$-affinoid space.
A subset $S\subset X$ is said to be an \OS \ of $X$ if 
there exist $n\in \mathbb{N}$, 
$r>1$, and 
$T\subseteq X \times \B_r^n$ a \GSA \ such that 
$S = \pi (T \cap ( X \times \E ) )$ where
$\pi : X \times \Br^n \to X$ is the natural projection. 
\end{defi}

\begin{lemme}
\label{lemme_inverse}
Let $f:Y \to X$ be a morphism of strictly $k$-affinoid spaces and 
$S$ an overconvergent subanalytic subset of $X$.
Then $f^{-1}(S)$ is an overconvergent subanalytic subset of $Y$. 
In particular, if $V$ is a strictly affinoid domain of $X$ and $S$ an overconvergent 
subanalytic subset of $X$, then $S\cap V$ is an 
overconvergent subanalytic subset of $V$.
\end{lemme}

\begin{proof}
Let $r>1$ and $T\subseteq X \times \Br^n$ be a  semianalytic 
subset such that 
$S= \pi( T\cap (X\times \B^n))$.
Let us consider the following Cartesian diagram: 
\begin{equation}
\label{cart_sq}
\xymatrix{
Y \times \Br^n \ar[r]^{f'} \ar[d]^{\pi'} & X\times \Br^n \ar[d]^{\pi} \\
Y \ar[r]^f & X
}
\end{equation}
Then 
$f^{-1}(S) = f^{-1} (\pi (T \cap (X \times \B^n))) =
\pi' (f'^{-1} (T\cap (X\times \B^n)))$. 
The last equality holds because \eqref{cart_sq} is a Cartesian diagram.
Now 
$\pi' (f'^{-1} (T\cap (X\times \B^n))) = 
\pi' (f'^{-1} (T)\cap (Y\times \B^n))
= \pi'^{-1} (T' \cap (Y\times \B^n))$ where 
$T' =f'^{-1}(T)$ is a  semianalytic subset of $Y\times \Br^n$.
Hence $f^{-1}(S) = \pi' (T' \cap ( Y\times \B^n))$ is an overconvergent subanalytic subset of $Y$.
\end{proof}

\begin{lemme} 
\label{immersion}
Let $X$ and $Y$ be strictly $k$-affinoid spaces, and 
let $\varphi : X \to Y$ be a closed immersion. 
\begin{enumerate}
\item 
If $S$ is a  semianalytic subset of $X$, then $\varphi(S)$ is a 
 semianalytic subset of $Y$.
\item Let $S$ be an overconvergent subanalytic subset of $X$, then 
$\varphi(S)$ is an overconvergent subanalytic subset of $Y$.
\end{enumerate}
\end{lemme}

\begin{proof}
\noindent
\begin{enumerate}
\item 
Write 
$Y=\affin{A}$ and 
$X = \mathcal{M} \left(\mathcal{A}/ \mathcal{I}\right)$ where 
$\mathcal{I} = (a_1, \ldots , a_m)$ is an ideal of $\mathcal{A}$.
Then, if $S= \{ x\in X \ \big| \ |f_i(x)| \Diamond_i |g_i(x)|, \ i=1\ldots n \}$ with $f_i,g_i \in \mathcal{A}/ \mathcal{I}$, we can find functions $F_i,G_i \in \mathcal{A}$ 
such that 
$\overline{F_i} =f_i$ and $\overline{G_i} = g_i$. In that case one checks that, 
\[\varphi (S) = \{ y\in Y \ \big| \ |F_i(y)| \Diamond_i |G_i(y)| , \ i=1\ldots n \} \cap 
\{y \in Y \ \big| \ a_j(y)=0, \ j=1\ldots m\},\] 
which is indeed  semianalytic.
\item 
By definition there exists a semianalytic subset 
$T \subseteq X\times \Br^n $ for some 
$r>1$ such that 
$S= \pi(T \cap (X \times \B^n ))$. 
We then consider the following cartesian diagram:
\[
\xymatrix{
X\times \Br^n \ar[r]^{\varphi'} \ar[d]^{\pi'} & Y\times \Br^n \ar[d]^{\pi} \\
X \ar[r]^{\varphi} & Y
} \]
But $\varphi'$ is also a closed immersion, so according to $(1)$,
 $T' = \varphi'(T)$ is a  semianalytic subset of 
 $Y\times \Br^n$. Then one checks that 
 \[\pi \left(T'\cap ( Y\times \B^n ) \right) 
 =\pi \left( \varphi'(T) \cap (Y \times \B^n) \right) 
 =\pi \left( \varphi' ( T \cap ( X \times \B^n) \right) =
 \varphi( \pi' ( T \cap (X \times \B^n ) ) = \varphi(S). \]
\end{enumerate} 
\end{proof}

\begin{lemme}
\label{rem_rayon}
Let us assume that $\underline{s} \in \val^n$. 
Then, $k\{\underline{s}^{-1}T \}$ is a strictly $k$-affinoid algebra 
(see \cite[2.1.1]{Berko90} and \cite[6.1.5.4]{BGR}). For the same reasons, if 
$\underline{r} > \underline{s}$, and 
$S \subseteq X\times \Bur$ is a  semianalytic subset, then 
$\pi \left( S\cap (X \times \Bus )  \right)$ is an overconvergent 
subanalytic subset of $X$.\end{lemme}
\begin{proof} 
Indeed let $\underline{s} \in \val^n$ and $\underline{r} \in \mathbb{R}^n$ such that 
$\underline{s} < \underline{r}$, and $S \subseteq X \times \Bur$ a  
semianalytic subset of $X \times \Bur$. 
Let us show that $\pi ( S \cap (X \times \Bus )$ is overconvergent subanalytic in the sense 
of definition \ref{defi_suba}. 
To avoid complications, we assume that 
$n=1$ (but the proof is similar for an arbitrary $n$). 
Let then $s\in \val$ and $r>s$. 
Up to multiplication by some $\mu \in k^{\times}$ small enough, 
we can assume that $s \leq 1$. Since $s \in \val$, there exist
$\lambda \in k^{\times}$ and $m\in \mathbb{N}$ such that $s^m = |\lambda|$. 
Then in 
\[\B_{ (r, \left( \frac{r}{s}\right)^m    )} = 
\mathcal{M} ( k\{r^{-1}y , \left( \left( \frac{r}{s} \right)^m \right)^{-1} t \} ) \]
let us consider the Zariski-closed subset defined by 
$y^m= \lambda t$, i.e. $V(y^m-\lambda t)$.
Then, the map: 
\[
\begin{array}{rcl}
\Br & \to & \B_{(r, \left( \frac{r}{s} \right)^m )} \\
x    & \mapsto & (x,\frac{x^m}{\lambda} )
\end{array}
\]
identifies $\Br$ with the Zariski closed subset 
$V(y^m- \lambda t )$ and moreover, since $s\leq 1$ 
\[
\begin{array}{rcl}
\Bs & \to & \B^2 \\
x & \mapsto & (x,\frac{x^m}{\lambda})
\end{array}
\]
identifies $\Bs$ with the Zariski-closed subset of $\B^2$, 
$V(y^m - \lambda t)$. Taking the fibre product with $X$ we then obtain:
\[
\xymatrix{
X\times \Br \ar[r]^{\simeq} & V(y^m-\lambda t) \ar@{^{(}->}[r]^{\alpha} & 
X\times \B_{(r,\left( \frac{r}{s} \right)^m)} \\
X \times \Bs \ar@{^{(}->}[u] \ar[r]^{\simeq} \ar[rd]_{\pi} & V(y^m - \lambda t) \ar@{^{(}->}[u] \ar@{^{(}->}[r]^{\beta} & X\times \B^2 \ar@{^{(}->}[u] \ar[ld]^{\pi}\\
 & X &
}
\]
Hence if $S \subseteq X \times \Br$ is  semianalytic, $S' := \alpha (S)$ is also  semianalytic in 
$X \times \B_{(r, ( \frac{r}{s} )^m )}$ 
and $\alpha (S) \cap ( X \times \B^2 ) = \beta ( S \cap ( X \times \Bs ) )$.
So $\pi ( S \cap ( X \times \Bs ) ) = \pi (S' \cap ( X \times \B^2 )$ is well 
 overconvergent subanalytic in the sense of definition \ref{defi_suba}.  
\end{proof}

\subsection{Weierstrass preparation}
\label{section1.3}
In this section, $A$ will be an ultrametric complete normed ring 
i.e. satisfies the inequality 
$\|ab\| \leq \|a\|\|b\|$ and $\|a+b\| \leq \max (\|a\|, \|b\|)$ \cite[1.2.1.1]{BGR}. \par
If $r>0$, on $A\{r^{-1}T\}$ we will consider the following norm: if $ g=\sum\limits_{n\in \mathbb{N} } a_nT^n \in A\{r^{-1}X\}$ then $\|g\| = \max\limits_{n\geq 0} \|a_n\|r^n$.\par 
If $m\in \mathbb{N}$, we will denote by $A_m[T]$ the subset of $A[T]$ made 
of the polynomials of degree less or equal to $m$.

\begin{defi}
An element $u\in A$ is a multiplicative unit  
if $u$ is invertible and for all
$a\in A$, $\|ua\| = \|u\|\|a\|$.
\end{defi}
Note that if $u$ and $v$ are multiplicative units, so is $uv$.

\begin{lemme}
An element $u\in A$ 
is a multiplicative unit if and only if 
$u\in A^*$ and
$\|u^{-1}\| = \|u\|^{-1}$.
\end{lemme}
\begin{proof}
If $u$ is a multiplicative unit, 
$1=\|uu^{-1}\| = \|u\| \|u^{-1}\|$, so $\|u^{-1} \| = \|u\|^{-1}$. \par
Conversely let us assume that $u$ is invertible and that
$\|u^{-1}\| = \|u\|^{-1}$. Let then $a\in A$. The following holds:
\[\|a\| = \|u^{-1} (ua) \| \leq \|u^{-1}\| \|ua\| = \|u\|^{-1} \|ua\|.\]
So $\|ua\| \geq \|u\| \|a\|$. Since in any case the reverse inequality 
$\|ua\| \leq \|u\| \|a\| $ holds, we conclude that 
$\|ua\| = \|u\| \|a\|$.
\end{proof}

\begin{rem}
\label{normemultiplicatif} 
As a consequence, if $u\in A$ and 
$\|u\|<1$, then $(1+u)$ is a multiplicative unit 
because  
\[\|1+u\| = 1 = \|\sum\limits_{n\geq 0 } (-u)^n\| = \|(1+u)^{-1} \|\]
Let us note also that if $u$ is a multiplicative unit, 
for all $x\in \mathcal{M}(A)$, $|u(x)| =\|u\|$. 
Indeed, the definition of \affin{A}  implies that
\begin{equation} 
\label{strictineg}
|u(x)|\leq \|u\|,
\end{equation}
hence 
$1 = |u(x)||u^{-1}(x)| \leq \|u\| \|u^{-1}\| = 1$.
So the inequality \eqref{strictineg} could not be strict, thus 
$|u(x)|= \|u\|$.
\end{rem}

\begin{rem}
If  $\varphi \ : \ A \to B$ 
is a contractive morphism of normed rings 
(i.e.  $\|\varphi(a)\| \leq \|a\| $ for all $a$ in $A$), 
then $\varphi$ sends multiplicative units to 
multiplicative units.
Indeed we have the sequence of inequalities:
\[1 = \| \varphi(u) \varphi(u)^{-1}\| \leq 
\| \varphi(u)\| \| \varphi(u^{-1})\| \leq 
\|u\| \|u^{-1} \| =1.\]
So there were only equalities and $\varphi(u)$ is a 
multiplicative unit because $\| \varphi(u)\| = \|u\|$, and  
$\| \varphi(u)^{-1}\| = \|u^{-1}\| = \|u\|^{-1} = \| \varphi(u) \|^{-1}$.\par
This remark will apply in the following context:
when $\mathcal{A}$ 
is a strictly $k$-affinoid algebra and we look at a morphism 
$\varphi : \mathcal{A} \to \mathcal{B} = \mathcal{A}\{r^{-1}T \} / I$ 
with $I$ any ideal, and $\mathcal{B}$ is 
equipped with the quotient norm inherited from $\mathcal{A}\{r^{-1}T\}$. 
In this situation, $\varphi$ is contractive. 
This is the case when we consider $\varphi$ 
the morphism of a constructible datum 
$(Y,S) \stackrel{\varphi}{\dashrightarrow} X$.\par
Note that if $\varphi$ is not contractive, multiplicative units are not necessarily preserved. 
For instance consider $\mathcal{A} = k\{t\}$ and $\mathcal{B} = k\{2^{-1}x,y \}/(y-x^2)$ that we equip with the residue norm. 
These $k$-affinoid algebras are isomorphic through 
$\varphi : t \mapsto x$, and if we choose $\pi \in k$ such that 
$\frac{1}{2} < |\pi|<1$, then $u:=1+\pi t$ is a multiplicative unit of 
$\mathcal{A}$, but not $\varphi(u)$. \par 
Note however that if the field $k$ is stable (for instance in our situation, where $k$ is a non-archimedean complete field, $k$ is stable if $char(\tilde{k} ) =0$, or if it is algebraically closed, or a discrete valuation field \cite[3.6.2]{BGR}),
for a suitable choice of norm, any morphism of reduced affinoid algebras is contractive. 
Indeed, if $k$ is stable, 
and $\mathcal{A}$ is a reduced affinoid algebra, 
then it is a distinguished affinoid algebra \cite[6.4.3]{BGR}, i.e. the supremum seminorm is a residue norm on $\mathcal{A}$. 
If $\mathcal{B}$ is reduced, for the same reason,  the supremum seminorm is an admissible norm on it.
So if we equip $\mathcal{A}$ and $\mathcal{B}$ with the supremum norm,  any morphism of affinoid algebras 
$\varphi : \mathcal{A} \to \mathcal{B}$ is contractive.
\end{rem}

\begin{defi}
Let $r>0$ be a real number and $s\in \N$.
An element $g= \sum\limits_{n\geq 0} g_nT^n$ of 
$A\{r^{-1}T \}$ 
is called $T$-distinguished of order $s$ if
$g_s$ is a multiplicative unit,  
$\|g_s\|r^s=\|g\|$ 
and for all $n>s$, 
$\|g_n\|r^n < \|g_s\|r^s$. 
Note that in that case, $g$ is necessarily a non zero element since $g_s\neq 0$.
\end{defi}

\begin{rem}
\label{rem_distingue}
We can extend the previous remark saying that if
$\varphi : A \to B$ is a contractive morphism and 
$g = \sum\limits_{n \in \mathbb{N} } g_n T^n \in A \{r^{-1} T \}$ 
is $T$-distinguished of order $s$, 
then $\varphi(g) = \sum\limits_{n\in \mathbb{N} }\varphi(g_n) T^n \in   B\{r^{-1}T \}$ 
and it is an $T$-distinguished element of 
$B\{r^{-1}T\}$ of order $s$.
This applies in particular when $\varphi$ is the 
morphism of a constructible datum 
$(Y,S) \stackrel{\varphi}{\dashrightarrow} X$.
\end{rem}

\begin{lemme}
\label{multiplicatif}
Let $g=\sum\limits_{m \in \mathbb{N} } g_m T^m\in A\{r^{-1}T\}$ 
be $T$-distinguished of order $s$. 
\begin{enumerate}
\item 
Then for all  
$q=\sum\limits_{k\in \mathbb{N} } q_kT^k\in \Ar $,  
$\|gq\| = \|g\|\|q\|$.
\item 
Let us set 
$gq= \sum\limits_{l \in \mathbb{N} } c_lT^l$, and let us assume that $q\neq 0$. 
Let us denote by $k_0$ the greatest rank such that
$\|q_{k_0}\|r^{k_0} = \|q\|$. Then 
$\|gq\| = \|c_{s+k_0}\|r^{s+k_0}$ and $\|c_{s+k_0} \| = \|g_s \| \|q_{k_0} \|$. 
\end{enumerate}
\end{lemme}

\begin{proof}
First, without any hypothesis, it is true that
\begin{equation}
\label{inegtriv}
\|gq\| \leq \|g\|\|q\|.
\end{equation}\par
Conversely, by definition,  
\begin{equation}
\label{egW}
c_{s+k_0} = \sum\limits_{m+k=s+k_0} g_mq_k .
\end{equation} 
Let then $m$ and $k$ be two integers such that $m+k = s+k_0$. \par
If $k>k_0$, by definition of $k_0$, $\|q_k\|r^k < \|q_{k_0}\|r^{k_0}$.
So, using that $g_s$ is a multiplicative unit, we obtain:
\[\|g_mq_k\|r^{s+k_0}=\|g_mq_k\|r^{m+k} 
\leq \|g_m\|r^m \|q_k\|r^k 
< \|g_s\|r^s\| q_{k_0}\|r^{k_0}
=\|g_sq_{k_0}\|r^{s+k_0}. \] 
Thus,  
\begin{equation} 
\label{inegW1}
\|g_mq_k\| < \|g_sq_{k_0}\|.
\end{equation}\par
If $k<k_0$, then $m>s$, and since 
$\|g_m\|r^m< \|g_s\|r^s$ (because $g$ is $T$-distinguished of 
order $s$), we obtain with the same reasoning, that
\begin{equation}
\label{inegW2}
\|g_mq_k\| < \|g_sq_{k_0}\|.
\end{equation}\par
Thus, \eqref{egW}--\eqref{inegW2} and the ultrametric inequality imply that 
$\|c_{s+k_0}\| = \|g_sq_{k_0}\|$. 
And since $g_s$ is a multiplicative unit, $\|g_sq_{k_0}\|=\|g_s\|\|q_{k_0}\|$.\par 
Finally we obtain that  
$\|gq\| \geq \|g_s\|r^s\|q_{k_0}\|r^{k_0} = \|q\|\|g\|$, 
which with \eqref{inegtriv} ends the proof.
\end{proof}

\begin{prop}\textbf{Weierstrass Division.}
Let $g \in \Ar $ be $T$-distinguished of order $s$. 
If $f=\sum\limits_{n\in \mathbb{N} } f_nT^n \in \Ar$ there exists an unique couple
$(q,R) \in \Ar \times A_{s-1}[T]$ such that 
\begin{equation}
\label{existencedec}
f=gq+R. 
\end{equation}
Moreover 
\begin{equation}
\label{borne}
\|f\| = \max ( \|g\|\|q\| , \|R\|).
\end{equation}
\end{prop}

\begin{proof}
First, let us show that if a couple $(q,R)$ satisfies \eqref{existencedec}, 
then it must satisfy the equality \eqref{borne}. Because of 
the ultrametric inequality, 
$\|f\| \leq \max (\|g\| \|q\| , \|R\| )$. 
For the reverse inequality, we distinguish two cases.\par
If $\|gq\| \neq \|R\|$, then 
$\|f\|=\max (\|gq\|,\|R\|) =
\max ( \|g\|\|q\| , \|R\| )$ according to
 lemma \ref{multiplicatif}.\par
Otherwise  $\|gq\|=\|g\|\|q\| = \|R\|$, and we use again 
lemma \ref{multiplicatif} and its notation
(so $gq = \sum\limits_{l\in \mathbb{N}} c_lT^l$).
We get $\|gq\| = \|c_{s+k_0}\|r^{s+k_0}$. 
Since $R$ is a polynomial of degree $d$ with $d<s$, and since $f=gq+R$, and $d<s+k_0$, the coefficient $f_{s+k_0}$ of $f$ is 
$c_{s+k_0}$, hence 
$\|f\| \geq  \|c_{s+k_0}\|r^{s+k_0}  = \|g\|\|q\|$. \par
This finally proves that  $\|f\| = \max (\|g\|\|q\| , \|R\|)$.\par
From this we can conclude that the couple 
$(q,R)$ is unique
because if $f=gq'+R'$ is another decomposition, we have 
$0=g(q-q') + (R-R')$ and since $\|g\|\neq 0$, 
$\|q-q'\|=\|R-R'\| = 0$, i.e. $R=R'$ and $q=q'$. \par
Let us now show the existence of such a decomposition.
Let us set 
\[g' := \sum\limits_{m=0}^s g_mT^m.\] 
In particular, $\|g\| = \|g'\|$ because $g$ is $T$-distinguished of degree $s$.
Let us set 
\[\kappa := \frac{\max\limits_{m>s} ( \|g_m \|  r^m) }{\|g_s\|r^s} = 
\frac{\max\limits_{m>s} ( \|g_m\|r^m )}{\|g\|}.\]
Since $g$ is 
$T$-distinguished of order $s$, $\kappa<1$. 
Actually, if $\kappa = 0$ (which would mean that $g=g'$), replace $\kappa$ by $\frac{1}{2}$. 
In any case $\|g-g'\| \leq \kappa \|g\|$ and $\kappa \in ]0,1[$.\par
Next, let $N\in \mathbb{N}$ and let us set
\[f' := \sum_{k=0}^N f_kT^k.\]  
Let us assume that $N$ is big enough to satisfy $\|f-f'\| \leq \kappa\|f\|$. In particular, 
$\|f'\|=\|f\|$. \par
By definition and hypothesis, $g' \in A[T]$ is of degree $s$ and possesses an invertible dominant coefficient, which is $g_s$. 
Hence in $A[T]$, one can carry out euclidean division by $g'$ \cite[4.1.1]{Lang}, 
which gives
$f'=g'q+R$, with $R \in A_{s-1}[T]$ and $q \in A[T]$.
We can then apply the norm equality 
\eqref{borne} that we have shown in the first part of the proof,
(because $g'$ is also $T$-distinguished of order $s$):
$\|f'\|=\max ( \|g'\|\|q\| , \|R\|)$.  
In particular $\|q\| \leq \frac{\|f'\| }{\|g'\|}= \frac{\|f\|}{\|g\|} $ so that 
\[\|g\| \|q \| \leq \|f\|.\] 
Moreover $\|R\| \leq \|f'\|=\|f\|$. Thus the following holds:
\[f=f' + (f-f') =
g'q+R + (f-f')
=gq +R + (f-f')+ (g'-g)q.\]
By definition of $g'$ and of $\kappa$, 
$\|g'-g\| \leq \kappa \|g\|$, so
\begin{equation}
\label{ineg1}
\|(g'-g)q\| \leq  \|g\|\|q\| \kappa  \leq \kappa\|f\|
\end{equation} 
In addition, by hypothesis, 
\begin{equation}
\label{ineg2}
\|f-f'\| \leq \kappa \|f\|.
\end{equation} 
Hence if we set
\[h:= f-f' + (g'-g)q = f-(gq+R),\]
according to \eqref{ineg1} and \eqref{ineg2}, we obtain that $\|h\| \leq \kappa \|f\|$.\par 
To sum up, we have found some 
$\kappa \in ]0,1[$ such that 
\begin{equation}
\label{keypoint} 
\forall f\in A\{r^{-1}T\}, \  \exists q' \in A\{r^{-1}T\}, \ \exists R' \in A_{s-1}[T] 
\ \text{such that} \ \|f-(gq'+R')\| \leq \kappa \|f\|.
\end{equation}
This allows us to 
define by induction two Cauchy sequences  
$(q^i) \in \Ar $ and $(R^i) \in A_{s-1}[T]$
such that $\|f-(gq^i+R^i) \| \leq \kappa^i \|f\|$
in the following way.\par
We start with $(q^0,R^0)=(0,0)$. \par
In order to perform the induction step, let $i>0$ be given and let us assume 
that $(q^i,R^i)$ is defined. 
We set
$h^i:= f-(gq^i+R^i)$, which by induction hypothesis fulfils
$\|h^i\| \leq \kappa^i \|f\|$.
According to \eqref{keypoint}, we can define
$q' \in  \Ar $ and $R' \in A_{s-1}[T]$ such that
$h^i = gq' + R' +h'$ with
$\|q'\|\leq \frac{\|h^i\|}{\|g\|} \leq \kappa ^i \frac{\|f\|}{\|g\|} $, 
and $\|R'\| \leq \|h^i\| \leq  \kappa^i \|f\|$ and 
$\|h'\| \leq \kappa \|h^i\| \leq \kappa^{i+1} \|f\|$.
Then we set
$q^{i+1} :=q^i +q'$ and
$R^{i+1} :=R^i +R'$.\par
Then
$\|f-( gq^{i+1} +R_{i+1}  ) \|
=\|h^i -(gq +R) \| =\|h'\| \leq \kappa^{i+1} \|f\|$. 
By construction
$\|q^{i+1} -q^i\| =\|q'\|\leq \kappa^i \frac{\|f\|}{\|g\|}$ and 
$\|R^{i+1} -R^i\| = \| R' \| \leq \kappa^i\|f\|$, so these 
sequences are well Cauchy sequences. This ends our induction.\par 
Now, by completeness of
$\Ar $ and $A_{s-1}[T]$ the sequences $(q^i)$ and $(R^i)$ have a limit, that we denote 
by $q\in \Ar$ and $R \in A_{s-1}[T]$, which satisfy
$f=gq+R$ as we wanted. 
\end{proof}

\begin{cor}\textbf{Weierstrass Preparation}.
\label{preparation}
Let $g\in \Ar$ be a $T$-distinguished element of order $s$. 
There exists an unique couple
$(w,e) \in A_s[T] \times \Ar $ such that $w$ is a monic polynomial 
of degree $s$, $e$ is a multiplicative unit of $\Ar$, and $g=ew$.
\end{cor}

\begin{proof}
Using Weierstrass division, we can write
$T^s = gq+R$ with
$\|T^s\| = \max (\|g\|\|q\| , \|R\|)$, and 
$R\in A[T]_{s-1}$. 
Let us set 
\[w:=T^s-R=gq.\] 
So $w \in A_s[T]$ is a monic polynomial.
Since $g$ is $T$-distinguished of order $s$, according to lemma \ref{multiplicatif},
and if we denote by
$k_0$ the greatest index such that 
$\|q_{k_0}\|r^{k_0} = \|q\|$, and
$w= \sum\limits_{l=0}^s w_lT^l$, we obtain 
\[ \|w\| = \|gq\| = \|(gq)_{s+k_0}\|r^{s+k_0} = \|w_{s+k_0}\|r^{s+k_0}. \] 
But since $w\in A_s[T]$, necessarily,  $s+k_0=s$ 
and $k_0=0$. Hence, by definition of $k_0$, for 
all $k>0$, $\|q_0\|>\|q_k\|r^k$.\par
The coefficient of degree $s$ in $gq$ being $1$, (because $gq=T^s-R$), 
we have the equality 
\[1 = g_0q_s +g_1q_{s-1} + \ldots + g_sq_0\] 
and since $k_0=0$, 
and $g$ is $T$-distinguished of order $s$, we obtain,  
with the same reasoning that we have already used in the course of the proof of lemma \ref{multiplicatif}, that $\|g_sq_0\|>\|g_{s-i}q_i\|$ for $i=1\ldots s$. 
So $\|g_sq_0\|=\|1\|=1$, and 
$g_sq_0=1-(g_{s-1}q_1 + \ldots g_0q_s)$, with 
$\|g_sq_1 + \ldots g_0q_s \| <1$.
Thus, $g_sq_0$ is a multiplicative unit. 
Moreover, since $g_s$ is also a multiplicative unit, $q_0$ is also a multiplicative 
unit, and $\|q_0\|=\|g_s\|^{-1}$.
Hence 
\begin{equation}
\label{qmult}
q=q_0(1+\frac{q_1}{q_0}T + \ldots + \frac{q_k}{q_0}T^k + \ldots )
\end{equation} 
and since $k_0 =0$ (so $\|q_i\|r^i < \|q_0\|$ for $i>0$) and $q_0$ is a multiplicative unit,
$\| \frac{q_i}{q_0} \|r^i <1$ for all $i>0$.
Hence 
\[1+\frac{q_1}{q_0}T + \ldots + \frac{q_k}{q_0}T^k + \ldots\] 
is a multiplicative unit of 
\Ar, and according to \eqref{qmult}, $q$ is also a multiplicative unit. 
So $g=q^{-1}(T^s-R)$, with $q^{-1}$ a multiplicative unit and
$T^s-R$ a monic polynomial of degree $s$. 
So if we set 
$e:=q^{-1}$, and $w=T^s-R$ we have the expected result: $g=ew$.\par
As for the uniqueness of this decomposition, if 
$g=ew$, $e$ and $w$ being as in the statement of 
the corollary, then $w=T^s+R$ with $R\in A_{s-1}[T]$, and 
$T^s = w-R = e^{-1}g+(-R)$ which is the Weierstrass division of $T^s$ by $g$.
Hence $e$ and $R$ are unique and $w$ too 
because $w=T^s + R$.
\end{proof}

Let us assume that $\mathcal{A}$ is a $k$-affinoid algebra, 
let $(r_1, \ldots , r_n)$ be a polyradius, and let us set
$A:=\mathcal{A}\{r_1^{-1}T_1, \ldots , r_{n-1}^{-1}T_{n-1} \}$.
Then if we set 
$r=r_n$, 
$\mathcal{A}\{r_1^{-1}T_1 , \ldots , r_n^{-1} T_n\} = \Ar$, 
and we can introduce the notion of an element $T$-distinguished, apply Weierstrass theory 
to them, which corresponds to the classical one, especially if 
$\mathcal{A}=k$, where we find the classical Tate algebra 
$k\{r_1^{-1}T_1,\ldots , r_n^{-1}T_n\}$.\par
Now we state a result that we will need in the next section.

\begin{lemme}
\label{lemme_simple}
Let $\varepsilon >0$ be given and 
$\underline{r}>0$ be a polyradius. 
Let us assume that $A$ is Noetherian, and let us consider 
\[f= \sum_{\nu \in \mathbb{N}^n} f_{\nu} T^{\nu} \in A\{\underline{r}^{-1}T \} .\]
Then there exists a finite subset $J \subseteq \mathbb{N}^n$, and for all 
$\nu \in J$, a series $\phi_{\nu} \in A\{\underline{r}^{-1}T \}$ satisfying 
$\| \phi_{\nu} \| < \varepsilon$, such that 
\[f = \sum_{\nu \in J} f_{\nu}(T^{\nu} + \phi_{\nu} ) \]
and such that in the $\phi_{\nu}$'s, no terms $T^{\mu}$ with $\mu \in J$ appear. Moreover, if we fix some $\mu \in \N^n$, we can assume that $\mu \in J$. 
\end{lemme}

\begin{proof}
Let us denote by $\mathcal{I} $ the ideal generated by 
the family $\{ f_{\nu} \} _{\nu \in \mathbb{N}^n }$.
Since $A$ is Noetherian, there exists $J$ 
a finite subset of $\mathbb{N}^n$ such that  
$\mathcal{I} = A.(f_{\nu} )_{\nu \in J}$.
So for all $\mu \in \mathbb{N}^n \setminus J$ one can find a decomposition
$f_{\mu} = \sum\limits_{\nu \in J } f_{\nu}a_{\mu}^{\nu} $ with $a_{\mu}^{\nu} \in A$.
In fact, using \cite[3.7.3]{BGR}, we can even assume\footnote{Indeed, consider 
\[ \begin{array}{cccc}
\psi  : & A^{J} & \to & \mathcal{I} \\
        & (a_{\nu})_{ \nu \in J} & \to & \displaystyle \sum_{\nu \in J} a_{\nu}f_{\nu}
        \end{array} \]
According to \cite[3.7.3.1]{BGR}, $\mathcal{I}$ is a complete normed 
$A$-module, and $\psi$ is a continuous map of normed $A$-modules. 
Hence there exists 
a constant $C$ such that 
$\| \psi(x) \| \leq C \|x\|$ for all $x\in A^J$.} 
that there exists a real constant $C>0$ such that
\begin{equation}
\label{eqJnu}
\forall \mu \in \mathbb{N}^n, \ \forall \nu \in J, \ \|a_{\mu}^{\nu} \| \leq C \|f_{\mu} \|.
\end{equation}
Then, let us  define for $\nu \in J$   
\[\phi_{\nu} = \sum\limits_{\mu \in \mathbb{N}^n \setminus J }
a_{\mu}^{\nu} T^{\mu}.\] 
Since 
$\|a_{\mu}^{\nu} \| \leq C \|f_{\mu} \| $, $\phi_{\nu} \in \ara $.  
Hence, in $A\{\underline{r}^{-1}T \}$, the following equality is satisfied: 
\begin{equation}
\label{eq_formsimple}
f= \sum\limits_{\nu \in J} f_{\nu} ( T^{\nu} + 
\sum\limits_{\mu \N^n \setminus J} a_{\mu}^{\nu} T^{\mu} ) =
\sum\limits_{\nu \in J} f_{\nu} (T^{\nu} + \phi_{\nu} ).
\end{equation}
Now, if $\nu_0 \notin J$ we set 
$J'=J \cup \{\nu_0\}$, 
$\phi'_{\nu_0}:=0$, and for 
$\nu\in J$, 
$\phi'_{\nu}:= \sum\limits_{\mu\in \N^n \setminus J'} {a}_{\mu}^{\nu} T^{\mu}$. 
One checks that the properties mentioned above still hold, namely 
$\|{a}^{\mu}_{\nu} \| \leq C \|f_{\mu}\|$, where the constant $C$ has not been changed, and 
\[ f = \sum_{\nu \in J'} f_{\nu} (T^{\nu} + \phi'_{\nu} ).\]
Moreover, 
\[C\| f_{\mu}\| \underline{r}^{\mu} \xrightarrow[|\mu| \rightarrow + \infty]{} 0,\] so there exists a finite set $K\subset \N^n$ such that 
\[ \forall \nu \in J, \ \forall \mu \in N^n \setminus K, \ \|a_{\mu}^{\nu} \| < \varepsilon.\]
Hence if we increase $J$ adding the elements of $K \setminus J$ to $J$, 
we will manage to obtain a decomposition 
\[ f = \sum_{\nu \in J} f_{\nu}(T^{\nu} + \phi_{\nu} ) \] 
such that  $\| \phi_{\nu} \| < \varepsilon$ for all $\nu \in J$.
\end{proof}

\subsection{Equivalence of the two notions}
\label{section1.4}

From now on, 
$\mathcal{A}$ will be a $k$-affinoid algebra, and 
$\underline{r} \in {\mathbb{R}^*_+}^n$ a polyradius such that
$\underline{r} > \underline{1}$ and we will set 
$\ara = \Arn$. If $\nu \in \N^n$ we will set 
\[ T^\nu := T_1^{\nu_1}T_2^{\nu_2} \ldots T_n^{\nu_n}.\] 
If $\nu = (\nu_1, \ldots , \nu_n) \in \mathbb{N}^n$, we will set 
\[|\nu|_{\infty} = \max\limits_{i=1\ldots n } \nu_i.\] 
If $\underline{r} \in {\mathbb{R}^*_+}^n$ and $\nu \in \mathbb{N}^n$, we will set 
\[ \underline{r}^{\nu} = \prod_{i=1}^n r_i^{\nu_i}. \]
When $\mu, \nu \in \mathbb{N}^n$, we will say that 
$\mu <_{lex} \nu$ when $\mu$ is smaller than $\nu$ with respect to the lexicographic order, 
that is to say when there exists an index $m$ such that 
$\mu_m < \nu_m$ 
and $\mu_{m-1}=\nu_{m-1}, \ldots , \mu_{1} = \nu_{1}$. \par
\label{notationfx}
We will use the following notation. 
If $\A$ is a $k$-affinoid algebra, 
$f= \sum\limits_{n\in \mathbb{N} } a_n T^n \in \mathcal{A}\{r^{-1} T \}$ 
and $x\in \affin{A}$, we will denote by $f_x$ the element of 
$\mathcal{H}(x) \{r^{-1}T\}$ defined by 
\[f_x = \sum_{n\in \N} a_n(x)T^n.\]\par 
Since $\mathcal{A}$ is Noetherian, we can apply lemma \ref{lemme_simple} to it.

\begin{prop}
\label{unite} 
Let  
$f = \sum\limits_{\nu \in \mathbb{N}^n } f_{\nu } T^{\nu} \in \ara$.
There exists a constructible covering of $X$, 
$(X_i,S_i) \stackrel{\varphi_i}{\dashrightarrow} X$, $i=0..N$, 
such that, if we consider the following Cartesian diagrams:
\[
\xymatrix{
(X_i,S_i) \ar@{-->}[r]^{\varphi_i} & X \\
X_i\times \Er \ar[u]^{\pi_i } \ar@{-->}[r]^{\varphi_i '} & X\times \Er \ar[u]^{\pi} 
}
\]
and if we denote by $\mathcal{A}_i$ the $k$-affinoid algebra
of $X_i$, for all 
$i=1..N$, there exist 
$a_i \in \mathcal{A}_i$ and a function 
\[g_i = \sum\limits_{\nu \in \mathbb{N}^n } g_{i,\nu}  T^{\nu } \in 
\mathcal{A}_i\{\underline{r}^{-1 }T \} \] 
such that
\begin{itemize}
\item  For all $i$, the family 
$\{g_{i,\nu} \}_{\nu \in \mathbb{N}^n } $ generates the unit ideal in 
$\mathcal{A}_i$.
\item For all $i$,  
$\varphi_i'^* (f)_{| \pi_i^{-1}(S_i) } = (a_ig_i )_{|\pi_i^{-1} (S_i) }$.
\end{itemize}

\end{prop}

\begin{proof}
According to lemma \ref{lemme_simple} (here we will not use the 
extra condition $\| \varphi_{\nu} \| < \varepsilon$ of this lemma), 
we can find 
a finite subset $J \subseteq \mathbb{N}^n$, and
for $\nu \in J$ some $\phi_{\nu} \in \ara$ such that:  
\[f=\sum\limits_{\nu \in J} f_{\nu} (T^{\nu} + \phi_{\nu} ).\]
Let us fix any $r>1$, and for each  
$\nu \in J$, let us consider the constructible datum 
$(X_{\nu}, S_{\nu} ) \stackrel{\varphi_{\nu} }{\dashrightarrow} X$
where the affinoid algebra of
$X_{\nu} $ is 
$  \mathcal{A}  \{ r^{-1}t_{\mu} \}_{\mu \in J \setminus \{ \nu \} } / (f_{\mu} -t_{\mu} f_{\nu} )$, 
and 
\[S_{\nu} := 
\{ x \in X_{\nu} \ \big| \ \ |f_{\kappa}(x)| \leq  |f_{\nu}(x)| \ \ \forall \kappa \in J \setminus \{ \nu \} \ \text{and} \ f_{\nu}(x) \neq 0 \}.\]
This gives rise to the following cartesian diagrams:
\[
\xymatrix{
(X_{\nu} ,S_{\nu} ) \ar@{-->}[r]^{\varphi_{\nu} } & X \\
X_{\nu}\times \Er \ar[u]^{\pi ' } \ar@{-->}[r]^{\varphi_{\nu} '} & X\times \Er \ar[u]^{\pi} 
}
\]
Now,
\[\varphi_{\nu}'^*(f) =
f_{\nu} \big( T^{\nu} +\phi_{\nu} + 
 \sum\limits_{\mu \in J \setminus \{\nu \} } t_{\mu} (T^{\mu} + \phi_{\mu} ) \big). \]
For $\nu \in J$, we set 
\[g_{\nu} = T^{\nu} + \phi_{\nu} + 
\sum_{\mu \in J\setminus \{ \nu \} } t_{\mu}(T^{\mu} + \phi_{\mu} ).\]
Hence, 
\[\varphi_{\nu}'^{*}(f)= f_{\nu}g_{\nu}.\]
Moreover, if we set $g_{\nu}=\sum\limits_{\mu \in \mathbb{N}^n} g_{\nu , \mu} T^{\mu}$, 
according lemma \ref{lemme_simple}, 
the coefficient of index $\nu$, $g_{\nu, \nu }$, is $1$, 
so the coefficients of $g_{\nu}$ generate the unit ideal.
Finally, let us denote by $\mathcal{I}$ the ideal of 
$\A$ generated by the family $(a_\nu)_{\nu \in J}$. 
By construction, $\mathcal{I}$ 
also equals the ideal generated by $(a_\nu)_{\nu \in \N^n}$.
Then, according to the definition of the $S_\nu's$:
\[\bigcup\limits_{\nu \in J} \phi_{\nu} (S_{\nu} ) = 
\{x\in X \ \big| \ \exists \nu\in J  \ \text{such that} \ f_{\nu}(x)\neq 0 \}
= X \setminus  V(\mathcal{I}). \]
Thus, if we set $S_0=V(\mathcal{I})$, then 
 $(X,S_0) \xrightarrow[]{id} X$ is an \ECD \ 
and $id^*(f)_{|S_0} =f_{|S_0} =0$. \par 
Now if we regroup the constructible data 
$(X_{\nu}, S_{\nu} ) \stackrel{\varphi_{\nu} }{\dashrightarrow} X$, for $\nu \in J$, with 
$(X,S_0) \xrightarrow{\varphi}  X$, we obtain the desired constructible covering.
\end{proof}

\begin{defi}
\label{defiWeieraut}
Let $\underline{r} \in (\R_+^*)^n$ be a polyradius and 
$d_1, \ldots, d_{n-1}$ some integers such that 
\begin{equation}
\label{condWa}
\forall \  i=1\ldots n-1, \ r_n^{d_i} \leq r_i.
\end{equation}  Then 
\[\sigma : 
\begin{cases}
T_i \mapsto & T_i + T_n^{d_i} \hspace{1cm} \text{for} \ 1\leq i\leq n-1 \\
T_n \mapsto & T_n 
\end{cases} \]
is an automorphism of \ara. 
We will call such an automorphism (as well as the automorphism it induces on 
the $k$-analytic space $\B_{\underline{r}}$) a Weierstrass automorphism.
\end{defi}
\begin{rem}
\label{Weierstab}
If $\underline{r}> \underline{1}$, we will use that $\sigma$ induces a 'classical' 
Weierstrass automorphism of 
$\A\{T_1,\ldots,T_n\}$, hence of $X\times \B^n$. 
\end{rem}
Recall the following classical result. If $f\in k\{T_1,\ldots,T_n\}$, then there exists 
a Weierstrass automorphism $\sigma$ of $k\{T_1,\ldots,T_n\}$ such that 
$\sigma(f)$ is $T_n$-distinguished. 
Roughly speaking, the next lemma says that 
if $\A$ is a $k$-affinoid algebra, $f\in \A\{T_1,\ldots,T_n\}$ is overconvergent, then locally on $X=\affin{A}$, we can obtain an analogous result.

\begin{prop}
\label{proplocalweie}
Let $\A$ be a $k$-affinoid algebra. 
Let $X = \affin{A}$ and let $x\in X$.
Let $\underline{r}\in \R^n$ be a polyradius such that $\underline{r}>1$.  
\begin{enumerate}
\item 
Let $f\in \ara$ such that 
$f_x \neq 0$. Then there exist an affinoid neighbourhood 
$V = \affin{B}$ of $x$, a polyradius $\underline{\rho}$ such that 
$1<\underline{\rho} \leq \underline{r}$, and $\sigma$ a Weierstrass automorphism of 
$\mathcal{B}\{ \underline{\rho}^{-1} T \}$ such that in 
$\mathcal{B}\{ \underline{\rho}^{-1} T \}$
\[\sigma(f) = ag\]
where $a\in \mathcal{B}$ and $g \in \mathcal{B}\{ \underline{\rho}^{-1} T \}$
is $T_n$-distinguished. 
\item 
More generally, let us consider $m$ functions 
$f_1,\ldots,f_m \in \ara$ such that for all $i$
$(f_i)_x \neq 0$. Then there exist an affinoid neighbourhood 
$V = \affin{B}$ of $x$, a polyradius $\underline{\rho}$ such that 
$1<\underline{\rho}\leq  \underline{r}$, and $\sigma$ a Weierstrass automorphism of 
$\mathcal{B}\{ \underline{\rho}^{-1} T \}$ such that for all $i$
\[\sigma(f_i) =a_ig_i\]
where $a_i\in \mathcal{B}$ and $g_i \in \mathcal{B}\{ \underline{\rho}^{-1} T \}$
is $T_n$-distinguished. 
\end{enumerate}
\end{prop}

\begin{proof}
We first prove (1).\par
Step 1. 
Let us write 
\[f = \sum_{\nu \in \N^n} f_\nu T^\nu \in \ara.\]
Let us consider $\mu\in \N^n$ the greatest index with respect to the lexicographic order such that 
\[ \max_{\nu \in \N^n} |f_{\nu}(x)| = |f_\mu(x)|.\]
Since by assumption $f_x\neq 0$, it is true that $f_\mu(x) \neq 0$. 
According to lemma \ref{lemme_simple}, there exists a finite set 
$J\subset \N^n$ such that $\mu\in J$, and for each $\nu \in J$ a 
series $\phi_\nu \in \ara$ which satisfies $\| \phi_\nu\|_{\ara} <1$ such that
\begin{equation}
\label{egfond}
 f= \sum_{\nu \in J} f_\nu(T^\nu + \phi_\nu).
\end{equation}\par 

Step 2. Let us consider some 
$\nu \in J$ and let us assume that 
\[ |f_\nu(x)| < |f_\mu(x)|.\]
Then we pick some $a,b\in \R$ such that 
\[|f_\nu(x)| <a<b< |f_\mu(x)|.\]
Next, let us introduce the affinoid domain of $X$:
\[W:= \{ z\in X \st |f_\nu(z)|\leq a \ \and \ b\leq |f_\mu(z)|  \} = \affin{B}.\]
By construction, $W$ is an affinoid neighbourhood of
$x$, $f_\mu$ is invertible in $\mathcal{B}$ and 
\[ \left\| \frac{f_\nu}{f_\mu} \right\|_{\mathcal{B}} \leq \frac{a}{b} <1.\]
So we can write:
\[f_\nu(T^\nu + \phi_\nu) = f_\mu 
\left( \frac{f_\nu}{f_\mu} (T^\nu +\phi_\nu) \right).\]
Next we consider some polyradius 
$\underline{1}<\underline{\rho} \leq \underline{r}$. 
Clearly 
\[ \underline{\rho}^\nu \xrightarrow[\underline{\rho} \to 1]{} 1 .\]
So we can chose some $\underline{\rho}$ close enough to $\underline{1}$ such that 
\[ \left\| \frac{f_\nu}{f_\mu}  T^\nu \right\|_{\Brho} <1.\]
Since we already knew that 
$\|\phi_\nu\|_{\Brho}<1$ it follows that 
\[ \left\| \frac{f_\nu}{f_\mu} ( T^\nu+ \phi_\nu) \right\|_{\Brho} <1.\]
But since 
\[f_\nu(T^\nu +\phi_\nu) = f_\mu \left( \frac{f_\nu}{f_\mu}(T^\nu + \phi_\nu ) \right),\]
if we set 
\[\phi'_\mu := \phi_\mu + \frac{f_\nu}{f_\mu} ( T^\nu+ \phi_\nu)\]
we still have that 
$ \| \phi'_\mu \|_{\Brho}<1$ and 
\[f_\mu(T^\mu +\phi_\mu) + f_\nu (T^\nu +\phi_\nu) = 
f_\mu( T^\mu + \phi'_\mu).\]
Hence we can remove $\nu$ from $J$ and replace $\phi_\mu$ by $\phi'_\mu$. 
The equality \eqref{egfond} will still be satisfied. \par  
If we repeat this process for each 
$\nu \in J$ such that 
$|f_\nu(x)| < |f_\mu(x)|$, we can assume that 
\[ \forall \nu \in J, \ |f_\nu(x)| = |f_\mu(x)|.\] 
Thus, according to the definition of $\mu$, this implies that $\mu$ is the greatest index in $J$ with respect to the lexicographic order. 
 \par
Step 3. 
Then we set 
\[ d:= 1 +\max_{\nu \in J} |\nu|.\]
Since by assumption 
$\underline{1}<\underline{r}$, if we take $s>1$ which is close enough to $1$, we can assert that 
\begin{equation}
\label{inegsr}
 \underline{1} < (s^{d^{n-1} }, s^{d^{n-2} } , \ldots,s^d,s) \leq \underline{r}.
 \end{equation}
We fix a number $s>1$, which satisfies \eqref{inegsr}, and we set 
\begin{equation}
\label{defrho}
\underline{\rho} := (s^{d^{n-1} }, s^{d^{n-2} } , \ldots,s^d,s).
\end{equation}
In these conditions, 
it is easy to check that $\underline{\rho}$ satisfies condition \eqref{condWa}
of definition \ref{defiWeieraut} 
(actually, $\underline{\rho}$ has been 
defined in \eqref{defrho} to further this goal), so 
\[
\sigma : \left\{
\begin{array}{lrcl}
       &   T_1    & \mapsto & T_1+T_n^{d^{n-1}} \\
         &  \vdots  &         &  \vdots \\
         & T_i      & \mapsto & T_i + T_n^{d^{n-i}} \\
         & \vdots   &         & \vdots \\
         & T_{n-1}  &\mapsto  & T_{n-1} +T_n^d \\
         & T_n      & \mapsto & T_n
\end{array}
\right. 
\]
defines a Weierstrass automorphism of  $\Brho$. 
Then, for $\nu \in J \setminus \{\mu\}$ 
\[\sigma(f_\nu(T^\nu + \phi_\nu)) = f_\nu(\sigma(T^\nu) + \sigma(\phi_\nu) ) 
= f_\mu \left( \frac{f_\nu}{f_\mu} (\sigma(T^\nu) + \sigma(\phi_\nu)) \right) .\]
Since 
$\|\sigma(\phi_\nu) \|_{\Brho} = \|\phi_\nu\|_{\Brho} <1$, in fact we can chose 
$s$ close enough to $1$, so that 
\begin{equation}
\label{inegs'}
 s\| \phi_\nu\| <1.
 \end{equation}
Then we make the following calculation. If $\nu\in J$:
\begin{equation}
\|\sigma(T^\nu) \|_{\Brho} = \| T^\nu\|_{\Brho} = \prod_{k=1}^n\left(s^{d^{n-k}} \right)^{\nu_k} = s^{ \big(\sum\limits_{k=1}^n \nu_kd^{n-k} \big)} .
\end{equation} 
Moreover, we remark that 
$\sum\limits_{k=1}^n\nu_k d^{n-k}$ is nothing else but the integer encoded by $\nu$ 
in base $d$. 
Since by assumption, for all $\nu \in J\setminus \{\mu\}$ we have 
$\nu <_{lex} \mu$, it follows that for $\nu \in J\setminus \{\mu\}$
\[ \sum_{k=1}^n\nu_k d^{n-k} +1 \leq \sum_{k=1}^n\mu_k d^{n-k}.\]
As a corollary, 
\begin{equation}
\label{inegmunu}
s\|\sigma(T^\nu) \|_{\Brho}   \leq  \| \sigma(T^\mu) \|_{\Brho} .
\end{equation}
Let us now consider some $s'\in \R$, such that  
$1<s'<s$ and let us consider 
\[V:= \{ z\in X \st \forall \nu\in J\setminus \{\mu\}, \ |f_\nu(z)|\leq s'|f_\mu(z)| \}.\]
Then by construction, $V$ is an affinoid neighbourhood of $x$. 
Let us then replace  $\mathcal{B}$ by the affinoid algebra of $V$. 
Then by construction still, 
for all $\nu \in J\setminus \{\mu \}$, 
\[ \| \frac{f_\nu}{f_\mu} \|_{\mathcal{B}} \leq s' <s.\]
So according to \eqref{inegmunu}
\[ \| \frac{f_\nu}{f_\mu} \sigma (T^\nu) \|_{\Brho} < s \| \sigma(T^\nu)\|_{\Brho}
\leq \| \sigma(T^\mu) \|_{\Brho} .\]
So, according to \eqref{inegs'}, we can assume that 
\[ \| \frac{f_\nu}{f_\mu} \sigma(\phi_\nu) \|_{\Brho}  
\leq s' \| \sigma(\phi_\nu) \|_{\Brho}  = s' \|\phi_\nu \|_{\Brho}  <1 
\leq \|\sigma (T^\nu) \|.\]
Thus 
\[\sigma(f_\nu (T^\nu +\phi_\nu)) = 
f_\mu (\frac{f_\nu}{f_\mu} ( \sigma(T^\nu) + \sigma( \phi_\nu)) )\]
where 
\[ \| \frac{f_\nu}{f_\mu} ( \sigma(T^\nu) + \sigma( \phi_\nu)) \|_{\Brho}  
<\| \sigma(T^\mu) \|_{\Brho} .\] \par 
Step 4.
So 
\[\sigma(f) = 
f_\mu \left( \sigma(T^\mu) + \sigma(\phi_\mu) + \sum_{\nu\in J\setminus \{\mu\} } \frac{f_\nu}{f_\mu} ( \sigma(T^\nu) + \sigma( \phi_\nu)) \right)\]
Hence if we set 
\[ \phi = \sigma(\phi_\mu) + \sum_{\nu \in J\setminus\{\mu\}}  \frac{f_\nu}{f_\mu}(\sigma(T^\mu) + \sigma(\phi_\nu))\]
the preceding inequalities imply that  
$\|\phi\|_{\Brho} < \|\sigma(T^\mu)\|_{\Brho}$, and by construction  
\[ \sigma(f) = f_\mu (\sigma (T^\mu) + \phi) ).\]
It follows that 
$\sigma(T^\mu) + \phi$ is $T_n$-distinguished of order 
$\sum_{k=1}^n \mu_k d^{n-k}$, which ends the proof of (1). \par
For the proof of (2), it suffices to remark that we could have handled the proof of (1) simultaneously for all the $f_i's$. The main point being that in step 3, we have to take some $d$ big enough that works for all $f_i's$ simultaneously.
\end{proof}

\begin{lemme}
If $S$ is an overconvergent constructible subset of $X$, then $S$ is an 
overconvergent subanalytic subset of $X$.
\end{lemme}
\begin{proof}
It is sufficient to prove that if 
$(Y,T) \stackrel{\varphi}{\dashrightarrow} X$ is a constructible datum, then 
$\varphi(T)$ is overconvergent subanalytic in $X$. \par 
We claim that if $\varphi$ is a constructible datum of complexity $n$, 
there exist some polyradii 
$\underline{s},\underline{r} \in \R^n$ such that $\underline{s}\in \val^n$ and 
$0< \underline{s}<\underline{r}$, and some closed immersion $\iota$:
\[ \xymatrix{
Y \ar@{-->}[rd]^\varphi \ar@{^{(}->}[r]^\iota & X \times \Bur \ar[d]^{\pi} \\
 & X }
 \]
such that 
$\iota(T)\subset X \times \Bus$. 
Indeed this follows from the definition of a constructible datum, and is proved easily by induction on the complexity of the constructible datum $\varphi$. \par 
 Hence $\varphi(T) = \pi (\iota(T))$, and since 
 $\iota(T)$ is a semianalytic subset of 
 $X \times \Bur$ such that 
 \[ \iota(T) \subset X \times \Bus\]
 it follows that 
 $\pi (\iota(T))$ is an overconvergent subanalytic subset of $X$.
\end{proof}

\begin{theo}
\label{theo_eq}
Let $S \subset X$. If $S$ is overconvergent subanalytic, $S$ is also 
overconvergent constructible.
\end{theo}

\begin{proof}
Let $S$ be an overconvergent subanalytic subset of $X$. 
By definition, there exist 
$r>1$, $R$ a 
semianalytic subset of $X \times \Er$ such that 
$S = \pi \left( R \cap (X \times \E ) \right)$. 
We then show  by induction on $n$ that $S$ is overconvergent constructible.\par
If $n=0$, there is nothing to prove since in that case, $S$ is 
then a \GSA \ of $X$, in particular it is an \OC.\par 
Let then $n>0$ be given and let us assume that the theorem holds for 
integers $<n$. 
In order to prove the theorem, we can actually assume that $R$ is a basic semianalytic subset (see remark \ref{rembasicsa}), i.e. 
that there are 
$2m$ functions 
$f_1, \ldots , f_m , g_1, \ldots , g_m \in \ara$ and 
$\Diamond_j \in \{ \leq , < \}$ for $j=1 \ldots m$ 
such that 
\begin{equation}
\label{formulaT}
R = \{x\in X \times \Br^n  \ \big| \ |f_j(x)| \Diamond_j |g_j(x)| \ j=1\ldots m \}.
\end{equation} \par 
Step 1. 
According to proposition \ref{unite} 
we can find a constructible covering 
$(X_i,S_i) \stackrel{\varphi_i}{\dashrightarrow} X$  
where $X_i = \mathcal{M}(\mathcal{B}_i)$ which induces the following cartesian diagram
\[
\xymatrix{ 
X_i\times \Bur^n \ar[r]^{\varphi'_i} \ar[d]_{\pi_i} & X \times \Bur^n \ar[d]^\pi \\
X_i \ar@{-->}[r]^{\varphi_i} & X }
\] 
such that for all $j=1 \ldots m$, 
\begin{align}
\label{Fij}
{\varphi'}_i^*(f_j)_{| \pi_i^{-1}(S_i)} = (a_j^iF_j^i)_{|\pi_i^{-1}(S_i)} \\ 
\label{Gij}
{\varphi'}_i^*(g_j)_{| \pi_i^{-1}(S_i)} = (b_j^iG_j^i)_{| \pi_i^{-1}(S_i)}
\end{align}  
where 
$a^i_j, b^i_j \in \mathcal{B}_i$, 
$F^i_j, G^i_j \in \mathcal{B}_i\{ \underline{r}^{-1} T \}$, and the coefficients of 
$F_j^i$ (resp. of $G_j^i$) generate the unit ideal in $\mathcal{B}_i$. 
Then for each $i$ we set 
\begin{align*}
R_i := \{x\in X_i \times \Br^n  \ \big| \ 
|a^i_jF^i_j(x)| \Diamond_j |b^i_jG^i_j(x)| \ j=1\ldots m \}
\end{align*} 
So \eqref{Fij} and \eqref{Gij} imply precisely that  
\[ R_i \cap \pi_i^{-1}(S_i)= {\varphi'}_i^{-1}(R) \cap \pi_i^{-1}(S_i).\]
So if we set 
\[U_i := \pi_i( R_i \cap (X_i\times \B^n)) \]
then, 
\[ \varphi_i (S_i \cap U_i) = \varphi_i(S_i) \cap S\]
hence since the $\varphi_i(S_i)$ form a covering of $X$,  
\[ S = \bigcup_{i=1}^n \varphi(S_i\cap U_i).\] 
So if we prove that $\varphi_i (S_i \cap U_i)$ is 
overconvergent constructible, we are done. \par  
But actually, since each $S_i$ is overconvergent constructible in 
$X_i$ (it is even semianalytic, see remark \ref{remcdsa})  if we prove that $U_i$ is an overconvergent constructible subset 
of $X_i$, then it will follow that 
$S_i \cap U_i$ is an overconvergent constructible subset of $X_i$, and then according to 
proposition \ref{prop_gen} (2), $\varphi_i (S_i \cap U_i)$ will be overconvergent constructible in $X$. Thus, we restrict to prove that $U_i$ is overconvergent constructible in $X_i$.\par
Step 2. We can then replace $X$ by one of the $X_i$'s and assume that $R$ is defined by 
\begin{equation}
\label{formulaT1}
R = \{x\in X \times \Br^n  \ \big| \ |a_j f_j(x)| \Diamond_j |b_j g_j(x)| \ j=1\ldots m \}
\end{equation}
with $a_j,b_j \in \A$, $f_j,g_j \in \ara$ such that for all $j$, 
the coefficients of $f_j$ (resp. of $g_j$) generate the unit ideal of $\A$. In this situation we must show that $S$ is overconvergent constructible in $X$ where
\[S =\pi(R \cap (X\times \B^n)).\]
\par
Let then $x\in X$. The above property of the $f_j$'s and $g_j's$ implies that 
$(f_j)_x\neq 0$ and $(g_j)_x \neq 0$. 
So we can apply proposition \ref{proplocalweie} to them. 
Thus there exist
an affinoid neighbourhood $V =\affin{B}$ of $x$, some polyradius 
$\underline{1} < \underline{\rho} \leq  \underline{r}$ and some Weierstrass automorphism 
$\sigma$ of 
$\Brho$ such that for each $j$, 
\begin{align}
\label{sigmaeq1}\sigma(f_j) = \alpha_jF_j \\
\label{sigmaeq2} \sigma(g_j) = \beta_j G_j
\end{align} 
where $\alpha_j,\beta_j\in \mathcal{B}$ and $F_j,G_j$ are $T_n$-distinguished elements 
of $\Brho$.
Let us then consider the following commutative diagram:
\[ \xymatrix{
V \times \B_{\underline{\rho}} \ar[r]^{\overset{\sigma}{\sim}} \ar[rrd]_{\pi''} & 
V \times \B_{\underline{\rho}} \ar[rd]^{\pi'} \ar[r]^\iota & X\times \Bur \ar[d]^\pi \\
   &  & X }
   \]
where $\iota$ is the embedding of the affinoid domain $V \times \B_{\underline{\rho}}$ in 
$X \times \B_{\underline{\rho}}$.
Then let us set
\begin{align*}
 R':= \iota^{-1} (R) \ \and \  R'' := \sigma^{-1} (\iota^{-1}(R))  . 
\end{align*} 
First it is clear that 
\begin{align}
\notag S\cap V & = \pi(R \cap(X \times \B^n) ) \\
 \notag        & = \pi( R \cap (V \times \B^n)) \\
\notag         &=\pi'(R' \cap (V \times \B^n)) \\
\label{eqpip}&= \pi'' (R'' \cap (V \times \B^n ))
\end{align} 
For the last equality \eqref{eqpip}, we use  that the Weierstrass automorphism 
$\sigma$ induces an isomorphism  of $V \times \B^n$ as noticed in remark \ref{Weierstab}. \par
But since we know that being overconvergent constructible 
is a local property (see corollary \ref{corlocalover}), 
if we prove that $S\cap V$ is overconvergent constructible, 
then since $x$ has been taken arbitrarily, and since 
$V$ is an affinoid neighbourhood of $x$, this will conclude the proof. 
So we can restrict to prove that 
$ \pi'' (R'' \cap (V \times \B^n ))$ is overconvergent constructible in $V$. 
Now according to \eqref{formulaT1}--\eqref{sigmaeq2}, 
$R''$ is a semianalytic subset of $V\times \B_{\underline{\rho}}$ defined by inequalities between functions 
$a_j\alpha_jF_j$, $b_j\beta_jG_j$, where 
$a_j,\alpha_j,b_j,\beta_j \in \mathcal{B}$ and $F_j,G_j \in \Brho$ are $T_n$-distinguished. \par
Step 3. 
So replacing $X$ by $V$, $R$ by $R''$, $a_j\alpha_j$ by $a_j$, 
$b_j\beta_j$ by $b_j$, $F_j$ by $f_j$ and $G_j$ by $g_j$, we can assume that 
\begin{equation}
\label{formulaT2}
R = \{x\in X \times \Br^n  \ \big| \ |a_j f_j(x)| \Diamond_j |b_j g_j(x)| \ j=1\ldots m \}
\end{equation}
where 
$a_j,b_j \in \A$ and $F_j,G_j \in \ara $ are $T_n$ distinguished in $\ara$. 
Then, we apply the Weierstrass preparation theorem \ref{preparation} to 
$f_j$ and $g_j$. 
So there exist $e_j,e_j' \in \ara$ some multiplicative units, 
and $w_j,w_j'$ some monic polynomials of 
$\mathcal{A}\{r_1^{-1}T_1, \ldots , (r_{n-1})^{-1} T_{n-1} \} [T_n]$ such that 
\begin{align*}
f_j = e_jw_j \\
g_j = e'_jw'_j.
\end{align*} 
So if we set 
\begin{align*}
P_j:= a_jw_j \\
Q_j := b_jw'_j,
\end{align*}
we have that $P_j,Q_j \in \mathcal{A}\{r_1^{-1}T_1, \ldots , (r_{n-1})^{-1} T_{n-1} \} [T_n]$. 
In addition, since $e_j,e'_j$ are  multiplicative unit, for all 
$x\in X\times \Bur$, 
$|e_j(x)| = \|e_j\| \in \val$. So we finally obtain that 
\begin{align}
\notag R &= \{x\in X \times \Br^n  \ \big| \ |a_j f_j(x)| \Diamond_j |b_j g_j(x)| \ j=1\ldots m \} \\
 & = \{x\in X \times \Br^n  \ \big| \  \|e_j\| |P_j(x)| \Diamond_j \|e'_j\| |Q_j(x)| \ j=1\ldots m \}.
 \end{align}
Let us consider the projection along the last coordinate of $\Bur$, 
\[ X\times \Bur \xrightarrow[]{\pi_1} X \times \B_{(r_1,\ldots,r_{n-1})} 
\xrightarrow[]{\pi_2} X \]
according to \cite[2.5]{Duc_sa}
$\pi_1(R \cap (X\times \B^n)$ is a semianalytic subset of 
$X \times \B_{(r_1,\ldots,r_{n-1})}$.  
So by induction hypothesis, 
\[\pi_2 ( \pi_1(R \cap (X\times \B^n) ) \] 
is overconvergent constructible in $X$. 
Since $\pi_2 \circ \pi_1 =\pi$, this proves that $S$ is overconvergent constructible and ends the proof. 
\end{proof}
We have then proved 
\begin{theo}
\label{theoequiv}
Let $X$ be a strictly $k$-affinoid space, and $S\subset X$. Then $S$ is 
overconvergent subanalytic if and only if $S$ is overconvergent constructible.
\end{theo}

Hence thanks to this theorem we can use some obvious properties of overconvergent 
subanalytic (resp. constructible) subsets to prove less obvious results 
about overconvergent constructible (resp. subanalytic) subsets. 
For instance we can obtain the non-trivial result concerning overconvergent subanalytic subsets:
 
\begin{prop}
Let $X$ be a strictly $k$-affinoid space. The class of overconvergent subanalytic subsets 
of $X$ is stable under finite boolean combination\footnote{In fact, the only non-trivial result is that overconvergent subanalytic subsets are stable under taking complements.}.
\end{prop}

\begin{proof}
This was proven for overconvergent 
constructible subsets in proposition \ref{prop_gen}.
\end{proof}

In the same way, we obtain a non-obvious stability property for 
overconvergent constructible subsets:
\begin{cor}
\label{projection}
 Let $\underline{r} \in \R^n$ be a polyradius such that $\underline{r} >\underline{1}$, 
and 
$S \subseteq X \times \Bur$ be an overconvergent subanalytic (or 
constructible) subset of 
$X\times \Bur$. 
Then 
$\pi \left(S \cap ( X\times \B^n ) \right)$ is an overconvergent subanalytic (or
constructible) subset of $X$.
\end{cor}

\begin{proof}
If $S$ is an overconvergent 
subanalytic subset of $X \times \Bur$, by definition, there exists 
$s>1$, an integer $m$ and $T$ a  semianalytic subset of  
$X \times \Bur \times \Bs^m$ such that 
$S = \pi_2(T \cap ( (X\times \Bur) \times \B^m))$ where
$\pi_2 :  (X\times \Bur) \times \Bs^m \to X\times \Bur$ is the natural projection.
Hence $\pi(S \cap ( X\times \B^n)) = 
\pi_2 ( T \cap ( (X \times \B^n) \times \B^m)) 
=\pi_2(T \cap ( X \times \B^{n+m} )$ where 
$\pi_2 : X\times \Bur  \times \Bs^m \to X$ is the natural projection 
(so $\pi_2 = \pi \circ \pi_1$).
Hence $S$ is an overconvergent subanalytic subset of 
$X$.
\begin{comment}
The statement is clear when $S$ is an \OS \ of $X \times \Es$, 
and is then true for $S$ an \OC \ thanks to the previous theorem.
\end{comment}
\end{proof}

\subsection{From a global to a local definition}
\label{section1.5}

\begin{defi}
\label{def_local}
Let $\mathcal{P}$ be the data, for each $k$-affinoid space $X$, of 
a family $\mathcal{P}_X$ of subsets of $X$. If 
$S$ is a subset of a $k$-affinoid space $X$, we will say that $S$ satisfies the 
property $\mathcal{P}$ if $S \in \mathcal{P}_X$. We will say that 
\begin{itemize}
\item The property $\mathcal{P}$ is a $G$-local property if for all $k$-affinoid spaces $X$ and 
any subset $S$ of $X$, $S$ satisfies the property $\mathcal{P}$ if and only 
if for all finite affinoid coverings $\{X_i\}$ of $X$, 
$S\cap X_i$ satisfies the property $\mathcal{P}$ relatively to $X_i$ 
(i.e. $S\cap X_i \in \mathcal{P}_{X_i}$).
\item the property $\mathcal{P}$ is a local property if for all affinoid spaces $X$ and any subset $S$ of $X$, 
$S \in \mathcal{P}_X$ if and only if for all $x\in X$, there exists an affinoid neighbourhood 
$U$ of $x$ such that $S\cap U \in \mathcal{P}_U$.
\end{itemize}
\end{defi}
If $S$ is a subset of a topological space $X$, we will denote by $ \overset{\circ}{S} $ 
the topological interior of $S$.
Note that using the compactness of affinoid spaces, saying that $\mathcal{P}$ is a local property is equivalent 
to requiring that for all $k$-affinoid spaces $X$ and any $S\subseteq X$, 
$S$ satisfies $\mathcal{P}$ if 
and only if for any finite affinoid covering $\{X_i\}$ of $X$ such that 
$\{ \overset{\circ}{X_i} \}$ is also a 
covering of $X$, then 
$S \cap X_i \in \mathcal{P}_{X_i}$.
As a consequence, if $\mathcal{P}$ is a $G$-local property, then it is also a local property.

\begin{exem}
A consequence of Kiehl's theorem \cite[9.4.3]{BGR} is 
that the class of Zariski-closed subsets of affinoid spaces defines a class which is $G$-local.
\end{exem}

\begin{defi}
Let $X$ be a good $k$-analytic space. A wide covering of 
$X$ is a covering $\{X_i\}$ such that the 
$X_i's$ are affinoid domains of $X$ and 
$\{ \overset{\circ}{X_i} \}$ is also a covering of $X$.
\end{defi}
\begin{prop}
\label{local}
Let $X$  be a strictly $k$-affinoid space, and $S$ a subset of 
$X$. The following assertions are equivalent:
\begin{enumerate}
\item 
$S$ is an overconvergent subanalytic subset of $X$.
\item For all wide covering $\{X_i\}$ of $X$, 
$X_i\cap S$ is an overconvergent subanalytic subset of 
$X_i$.
\item 
There exists a wide covering 
$\{X_i\}$ of $X$ such that 
$X_i \cap S$ is overconvergent subanalytic in $X_i$ for all $i$.
\item 
For all $x\in X$ there exists an affinoid neighbourhood $V$ of $x$ such that 
$V \cap S$ is overconvergent subanalytic in $V$. 

\end{enumerate} 
The property (4) implies that the class of overconvergent subanalytic subsets is local in the sense of definition \ref{def_local}.
\end{prop}

\begin{proof}
$(1) \Rightarrow (2)$ is obvious and is a consequence of lemma \ref{lemme_inverse}.\par
$(2) \Rightarrow (3)$ is clear. \par
$(3) \Leftrightarrow (4)$ is also clear.\par
$(4) \Rightarrow (1)$ follows from the analogous statement for overconvergent 
constructible subsets (corollary \ref{corlocalover}) and theorem \ref{theoequiv}.

\end{proof}

\begin{defi}
\label{def_locos}
Let $X$ be a good strictly $k$-analytic space. A subset
$S \subset X$ is called overconvergent subanalytic  if
for all $x\in X$ there exists $V$ a strictly affinoid neighbourhood of $x$ 
such that $S\cap V$ is overconvergent subanalytic in $V$ (according to definition \ref{defi_suba}).
\end{defi}
According to the last proposition, when $X$ is a $k$-affinoid space, this 
definition \ref{def_locos} is compatible with the previous one (definition \ref{defi_suba}).\par

\begin{defi}
Let $X$ be a good strictly $k$-analytic space. 
A subset $S$ of $X$ is called locally semianalytic if for all 
$x \in X$ there exists $V$ some strictly affinoid neighbourhood of $x$ such that  
$V \cap S$ is semianalytic in $V$.
\end{defi}

\begin{cor}
\label{cor_loc_ber}
Let $X$ be a good strictly $k$-analytic space.
The class of locally semianalytic subsets of $X$ is contained in the class of 
overconvergent constructible subsets of $X$.
\end{cor}

\begin{cor}
\label{cor_eqber}
Let $X$ be a strictly $k$-affinoid space and let 
$S\subset X$ be a subset of $X$. Then $S$ is 
an overconvergent subanalytic subset of $X$ if and only if 
there exist $r>1$, an integer $n$, 
and 
$T \subseteq X \times \Br^n$ a locally semianalytic subset, such 
that $S = \pi ( T \cap (X\times \E))$.
\end{cor}
\begin{proof}
The first implication is true because a 
 semianalytic subset of $X\times \Br^n$ is in particular 
a locally semianalytic subset of 
$X \times \Br^n$. \par
Conversely, if 
$S = \pi ( T \cap ( X\times \B^n ))$ where 
$T$ is a locally semianalytic subset of 
$X \times \Br^n$, then according to corollary \ref{cor_loc_ber}, 
$T$ is overconvergent subanalytic in 
$X \times \Br^n$, so according to 
corollary \ref{projection}, $\pi(T \cap (X \times \B^n))$ is also overconvergent subanalytic.
\end{proof}

\begin{lemme}
\label{lemme_iminv}
 Let $\varphi : Y \to X$ be a morphism of good strictly $k$-analytic spaces, and 
$S \subseteq X$ be a locally semianalytic subset of $X$. Then 
$\varphi^{-1}(S)$ is a locally semianalytic subset of $Y$. 
\end{lemme}
\begin{proof}
 Let $y\in Y$, $x=\varphi(y)$. There exists $V$ an affinoid neighbourhood
of $x$ such that 
$V \cap S$ is  semianalytic in $V$. 
Let $W$ be an affinoid neighbourhood of $y$ in $\varphi^{-1}(V)$. Then 
$W\cap \varphi^{-1} (S)$ is  semianalytic in $W$.
\end{proof}

If $\varphi : Y \to X$ is a morphism of $k$-analytic spaces, one can define the relative interior of $\varphi$, denoted 
by $Int(Y /X)$ which is a subset of $Y$. 
We refer to \cite[2.5.7]{Berko90} for the definition. 
The complementary set of $Int(Y/X)$ in $Y$ is called the relative boundary of $\varphi$ and denoted by $\partial(Y/X)$.
For these sets, the non-rigid points are essential.
For instance, if $\varphi : \B \to \mathcal{M}(k)$ is the structural morphism,  
$\partial( \B / \mathcal{M}(k))$ is simply the Gauss point.

\begin{theo}
\label{interieur}
Let $\varphi : Y \to X$ be a morphism of strictly $k$-affinoid spaces, and 
$U$ an affinoid domain of $Y$ such that 
$U \subseteq Int (Y / X)$. If $S$ is an overconvergent subanalytic subset of $Y$ then  
$\varphi(U\cap S)$ is an overconvergent subanalytic subset of $X$.
\end{theo}
\begin{proof}
According to \cite[Prop 2.5.9]{Berko90}  there exist
$\underline{r} > \underline{s} >0$ and 
$\mathcal{A}\{\underline{r}^{-1} T \} \to \mathcal{B}$ an admissible epimorphism which 
hence identifies $Y$ with a Zariski closed subset of $ X\times \B_{\underline{r}}$, such that 
under this identification, 
$U \subseteq X \times \B_{\underline{s}}$. 
We can assume that 
$\underline{s} \in \val^n$.
If we denote by $\Gamma(\varphi)$ the graph of $\varphi$, 
this induces a Zariski closed embedding of 
$Y \simeq \Gamma(\varphi) \xrightarrow{i}  X\times \B_{\underline{r}}$.
Now since $S$ is an overconvergent subanalytic subset of $Y$, according to lemma \ref{immersion}, $i(S)$ is an overconvergent subanalytic subset of 
$X \times \Bur $.
Finally, $U$ is a  semianalytic subset of $Y$ (because of Gerritzen-Grauert theorem), so $i(U)$ is also  semianalytic in $X \times \Bur$, and by assumption, 
$i(U) \subseteq X \times \B_{\underline{s}}$, so 
$i(U\cap S) \subseteq  X \times \Bur$ is an overconvergent constructible subset of
$X \times \Bur$, and according to corollary \ref{projection}, $\pi ( i(U\cap S) )$ is an overconvergent subanalytic subset of $X$. But this set is precisely $\varphi(U\cap S)$.
\end{proof}

As in algebraic geometry, the notion of a proper morphism of $k$-analytic spaces is a little subtle. 
If $\varphi : Y \to X$ is a morphism of $k$-analytic spaces, let us denote by $|Y| \to |X|$ the associated map of topological spaces.
Then $\varphi $ is said to be compact \cite[p. 50]{Berko90} if the map $|Y| \to |X|$ is proper (in the topological sense).
Finally, $\varphi$ is said to be proper \cite[p. 50]{Berko90} if $\varphi$ is compact and $\partial(Y/X)= \emptyset$.

\begin{prop}
\label{propstabbound}
Let $\varphi : Y \to X$ be a morphism of good strictly $k$-analytic spaces, 
$S$ an overconvergent subanalytic subset of $Y$ such that the map of topological spaces
$\overline{S} \to |X|$ is proper and 
$\overline{S} \subseteq \text{Int} (Y / X) $. Then $\varphi(S)$ is an overconvergent subanalytic subset of $X$.

\end{prop}
\begin{proof}
If $X'$ is an affinoid domain of $X$ and if we consider the cartesian diagram :
\[ \xymatrix{ 
S\subseteq Y \ar[r]^{\varphi} & X \\
S' \subseteq  Y' \ar[u]^{\psi'} \ar[r]^{\varphi'} & X' \ar[u]^{\psi} 
}\]
then $\psi'^{-1}(\overline{S})$ is closed in $Y'$ and contains $\psi'^{-1}(S)=S'$ so 
$S' \subseteq \overline{S'} \subseteq \psi'^{-1} (\overline{S})$, and since 
properness is stable under base change, 
$\psi'^{-1}(\overline{S}) \to |X'|$ is proper, and since $\overline{S'}$ is closed, 
$\overline{S'} \to |X'|$ is proper. Moreover, 
$\psi'^{-1} (\text{Int} (Y / X)) \subseteq \text{Int} (Y'/ X') $ ( \cite[3.1.3 (iii)]{Berko90} ) so 
$\overline{S'} \subseteq \psi'^{-1}(\overline{S}) \subseteq \text{Int} (Y' / X' )$. 
So $S'$ and $\varphi'$ fulfil the hypotheses of the 
proposition. Hence, since the property we want to check is local on $X$, 
we can assume that $X$ is a $k$-affinoid space, hence that 
$\overline{S}$ is compact.\par 
Now for every $y\in \overline{S}$ we can find 
an affinoid neighbourhood $U$ such that 
$U \subseteq \text{Int}(Y/X)$, because $\text{Int}(Y/X)$ is open \cite[2.5.7]{Berko90}. 
Then, $\varphi(U \cap S)$ is an overconvergent subanalytic 
subset of $X$ according to theorem \ref{interieur}. Since 
$\overline{S}$ is compact we can extract from this a finite covering of $\overline{S}$, 
which finishes to prove that 
$\varphi(S)$ is overconvergent subanalytic.
\end{proof}

\begin{cor}
Let $\varphi : Y \to X$ be a proper morphism of good strictly $k$-analytic spaces.
Let $S$ be an overconvergent subanalytic subset of $Y$. 
Then $\varphi(S)$ is an overconvergent subanalytic subset of $X$.
\end{cor}

\begin{defi}
A morphism $\varphi : Y \to X$ of good $k$-analytic spaces is locally extendible without boundary if, 
for all $y\in Y$, there exists an affinoid neighbourhood $U$ of $y$, 
$Y'$ a $k$-affinoid space that contains $U$ as an affinoid domain, and 
$\psi : Y' \to X$ that extends 
$\varphi_{|U}$, such that 
$U \subseteq \text{Int} (Y'/X)$.
\end{defi}
Remark that using again \cite[3.1.3 (iii)]{Berko90}, this property is stable under base change.

\begin{prop}
Let $\varphi : Y \to X$ be a compact morphism of 
good strictly $k$-analytic spaces which is locally extendible without boundary. Then $\varphi(Y)$ is an overconvergent subanalytic subset of $X$.
\end{prop}
\begin{proof}
We can  assume that $X$ is a $k$-affinoid space, so $Y$ is compact. 
Then for all $y \in Y$ we can find an affinoid neighbourhood $U$ of $y$ and $Y'$ 
a $k$-affinoid space that contains $U$, and 
$\psi : Y' \to X$ that extends $\varphi_{|U}$, such that 
$U \subseteq \text{Int} (Y'/X)$. 
Then, according to theorem \ref{interieur}, $\varphi(U)$ is 
an overconvergent subanalytic subset of $X$ (take $S=Y'$). 
Hence by compactness of $Y$, $\varphi(Y)$ is overconvergent subanalytic.
\end{proof}

\subsection{The non strict case}
\label{nonstrict}
In this section, $k$ will be an arbitrary non-Archimedean field (possibly 
trivially valued). \par
One of the advantages of Berkovich's approach is the possibility to use 
arbitrary $\lambda \in \R_+$ to define inequalities. It is then natural to  give the 
following definitions: 
\begin{defi}
\label{defisanonstrict}
Let $\A$ be a $k$-affinoid algebra, and let us set $X = \affin{A}$. 
\begin{enumerate}
\item 
A subset $S\subset X$ is called non-strictly semianalytic if it is a boolean combination of 
subsets 
\[ \{ x\in X \st \ |f(x)| \leq \lambda |g(x)| \}\]
where $f,g \in \A$ and $\lambda \in \R_+$.
\item 
A subset $S\subset X$ is called non-strictly 
overconvergent subanalytic if there exist an integer 
$n\in \N$, a real number $r>1$, 
a non-strictly semianalytic set 
$T \subset X \times \B_r^n$ such that 
$S=\pi(T \cap (X \times \B^n))$ where 
$\pi : X\times \B_r^n \to X$ is the first projection.
\end{enumerate}
\end{defi}
\begin{rem}
\label{remcomparnonstrict}
Let 
$X$ be a strictly $k$-affinoid space and let 
$S \subset X$. The following implication holds:
\[ S \ \text{is semianalytic} \Rightarrow 
S  \ \text{is non-strictly semianalytic}.\]
However, if 
$\val \subsetneq \R_+^*$, the converse implication is false.
Indeed, let $r\in ]0,1[$ such that $r\notin \val$, let 
$X = \B^1= \mathcal{M}(k\{T\})$ and let 
$S = \{ x\in \B \st |T(x)| =r \}$.
By definition, 
$S$ is a non-strictly semianalytic set of $\B^1$, but we claim 
that it is not semianalytic. Indeed, we will see in  \ref{proprigber} 
that semianalytic sets are entirely determined by their rigid points, that 
is to say, if $S_1$ and $S_2$ are semianalytic subsets of $X$, then, 
$S_1 = S_2$ if and only if 
$S_1 \cap X_{rig} = S_2 \cap X_{rig}$. Since in our example, 
$S \cap X_{rig} = \emptyset$, if $S$ was semianalytic, it would then be empty, 
but $S$ is not empty. Actually 
$S = \{ \eta_r\}$.
\end{rem}

\begin{defi}
Let $X$ be a $k$-affinoid space. 
Let $(X,S)$ be a $k$-germ, 
$f,g\in \A$, $0<s<r$ where $r,s\in \R$, and 
\[ Y =  \mathcal{M}( \A \{r^{-1}t \}/ (f-tg) ) \xrightarrow[]{\varphi} X \]
and $T = \varphi^{-1}(S) \cap R \cap \{y\in Y \st |f(y)| \leq s|g(y)| \neq 0\}$ 
where $R$ is a non-strictly 
 semianalytic subset of $Y$. 
Then we say that 
$(Y,T) \xrightarrow[]{\varphi} (X,R)$ is a 
non-strictly elementary constructible datum.
\end{defi}

The only difference with definition \ref{def_dce} is that 
we do not assume any more that 
$s\in \val$, and that $R$ is allowed to be non-strictly semianalytic, 
that is to say defined with inequalities involving 
some arbitrary $\lambda \in \R$. \par 
Then we mimic definition \ref{defidc}, and say that a 
non-strictly constructible datum 
\dc \ is a composite 
$\varphi = \varphi_1 \circ \ldots \circ \varphi_n$ where 
each $\varphi_i$ is a non-strictly elementary constructible datum. 
Finally, if 
$(X_i,S_i) \stackrel{\varphi_i}{\dashrightarrow} X, i=1\ldots n$ are 
$n$ non-strictly constructible data, we say that 
$S := \cup_{i=1}^n \varphi_i(S_i)$ is a non-strictly overconvergent constructible set. \par 
We claim that all results we have proven in this section 
for overconvergent subanalytic (resp. constructible) sets remain valid  
for non-strictly overconvergent subanalytic (resp. constructible) sets. For instance:
\begin{theo}
Let $X$ be a $k$-affinoid space.
$S\subset X$ is non-strictly overconvergent subanalytic if and only if it is 
non-strictly overconvergent constructible.
\end{theo}
In this context, we want to stress that for instance propositions 
\ref{local}, \ref{propstabbound} also remain true.

\section{Study of various classes}
\label{section2}
\subsection{Many families}
In this section $X = \affin{A}$ will be a strictly $k$-affinoid space.
The aim of this section is to first recall the definitions of the various classes of 
\emph{rigid/locally/strongly/D-semianalytic/subanalytic} subsets of $X$ that are defined  
in \cite{Sch_sub}.\par
We now give the following definitions. A subset $S \subseteq X$ is called :
\begin{enumerate}
 \item[(a)]  semianalytic if it is a boolean combination of 
subsets of the form 
$\{ x \in X \ \big| \  |f(x)| \leq |g(x)| \} $, with $f,g \in \mathcal{A}$.
\item[(b)] Locally semianalytic, if for all $x \in X$ there 
exists an affinoid neighbourhood $V$ of $x$ such that 
$S\cap V$ is 
semianalytic in $V$.
\item[(c)]
Rigid-semianalytic if there is a finite affinoid covering\footnote{If $X$ is an affinoid space we say that  
$\{X_i\}_{i=1}^n$ is a finite affinoid covering if for all $i$ $X_i$ is an affinoid domain of $X$ and $X=\cup_{i=1}^nX_i$.}
$\{X_i\}_{i=1}^n$ such that 
$S \cap X_i$ is  semianalytic in $X_i$ for all $i$.
\item[(d)] Overconvergent subanalytic has been defined in definition \ref{defi_suba}. 
As we proved 
in the previous section, this corresponds also to overconvergent constructible subsets. 
Moreover, our definition of overconvergent subanalytic subset is the same 
as the definition of globally strongly subanalytic of \cite[1.3.8.1]{Sch_sub}. 
In \cite{Sch_sub} it is proven and it is correct that this is equivalent to the class of 
 globally strongly \textbf{D}-semianalytic subsets \cite[1.3.2]{Sch_sub}.
\item [(e)]$G$-overconvergent subanalytic if there exists a 
finite affinoid covering $\{X_i\}$ of $X$ such that 
$S \cap X_i$ is overconvergent constructible in $X_i$ for all $i$. This corresponds to the 
notion of \emph{strongly} \textbf{D} -\emph{semianalytic} subset in
\cite[1.3.7.1]{Sch_sub}.
\item[(f)]
Strongly subanalytic if there exist an integer $n$,  a real number 
$r>1$, a subset 
$T \subseteq X \times \Br^n$ which is rigid-semianalytic, such that 
$S = \pi ( T \cap (X\times \B^n))$. This is definition 
\cite[1.3.8.1]{Sch_sub}, and we will give an equivalent definition in proposition \ref{eqdef}.
\item[(g)]
Locally strongly subanalytic if there exists a finite affinoid covering 
$\{ X_i\}$ of $X$ such that 
$S\cap X_i$ is strongly subanalytic in $X_i$ for all $i$. This is definition 
\cite[1.3.8.2]{Sch_sub}.
\end{enumerate}

In \cite{Sch_sub} it is stated that 
(d),(e),(f) and (g) are equivalent (equivalence of $(e),(f),(g)$ is stated in
\cite[Prop 4.2]{Sch_sub}, and the equivalence of $(d)$ and $(f)$ is stated in
\cite[Th 5.2]{Sch_sub} ). These results rest on \cite[lemma 4.1]{Sch_sub} which is false, and 
we will show indeed that $(d), (e)$ and $(f)$ correspond in general to three different classes. 
More precisely 
the aim of this section is to show that these classes satisfy the following relations:

\begin{figure}[H]
\small
\caption{The hierarchy}
\label{inclusions}
\centering
\[\begin{array}{ccccccccccc}
                  &	      &            &           &                      &                               & rigid-  &                               &         &             &     \\
                  &	      &            &           &                      &                               & semianalytic        &                               &         &             &     \\                  
 Locally         &           &           &         &   G-            & \rotatebox{45}{$\supsetneqq^{3}$} &                 &\rotatebox{-45}{$\supsetneqq^6$} &         &             &     \\
        strongly &\supseteq           &Strongly    &\supsetneqq^1&       overconvergent &                             &\rotatebox{90}{$\nsubseteq$}^4  \ \   \rotatebox{-90}{$\nsubseteq$}_5           &         &Locally   & \supsetneqq^8 &semi-  \\
  subanalytic    & \boxed{?}         & subanalytic&           & subanalytic          & \rotatebox{-45}{$\supsetneqq_2$}&                 &\rotatebox{45}{$\supsetneqq_7$}  &   semianalytic      &   &   analytic      \\
                            &  &           &           &                      &                               &overconvergent &                               &         &             &  \\
& & & & & & subanalytic & & & &                            
  
\end{array} 
\]
In this figure, 
Class A $\supsetneqq $ Class B, means that the class A \emph{properly} contains the class B, \\
and Class A $\nsupseteq$ Class B means that the Class A does not contain
the class B.
\end{figure}
In this diagram, all the inclusions are clear from the definitions, 
except the inclusion $7$ which states that the class of overconvergent subanalytic subsets contains 
the class of locally semianalytic subsets. 
But this is precisely the content of corollary \ref{local}. 
In comparison 
with what was stated in \cite{Sch_sub}, the most striking inequality is probably $\nsubseteq^5$  
which asserts that rigid-semianalytic subsets are not necessarily overconvergent subanalytic subsets whereas according to 
\cite[Th 5.2]{Sch_sub}, they should be overconvergent subanalytic. 
In other words, when you project overconvergent  
semianalytic subsets, you obtain a class which is not $G$-local (but however local for 
the Berkovich topology).\par 
In this section we will show that the inclusions (1)-(8) in figure \ref{inclusions} are all proper in general (in the next section we will explain that if $X$ is regular of dimension 2, overconvergent subanalytic subsets correspond to locally semianalytic subsets). 
We do not know if the inclusion on the left 
\[ locally \ strongly \ subanalytic \supseteq strongly \ subanalytic \] 
is proper. 

\subsection{Rigid-semianalytic subsets are not necessarily overconvergent subanalytic}
Here we prove the inequality $\nsubseteq^5$.
\begin{lemme}
\label{lemme_semialg}
Let $\eta \in X$ such that 
$\mathcal{O}_{X,\eta}$ is a field, $S\subset X$ a  semianalytic subset.
If $\eta \in \overline{S}$, then $\overset{\circ}{S}$ is non empty.
\end{lemme}

\begin{proof}
Since 
\[\bigcup_{i=1}^n \overline{S_i} = \overline{\bigcup_{i=1}^n S_i }\] 
we can assume that $S$ is a basic semianalytic subset, i.e is of the form:
\[S = \left( \bigcap_{i=1}^m \{x\in X \ \big| \  |f_i(x)|\leq |g_i(x)| \} \right)  
\bigcap \left( \bigcap_{j=1}^n \{x \in X \ \big| \  |F_j(x)| < |G_j(x)| \} \right).\]
We use the following decomposition
\[\{x\in X \ \big| \  |f_i(x)|\leq |g_i(x) | \} =
\{x\in X \ \big| \  f_i(x)=g_i(x)=0 \} \cup \{x\in X  \ \big| \  \ |f_i(x)|\leq |g_i(x)| \neq 0 \}\]
and using again that the adherence is stable under finite union, we can assume that 
$\eta \in \overline{S}$ and that $S$ is of the form: 
\[S= \bigcap_{i=1}^l \{x \in X \ \big| \ h_i(x)=0 \} 
\bigcap_{j=1}^m \{x\in X \ \big| \   \ |f_j(x)|\leq |g_j(x)|\neq 0 \} \bigcap_{k=1}^n \{x\in X \ \big| \  |F_k(x)|<|G_k(x)| \}.\]
Since the subsets $\{x\in X \ \big| \  h_i(x)=0\}$ are closed, contain $S$ and  
$\eta \in \overline{S}$, it follows that $h_i(\eta)=0$.\par
Since $\mathcal{O}_{X,\eta}$ is a field we can find an affinoid neighbourhood $V$ 
of $\eta$ such that 
${h_i}_{|V}=0$ for all $i$.
Hence $V\cap S \neq \emptyset$ (because $\eta \in \overline{S}$) and we can 
remove the $h_i's$, and assume that
\[V\cap S = \bigcap_{j=1}^m \{x \in V \ \big| \ |f_j(x)|\leq |g_j(x)|\neq 0 \}  
\bigcap_{k=1}^n \{x\in V \ \big| \  |F_k(x)|<|G_k(x)| \}.\]
This defines a strictly $k$-analytic domain of $X$, which is non empty, so its interior is also non empty, for instance its interior contains some rigid points.
\end{proof}

\begin{lemme}
\label{lemme_type2}
Let $\eta \in X$ and let us assume that $\mathcal{O}_{X, \eta}$ is a field. 
Let \dce \ be an elementary constructible datum with 
$Y =\mathcal{M}( \mathcal{A}\{r^{-1}t \}/(f-tg) )$ where
$T=\varphi^{-1} (S) \cap \{y\in R \ \big| \  |f(y)|\leq s|g(y)|\neq 0  \}$ 
with $0<s<r$, $s\in \val$ and $R$ a  semianalytic subset of $Y$. 
Let us assume that $\eta \in \overline{\varphi(T)}$. 
Then 
\begin{enumerate}[(a)]
\item $g(\eta) \neq 0$.

\item $|f(\eta) | \leq s |g(\eta)|$. 
 
\item There exists a neighbourhood $U$ of $\eta$ 
such that  
$\varphi^{-1}(U) \xrightarrow{\varphi_{| \varphi^{-1}(U) } } U$ is an isomorphism.
If we denote by $\eta '$ the only point of $\varphi^{-1}(U)$ such that 
$\varphi(\eta') = \eta$, then   
$\eta' \in \overline{T}$ and $\mathcal{O}_{Y,\eta'}$ is a field. 
\end{enumerate}

\end{lemme}

\begin{proof}
\noindent
 \begin{enumerate}[(a)]
 \item
If we had $g(\eta) = 0$, since $\mathcal{O}_{X,\eta}$ 
is a field, there would exist an affinoid neighbourhood of $\eta$, $V$, such that
$g_{|V} =0$. Since for $p\in T$,  
$g(\varphi(p))\neq 0$ we should have $\varphi(T) \cap V = \emptyset$ which is impossible since 
$\eta \in \overline{\varphi(T)}$. \\
\item
The subset $\{x\in X \ \big| \  |f(x)| \leq s |g(x)| \} $ is a closed subset of  $X$ 
which contains $\varphi(T)$, hence by assumption also $\eta$. \\
\item
If we set $U= \{y\in Y \ \big| \  g(y)\neq 0 \}$, 
$\varphi_{|U}$ identifies through an isomorphism $U$ with  
$\varphi(U) = \{x\in X \ \big| \ |f(x)| \leq r |g(x)|\neq 0 \}$ 
which is an analytic domain of $X$, and 
a neighbourhood of $\eta$ according to the two preceding points.
So $\eta \in \varphi(U)$, let us say $\eta = \varphi(\eta')$ with $\eta' \in U$.
Now,  
$\mathcal{O}_{Y, \eta'} \simeq \mathcal{O}_{X,\eta} $ is a field and 
$\eta' \in \overline{T}$. 
\end{enumerate}
\end{proof}

\begin{cor}
\label{lemme_interieur}
Let $\eta \in X$ such that $\mathcal{O}_{X, \eta}$
is a field, and let $U$ be an overconvergent subanalytic subset of $X$.
If $\eta \in \overline{U}$, then $\overset{\circ}{U} \neq \emptyset$. 
\end{cor}

\begin{proof}
First, according to theorem \ref{theo_eq}, we can assume that $U$ is 
an overconvergent constructible subset. 
Then, using similar arguments as in the beginning of lemma \ref{lemme_semialg}, 
we can assume that $U = \varphi(T)$ where 
$(Y,T) \stackrel{\varphi}{\dashrightarrow} X$ 
 is a constructible datum.
Hence $T$ is a  semianalytic subset of $Y$.
A repeated use of lemma \ref{lemme_type2} 
allows us to say that there exists an open neighbourhood $U$ of $\eta$,  
such that 
$\varphi^{-1} (U) \xrightarrow{\varphi_{| \varphi^{-1}(U)}} U$ 
is an isomorphism. 
Thanks to lemma \ref{lemme_type2} again, we can 
introduce $\eta '$, the only point of $\varphi^{-1}(U)$ such that 
$\varphi(\eta') = \eta$, and assert that
$\mathcal{O}_{Y,\eta'} $ is a field and that $\eta' \in \overline{T}$.
Now if we consider $V$ a strictly affinoid neighbourhood of $\eta'$ 
contained in $\varphi^{-1}(U)$, 
it is true that $\eta' \in \overline{T\cap V}$ (the adherence is here taken in $V$).
Now, $T\cap V$ is a  semianalytic subset of $V$ so 
according to lemma \ref{lemme_semialg}, $T\cap V$ has non empty interior
in $V$. 
We can then deduce that $T$ has non empty interior in $X$ 
whence $\varphi(T)$ has also non-empty interior. 
\end{proof}
Let $f = \sum_{n\in \N} a_nT^n$ be a series and $r\in \R_+^*$. 
We will say that the radius of convergence of $f$ is exactly $r$ when 
$|a_n|r^n \xrightarrow[n \to \infty]{} 0$ and $r$ is maximum for this property.
\begin{prop}
\label{CE1}
Let $X =\B^2= \mathcal{M}( k\{T_1,T_2 \} )$ be the closed bidisc, 
and let $0<r<1$ with $r \in |k^*|$, say 
$r= |\varepsilon|$ for some $\varepsilon \in k$,  
and let $f\in k\{r^{-1}u\}$ be some function whose radius of convergence is exactly 
$r$, and $\|f\| < 1$. We then define 
\[S = \{ x\in X \ \big| \  |T_1(x)| < r \ \text{and} \ T_2(x) = f(T_1(x)) \}.\]
Then $S$ is rigid-semianalytic but 
not overconvergent subanalytic. As a consequence, the class of 
overconvergent subanalytic subsets is not $G$-local. 
\end{prop}
\begin{proof}
In more concrete terms, 
$S$ is the set of points of the \emph{curve} whose equation is  $T_2=f(T_1)$, restricted 
to the subset $\{ |T_1|<r \}$.
Let us consider
$$\begin{array}{rcl}
\B & \xrightarrow{\psi} & X \\
u & \mapsto & (\varepsilon u , f(\varepsilon u ) ) 
\end{array}$$
and let us set $\eta = \psi(g)$ where $g$ is the Gauss point of 
$\B$.
Then $S \subseteq \psi(\B)$ and $\eta \in \overline{S}$.
According to \cite[proposition 4.4.6]{Duc_flatness}
$\mathcal{O}_{X,\eta}$ is a field.  
Furthermore $\overset{\circ}{S} = \emptyset$ because
$S \subseteq Z:=\{ x\in \B_{(r,1)} \ \big| \ T_2(x) = f(T_1(x)) \}$, 
which is a Zariski closed subset of dimension $1$ of 
$\B_{(r,1)}$, which itself is of pure dimension $2$, so 
$Z$ is nowhere dense in $\B_{(r,1)}$ \cite[2.3.7]{Berko90}. 
Hence according to corollary \ref{lemme_interieur}, $S$ 
is not overconvergent subanalytic. However, if we consider the covering of $X$ given by  
$X_1 = \{ x  \in X \ \big| \  |T_1(x)| \leq r \}$, 
$X_2 = \{x \ \in X \ \big| \  |T_1(x)| \geq r \}$, then 
$S \cap X_1$ is indeed  semianalytic in $X_1$ and 
$S\cap X_2 = \emptyset$, so $S$ is rigid-semianalytic. \par
Now since the class of overconvergent subanalytic subsets 
contains the class of  semianalytic subsets, if the class of 
overconvergent subanalytic subsets was $G$-local, 
it should contain the class of rigid-semianalytic subsets, but we have shown that this is false. 
Hence the class of overconvergent subanalytic subsets is not $G$-local.
\end{proof}

\begin{rem}
 Actually, this example gives directly a counterexample to 
\cite[lemma 4.1]{Sch_sub} which in our feeling is the source of mistakes in \cite{Sch_sub}.
\end{rem}
As a corollary of this we obtain: 
\begin{prop}
\label{CE8}
\label{cex_SA}
Let $0<s<r<1$ with $s \in \val$, $f\in k\{r^{-1}u\}$ whose radius of convergence is exactly $r$ 
such that $\|f\|<1$, and let us set 
$\B^2 = \mathcal{M}( k\{T_1,T_2\})$. Define :
\[ S = \{x\in \B^2 \ \big| \ |T_1(x)| \leq s \ \and \ T_2(x) = f(T_1(x)) \}.\]
Then $S$ is a locally semianalytic subset of $\B^2$ which is not a  semianalytic 
subset of $\B^2$.
\end{prop}

\begin{proof}
If $S$ was a  semianalytic subset of $\B^2$, we could find $T \subseteq S$ 
which contains infinitely many points of $S$ such that 
$T$ is a basic semianalytic subset, and even, a finite intersection of sets of the form 
$\{x\in \B^2 \ \big| \ |g_1(x)| < |g_2(x)| \}$, $\{x\in \B^2 \ \big| \ |g_1(x)| \leq |g_2(x)| \neq 0 \}$ and 
$\{x \in \B^2 \ \big| \ h(x)=0\}$. 
Since an intersection of the two first kind of sets is a strictly analytic domain, and 
$T \subseteq S$, and $\overset{\circ}{S} = \emptyset$, in this intersection, there must be a non-trivial set of the form
$\{x\in \B^2 \ \big| \ h(x)=0 \}$. 
Now, let us consider in 
$\B_{(r,1)} = \mathcal{M}(k\{r^{-1}T_1,T_2\})$ the Zariski-closed subset 
$Z= V(T_2-f(T_1),h)$. By assumption, it is infinite. 
Moreover, since $\|f\| < 1$, $T_2-f(T_1)$ is irreducible (see the lemma above) in 
$\mathcal{M} ( k\{r^{-1}T_1, T_2\} )$, 
so for reasons of dimension, 
in $\mathcal{M} ( k\{r^{-1}T_1, T_2\} )$, 
$V(T_2-f(T_1) ) \subseteq V(h)$. 
But now if we introduce (as in the preceding proof) 
\[\begin{array}{rcl}
\Br & \xrightarrow{\psi} & \B^2 \\
u & \mapsto & ( u , f(u ) ) 
\end{array}\]
and $\eta = \psi(g)$ where $g$ is the Gauss point of 
$\Br$, then 
$\eta \in V(h)$ (where we see now $V(h)$ as a Zariski closed subset of $\B^2$), 
$\mathcal{O}_{\B^2,\eta}$ is a field, but 
$\overset{\circ}{V(h)}=\emptyset$, and since 
$V(h)$ is a  semianalytic (so overconvergent subanalytic) subset of 
$\B_{(r,1)}$, this contradicts lemma \ref{lemme_semialg}. \par
Let us now show that $S$ is a locally semianalytic subset of $\B^2$. 
Indeed, 
take $0<s<t<r$ with $t,r\in \val$, and consider  
$X_1 = \{ x\in \B^2 \ \big| \  |T_1(x)| \leq r \}$ and 
$X_2 = \{ x\in \B^2 \ \big| \  |T_1(x)| \geq t \}$. 
They define a wide covering of $\B^2$ and $X_1\cap S$ (resp. $X_2\cap S$) is  semianalytic in $X_1$ (resp. $X_2$), so $S$ is well locally semianalytic in $\B^2$.
\end{proof}

We have implicitly used:
\begin{lemme}
If $f\in k\{r^{-1}x \}$ and $\|f\| \leq 1$, then
$F(x,y) := y - f(x)$ is irreducible in 
$k\{r^{-1}x,y\}$.
\end{lemme}

\begin{proof}
As we have already seen, $V(f)$ is isomorphic to $\B_r$, so is irreducible.
\begin{comment}
If we consider a decomposition 
\begin{equation}
\label{irred}
F(x,y) = G(x,y) H(x,y)
\end{equation}
since 
$\|f\| \leq 1$ we can replace $y$ by $f(x)$ and \eqref{irred} becomes:
\[0 = F(x,f(x) ) = G(x,f(x))H(x,f(x)) .\]
So one of the two factors $G(x,f(x))$ or $H(x,f(x))$ must be $0$.
For instance say $G(x,f(x))=0$. Since $y-f(x)$ is $y$-distinguished of order $1$, if 
we perform a euclidean division of $G$ by $y-f(x)$ we obtain: 
\[G(x,y) = (y-f(x))q +R(x)\] 
with $R(x) \in k\{r^{-1}x\}$. Since 
$G(x,f(x)) =0$, $R(x)=0$.
Hence $y-f(x) = (y-f(x))qH$ so $qH=1$, and $H$ is invertible. 
\end{comment}
\end{proof}

\subsection{The other inequalities}
We will now explain the other proper inclusions appearing in figure 1.
The following proposition will be implicitly used in the rest of this section. In addition, it illustrates that 
the mixture of overconvergence and rigid-semianalytic subsets 
(which is a $G$-local property), is somehow too strong, 
in the sense that in proposition \ref{eqdef} above, the overconvergence condition seems to have disappeared.

\begin{prop}
\label{eqdef}
Let $S \subseteq X$. The following properties are equivalent:
\begin{enumerate}
\item $S$ is strongly subanalytic. 
\item There exist $n\in \mathbb{N}$ and $T \subseteq X \times \B^n$ a rigid-semianalytic subset such that 
\[ S = \pi(T \cap ( X \times ( \overset{\circ}{\B})^n ) )\] 
where 
$\pi : X\times \B^n \to X$ is the natural projection.
\end{enumerate}
\end{prop}

\begin{proof}
Let us show that $(1) \Rightarrow (2)$. 
Let $S$ be a strongly subanalytic subset of $X$, so there exists 
$r>1$, $T\subseteq X \times \Br^n$ a 
rigid-semianalytic subset such that 
$S =\pi ( T \cap ( X \times \B^n ))$. 
Decreasing $r$ if necessary, we can assume that 
$|r|\in \valeur$. 
In fact, using similar arguments as the one given in 
remark \ref{rem_rayon}, we can even assume that 
$r\in |k|$. 
Then if we consider the homothety, which is 
an isomorphism: 
$h : X \times \Br^n \to X \times \B^n$, which can 
be defined as multiplication of each coordinate of 
$\Br^n$ by $\frac{1}{\lambda}$, this gives the 
following commutative diagram:
\[
\xymatrix{
X\times \Br^n \ar[r]^h \ar[rd]^{\pi} & X\times \B^n \ar[d]^{\pi'} \\
                     &    X
}\]
and $S=\pi(T \cap  ( X\times \B^n ) 
=\pi' \left(  h(T) \cap     ( X \times \B_{\frac{1}{r} }^n )     \right)$. 
Now $T' := h(T) \cap (X\times \B_{ \frac{1}{r} }^n )$ is a rigid-semianalytic subset of 
$X\times \B^n$ such that 
$T' \subseteq X \times (\overset{\circ}{\B} ) ^n$ 
and $S = \pi' (T') = \pi'(T' \cap (X \times  ( \overset{\circ}{\B})^n))$. \par
Conversely, let $T \subseteq X \times (\overset{\circ}{\B})^n$ be a rigid-semianalytic subset of $X\times \B^n$ and $S=\pi(T)$. 
For any $r>1$, we can define 
$X_0 = X \times \B^n$, and for $i=1\ldots n$, let
$X_i = \{ (x,t_1, \ldots ,t_n) \in X\times \Br^n\ \big| \ |t_i|\geq 1 \}$. 
So $\{X_i\}_{i\in \{0 \ldots n \}}$ is an admissible covering of 
$X\times \Br^n$. By assumption, $T\cap X_0$ is rigid-semianalytic, and 
$T\cap X_i = \emptyset$ for $i=1\ldots n$. So 
$T$ is rigid semianalytic in 
$X\times \Br^n$, and if we note 
$\pi : X\times \Br^n \to X$, $S = \pi(T)$, so $S$ is strongly subanalytic.
\end{proof}

\begin{prop}
 There exist strongly subanalytic subsets which are not $G$-overconvergent subanalytic.
\end{prop}

\begin{proof}
Let $r>1$,  
$X = \mathcal{M}(k\{x,y,z\}) = \B^3$, 
$Y=\mathcal{M}(k\{x,y,z,t\})$ and
$\pi : Y=\mathcal{M}(k\{x,y,z,t\}) \to X = \mathcal{M}(k\{x,y,z\})$, 
the natural projection. We now choose 
$f\in k\{t\}$ whose radius of convergence is exactly $1$, and 
such that $\|f\| \leq 1$, and 
$T= \{(x,y,z,t) \in Y \ \big| \  |t|<1, x=yt, z=yf(t) \}$.
It is a rigid-semianalytic subset of $Y$, and 
$S = \pi (T)$ is a strongly subanalytic subset of $X$ according to the previous proposition.
Since the family of closed balls with center the origin is 
a fundamental system of neighbourhoods of the origin,
if $S$ was $G$-overconvergent subanalytic, for 
some $1 \geq |\mu| = \varepsilon >0$ small enough, 
$S' :=S \cap \B_{\varepsilon}^3$ would be overconvergent 
subanalytic in $\B_{\varepsilon}^3$. 
We then fix a
$y_0\in k^*$ such that 
$0<|y_0|<\varepsilon$, i.e. 
$\frac{|y_0|}{|\mu|} <1$ and define 
$X' := \{ (x,y,z) \in \B_{\varepsilon}^3 \ \big| \  y=y_0\}$.
Now $X'$ is isomorphic to the 
bidisc 
$\B_{\varepsilon}^2=\{(x,y) \ \big| \ |x| \leq \varepsilon \ \text{and} \ |y| \leq \varepsilon \}$,
and $S'':= S \cap X'$ should be overconvergent constructible in 
$X'$ thanks to lemma \ref{immersion} (2). 
If we make a dilatation of $X'$ by $\frac{1}{\mu}$ it becomes the bidisc of radius $1$:
the new coordinates are $x',z'$ defined by $x=\mu x' $ and $z =  \mu z'$. 
Now, in these new coordinates: 
\[S'' = \{(x',z') \in \B^2  \ \big| \  |x'| < \frac{|y_0|}{|\mu|} \ \text{and} \ 
z' = \frac{y_0}{\mu}f(\frac{x'\mu}{y_0} ) \}\]
should be overconvergent subanalytic in $\B^2$.
If we put $r:= \frac{|y_0|}{|\mu|} <1$ and 
$g(x') = \frac{y_0}{\mu}f(\frac{x'\mu}{y_0} )$, then 
the radius of convergence of $g$ is precisely $r$, 
$\|g\| < 1$ so 
$S'' = \{ (x',z') \in \B^2 \ \big| \  |x'|<r \ \text{and} \ z'=g(x') \}$, $S''$ should be overconvergent subanalytic in $\B^2$, but we proved the negation in proposition \ref{CE1}.
\end{proof}

\begin{prop}
\label{CE:OSnotrigSA}
 There exist overconvergent subanalytic subsets which are not rigid-semianalytic.
\end{prop}
\begin{proof}
Let $1<r=|\lambda|$, and 
$f\in k\{r^{-1}X\}$ whose radius of convergence is exactly $r$, and such that 
$\|f\| < 1$. 
We set 
$X= \B^3 = \mathcal{M} (k\{x,y,z\})$, 
$Y =\mathcal{M} (k\{x,y,z, r^{-1}t\}) $, and 
\[T = \{ (x,y,z,t)\in Y \ \big| \  x=yt , \ z=yf(t) , \ |t| \leq 1 \} \] 
and $S = \pi (T)$, where 
$\pi : \mathcal{M} (k\{x,y,z, r^{-1}t\})  \to \mathcal{M} (k\{x,y,z\}) $ is 
the natural projection.
Then $S$ is overconvergent subanalytic.
If $S$ was rigid-semianalytic, there would exist 
$\mu \in k$, with 
$0<\varepsilon := |\mu| < 1$ such that 
$S' = S \cap \B^3$ is  semianalytic in 
$\B_{\varepsilon}^3$ (we \emph{again} use that if $V$ is an affinoid  
domain of $\B^3$ that contains the origin, then there 
exists $\varepsilon >0$ such that 
$\B_{\varepsilon}^3 \subseteq V$). Let us introduce 
$y_0 \in k^*$ such that 
$0< | y_0| <\frac{\varepsilon}{r}$.
In particular 
$\frac{|y_0|}{\varepsilon} = \left| \frac{y_0}{\mu} \right| < \frac{1}{r}$.
Then $X' = \{ (x,y,z) \in \B_{\varepsilon}^3 \ \big| \  y=y_0\}$ is 
a Zariski-closed subset of $\B_{\varepsilon}^3$, isomorphic to a bidisc $\B^2$.
Now, 
$S'' := S \cap X'$ is defined by 
\[S'' = \{(x,z) \in \B_{\varepsilon}^2 \ \big| \  \left| \frac{x}{y_0} \right| \leq 1 \ \text{and} 
\ z=y_0f( \frac{x}{y_0}) \}.\]
As we said, $X'$ is isomorphic to $\B^2$ with coordinates $(x',z')$ where  
$x= \mu x'$ and $z = \mu z'$. In these new coordinates, 
$S'' = \{(x',z') \in \B^2 \ \big| \ \left| \frac{x'\mu}{y_0}\right| \leq 1 \ \text{and} \ 
z'\mu = y_0f(\frac{x'\mu}{y_0})\}$.
If we define 
$g(x') = \frac{y_0}{\mu} f(\frac{x'\mu}{y_0})$ and 
$s=\frac{|y_0|}{\varepsilon} = \left| \frac{y_0}{\mu}\right| <\frac{1}{r}$,
then $g$ has a radius of convergence which is exactly $\rho$ where 
$s<\rho = \left| \frac{y_0}{\mu} \right| r <1$, and 
$\|g\| < \|f\| < 1$, so 
$S'' = \{ (x',z') \in \B^2 \ \big| \  |x'| \leq s \ \text{and} \ z'=g(x')\}$, should be 
semianalytic, but is not (see proposition \ref{cex_SA}).
\end{proof}

\begin{rem}
\label{rem:Osgood}
The example given in the above Proposition is very closed to the so called Osgood example \cite[Theorem 1]{Osgood}.
This example asserts that that the subset of $\C^3$ parametrized by 
\[ x=u \hspace{3cm} y=uv \hspace{3cm} z=uve^v\]
does not satisfy any relation of the form $F(x,y,z)=0$ where $F$ is a germ of analytic functions around the origin. See also \cite[2.3]{Bie01}. \par 
The non-archimedean analogue of this fact hold in the non-archimedean setting (see the introduction of \cite{LRSurf} for instance).
The example studied in the above proposition amounts to consider the set parametrized by 
\[ x=uv \hspace{3cm} y=v \hspace{3cm} z=vf(v)\]
where $f$ is transcendental. 
Osgood's original argument would have equally worked here, but 
let us stress that our argument is different.
\end{rem}

From this one can deduce:
\begin{cor}
Let $X$ be a strictly $k$-analytic space which contains a closed ball of dimension $\geq 3$.
Then there are overconvergent subanalytic subsets of $X$ which are not rigid-semianalytic. 
In particular, the class of overconvergent subanalytic subsets of $X$ properly 
contains the class of locally semianalytic subsets of $X$.
\end{cor}
In conclusion, in figure \ref{inclusions}, we have shown inequalities 
$1,4,5$ and $8$. Now $2,3,6,7$ are set-theoretical consequences of $4$, $5$ and of the inclusions from the left to the right.

\subsection{Berkovich points versus rigid points}
\label{section2.4}
Let $X= \affin{A}$ be a strictly $k$-affinoid space. 
We denote by $X_{\text{rig}}$ the set of rigid points of $X$.
When one deals with semianalytic or 
overconvergent subanalytic subsets $S$ of 
$X$, one can wonder if things change if we restrict to 
$S_{\text{rig}} = S \cap X_{\text{rig}}$.
Actually the following two propositions show that there is no difference if one works with Berkovich spaces or rigid spaces. \par 
To be precise, let us denote by 
$\mathcal{B}$ be the free boolean algebra whose set of variables 
consists in the set of \emph{formal inequalities} 
$\{|f| \leq |g| \}$, $\{|f| <|g| \}$ and $\{f=0\}$, for $f,g \in \mathcal{A}$. 
We denote by 
$SA_{\text{rig}}$ the class of  semianalytic subsets of 
$X_{\text{rig}}$ and by 
$SA_{\text{Ber}}$ the class of  semianalytic subsets of the Berkovich space $X$. 
Then we define natural maps 
$\alpha : \mathcal{B} \to SA_{\text{Ber}}$ and 
$\beta : \mathcal{B} \to SA_{\text{rig}}$ 
where for instance 
$\alpha ( \{|f| \leq |g| \} ) = \{x\in X \ \big| \ |f(x)| \leq |g(x)| \}$ and 
$\beta ( \{|f| \leq |g| \} )=\{x\in X_{\text{rig}} \ \big| \ |f(x)| \leq |g(x)| \}$. 
In addition we consider the forgetful map 
$\iota : SA_{\text{Ber}} \to SA_{\text{rig}}$: if $S \in SA_{\text{Ber}}$ is a semianalytic set, 
$\iota(S) = S\cap X_{\text{rig}}$. We then obtain the commutative diagram:
\[
\xymatrix{
  \mathcal{B} \ar[r]^{\alpha} \ar[rd]_{\beta} & SA_{\text{Ber}} \ar[d]^{\iota} \\
                                               & SA_{\text{rig}} 
}\]

\begin{prop}
\label{ineq_SA}
The map $\iota$ is bijective.
\end{prop}
\begin{proof}
First, $\iota$ is surjective by definition. \par  
Now if $\iota(S_1) = \iota(S_2)$, we must show that $S_1=S_2$. Considering 
$S_1\setminus S_2$ and $S_2 \setminus S_1$, we can restrict to show that if 
$S \in SA_{\text{Ber}}$ and $\iota(S) = \emptyset$, then $S=\emptyset$.
According to what has been previously done, we can assume 
that $S \in SA_{\text{Ber}}$ is a finite intersection of subsets of the form 
$\{x\in X \ \big| \  |f(x)| \leq |g(x)| \neq 0 \}$, $\{x \in X \ \big| \  |f(x)| <|g(x)| \}$ and 
$\{x\in X \ \big| \  h(x)=0\}$, and that 
$\iota(S) = S \cap X_{\text{rig}} = \emptyset $.
Passing to $Y = \mathcal{M} (\mathcal{A} / \mathcal{I})$ where $\mathcal{I}$ is 
the ideal generated by the functions $h$ appearing in the third case ($h(x) =0$), 
we can assume that $S$ is a finite intersection of subsets of the form: 
$\{x\in X \ \big| \  |f(x)| \leq |g(x)| \neq 0 \}$ and $\{x \in X \ \big| \  |f(x)| <|g(x)| \}$. 
But then it forms a non empty strictly analytic domain of $X$ so 
$S \cap X_{\text{rig}} \neq \emptyset$.
\end{proof}
If we denote by
$CD$ the family of finite subsets of constructible data of $X$, by
$OC$ the family of overconvergent constructible subsets of $X$, 
and $OC_{\text{rig}}$ the family of subsets of $X_{\text{rig}}$ which are the intersection 
of an element of $OC$ with $X_{\text{rig}}$, then we can define as above the following commutative diagram: 
\[
\xymatrix{
 CD \ar[r]^{\alpha} \ar[rd]_{\beta} & OC \ar[d]^{\iota} \\
                                               & OC_{\text{rig}} 
}\]
To be precise, if $\mathcal{D}\in CD$ is the set of the constructible data 
$(X_i,T_i) \stackrel{\varphi_i}{\dashrightarrow} X$, then 
\[\alpha ( \mathcal{D} ) = \bigcup_{i=1}^n \varphi_i(T_i).\]
\begin{prop}
\label{proprigber}
In the above digram, $\iota$ is a bijection.
\end{prop}
\begin{proof}
Since we showed that $OC$ (and $OC_{\text{rig}}$) is stable under complementary, 
we can restrict to show that if 
$S \in OC $ is such that $\iota (S) = S \cap X_{\text{rig}} = \emptyset$, then 
$S = \emptyset$. To show this we can even assume that 
$S = \varphi(T)$, where 
$(Y,T) \stackrel{\varphi}{\dashrightarrow} X$ is a constructible datum.
But, if $T$ is a non empty  semianalytic subset of $Y$, according to proposition 
\ref{ineq_SA}, $T_{\text{rig}} \neq \emptyset$, so since $\varphi$ preserves the rigid points, 
$\varphi (T)_{\text{rig}} = S_{\text{rig}}$ is non empty.

\end{proof}

\section{Overconvergent subanalytic subsets when dim$(X)=2$}
\label{section3}
In this section, $k$ will be non-archimedean algebraically closed field. 
In that case,  a $k$-analytic space $X$  is said to be quasi-smooth \cite[section 5]{Duc_flatness} if for all $x\in X$ the local ring $\mathcal{O}_{X,x}$ is regular.
\subsection{Algebraization of functions}

\begin{prop}\label{topol}
Let $X$, $Y$ be two $k$-affinoid spaces, so that we can consider the cartesian diagram 
\[\xymatrix{
 & X\times Y \ar[dl]_{\pi_1} \ar[dr]^{\pi_2} & \\
X           &          &     Y
}\]
Let $z\in X\times Y$, and let us denote by $z_1= \pi_1(z)$, $z_2=\pi_2(z)$. 
Let us assume that $z_2 \in Y(k)=Y_{\text{rig}}$.
\begin{enumerate}[(a)]
\item 
Let $V$ be an affinoid domain of $X\times Y$ such that 
$z\in V$. 
There exists an affinoid domain $U$ of $X$ (which contains $z_1$) such that
if $W$ is an affinoid neighbourhood  of $z_2$ small enough, 
$V \cap ( X \times W) = U\times W$. 
\item 
Let $\mathcal{V}$ be a neighbourhood of $z$. There exists 
$U$ (resp. $W$) an affinoid neighbourhood of $z_1$ (resp. $z_2$) such that 
$\mathcal{V} \supseteq U\times W$ 
\end{enumerate}
\end{prop}

\begin{proof}

\begin{enumerate}[(a)]
\item 
\cite[2.2]{Sch_plane} 
Let us set $X = \affin{A}$ and $Y=\affin{B}$. 
First, using the Gerritzen-Grauert theorem, we can assume that 
$V$ is a rational domain of $X\times Y$ defined by:
\[V = \{x\in X\times Y \st |f_i(x)| \leq |g(x)|, \ i=1\ldots n \ \and \ |g(x)|\geq r\}\]
where $f_i,g \in \A \widehat{\otimes}_k \mathcal{B}$, and $r>0$. 
Since we assume that $z_2\in Y(k)$, it makes sense to evaluate the functions 
$f_i,g$ at $z_2$, and we will denote by 
${f_i}_{z_2}$, $g_{z_2}$ the corresponding functions, that we see as elements of 
$\mathcal{A}$ and of $\A \widehat{\otimes}_k \mathcal{B}$. 
In addition, since $z_2$ is a rigid point of $Y$, there exists 
an affinoid neighbourhood $T$ of $z_2$ in $Y$ such that
\begin{align}
\label{schout}
\forall i \ \sup_{x\in X\times T} |(f_i - {f_i}_{z_2})(x)| <r \\
\label{schou2} \sup_{x\in X\times T} | (g-g_{z_2})(x)|<r.
\end{align}
Since $g = g_{z_2} +(g-g_{z_2})$, we conclude from \eqref{schou2} that if $x\in X\times T$,
\begin{equation}
\label{Schou1}
|g(x)| \geq r \Leftrightarrow |g_{z_2}(x)| \geq r.
\end{equation}
Then since also 
$g_i = {g_i}_{z_2} +(g_i -{g_i}_{z_2} )$, from \eqref{schout}, \eqref{schou2} and \eqref{Schou1}, we conclude that  
if $x\in X\times T$ 
\[\big( |g(x)| \geq r \ \and \ |f_i(x) \leq |g(x)| \big)  \Leftrightarrow 
\big( |g_{z_2}(x)| \geq r \ \and \ |{f_i}_{z_2}(x)| \leq |g_{z_2}(x)| \big).\]
Hence, if we set 
\[U = \{ x\in X \st |(f_i)_{z_2}(x)| \leq |g_{z_2}(x)| \ \and \ |g_{z_2}(x)|\geq r \},\]
then $V\cap( X \times T) = U \times T$. 
It then follows that if $W$ is an affinoid domain of $Y$ such that 
$W \subset T$, then 
$V\cap( X \times W) = U \times W$. 

\item 
We can assume that $\mathcal{V} =V$ is an affinoid neighbourhood of $z$.
In (a), $V \cap ( X \times W)$ is still a neighbourhood of $z$, since 
$W$ is an affinoid neighbourhood of $z_2$ (because $z_2$ is a rigid point). 
If we denote by $s_{z_2} : X \to X\times Y$ the section of $\pi_1$ defined by 
$s_{z_2}(t) =(t,z_2)$, then 
\[s_{z_2}^{-1} ( ( V \cap ( X \times W) )  =s_{z_2}^{-1}(U\times W )  = U\] is 
an affinoid neighbourhood of $x$ (since $s_{z_2}(x) = z$). 
Thus $U$ is also an affinoid neighbourhood of $z_1$.
\end{enumerate}
\end{proof}

\begin{rem}
Without the assumption that $z_2 \in Y(k)$ the previous corollary would be false. 
Take for instance 
$X= \mathcal{M}(k\{x\})$ and $Y = \mathcal{M}(k\{y\})$, and 
let 
$\varphi : \mathcal{M}(k\{t\}) \to X \times Y$ be defined by 
$\varphi(t) = (t,-t)$. Let $\eta $ be the Gauss point of 
$\mathcal{M}(k\{t\})$ and $z := \varphi(\eta)$. Let 
$V = \{ p \in \mathcal{M}(k\{x,y\}) \ \big| \ |(x+y)(p)| \leq \frac{1}{2} \}$.
It is a neighbourhood of $z$.
However, $\pi_1(z)$ (resp $\pi_2(z)$) is the Gauss point $z_1=\eta_X$  of 
$\mathcal{M}(k\{x\})$ (resp. $z_2=\eta_Y$ the Gauss point of $\mathcal{M}(k\{y\}) $ ). 
It is then easy to see,
according to the description of an affinoid domain of the unit disc as a
Swiss cheese, that there does not exist an affinoid neighbourhood $U $ (resp. $W$) of $\eta_X$ (resp. $\eta_Y$) 
such that $V \supseteq U\times W$. 
For instance for the reason that in $U$ there would necessarily exist
a rigid point $x_0 \in \{ x\in k \st |x|\leq 1\}$ such that $\overline{x_0} =\overline{0}$ and  in 
$W$ a rigid point $y_0$ such that $\overline{y_0} = \overline{1}$ but 
$(x_0,y_0) \notin V$ (where $\overline{x}$ corresponds to the reduction of $x$ in 
$\tilde{k}$). 
\end{rem}

\begin{lemme}
\label{lemmet2}
Let $x \in X= \affin{A}$, and let 
$f= \sum\limits_{n\in \mathbb{N} } a_n T^n \in \mathcal{A}\{r^{-1} T \}$.
Let us assume that $f_x \neq 0$. 
Then there exist  $V = \affin{B}$ an affinoid domain of $X$ which contains  
$x$,
$P \in \mathcal{B}[T]$, and $u\in \mathcal{B}\{r^{-1}T \}$ a multiplicative unit 
such that $f_{| V \times \Br} = uP$.
\end{lemme}

\begin{proof}
Since $f_x = \sum\limits_{n\in \mathbb{N} } a_n(x)T^n \neq 0$, this series 
is distinguished of some order $s\geq 0$ for some $s\geq 0$. 
We recall that this means that 
$|a_s(x)|r^s = \| f_x\|$ and that $s$ is the greatest rank for this property. \par
We now use lemma \ref{lemme_simple} in our specific situation where the polyradius 
$\underline{r}$ is in fact the real number $r$.
Hence we can introduce a finite subset $J \subseteq \mathbb{N}$ such that $s\in J$ and some series 
$\phi_n \in \mathcal{A}\{r^{-1} T \}$ for $n\in J$ satisfying 
$\| \phi_n \| <1$ such that 
$f= \sum\limits_{n\in J} a_n (X^n+\phi_n)$.\par
We then define $V$ as the rational domain: 
\[ V = \{ z\in X \ \big| \ |a_s(z)| = |a_s(x)| \ \text{and} \ 
|a_i(z)|r^i \leq |a_s(x)|r^s \ \text{for} \  i\in J\setminus \{s\} \}\]
and denote by $\mathcal{B}$ the affinoid algebra of $V$.
It is then true that $x\in V$. 
Moreover, 
on $V = \affin{B}$, one checks that 
$a_s$ is a multiplicative unit, and that on 
$\mathcal{B} \{r^{-1}T\}$, $f$ is distinguished of order $s$. One can 
then apply Weierstrass preparation (corollary \ref{preparation}) to conclude.
\end{proof}

\begin{rem}
The previous result (lemma \ref{lemmet2}) is false if we remove the assumption $f_x \neq 0$. \par
Indeed, let us consider a real number $r$ satisfying $0<r<1$, 
and let $f\in k\{r^{-1}x \}$ be a function 
whose radius of convergence is exactly $r$ and let us assume that $\|f\| <1$. 
Let then $\mathcal{A}= k\{y,t\}$, $X= \affin{A}$ the unit bidisc, $p$ the rigid point of $X$ 
corresponding to the origin, and let us consider
\[F(y,t,x) = y-tf(x) \in k\{y,t\}\{r^{-1}x\} = \mathcal{A}\{r^{-1}X \}. \] 
Then we claim that there does not exist $V=\affin{B}$ 
an affinoid domain of $X$ containing $p$ such that
$F_{|V\times \Br} = uP$ where $u$ is a multiplicative unit of 
$\mathcal{B}\{r^{-1}T\}$ and 
$P \in \mathcal{B}[t]$.
\end{rem}
\begin{proof}
Indeed otherwise, there would exist  some closed bidisc $V$ of radius  
$s = |\lambda| \in |k^{\times}|$ where $\lambda\in k^*$, and some 
$P \in k\{s^{-1}y,s^{-1}t\}[x]$ and 
a multiplicative unit $u\in k\{s^{-1}y,s^{-1}t \} \{r^{-1}x\}$ such that 
\begin{equation}
\label{contrex}
F_{|V\times \Br} = uP.
\end{equation} 
Let us  fix $t=\lambda$.  Then we consider 
\[G(y,x) = F(y,\lambda,x)=y-\lambda f(x) \in k\{y,r^{-1}x\}.\]
According to \eqref{contrex}, 
$G_{\Bs \times \Br} = u(y,\lambda,t) P(y,\lambda,t)$.
Replacing $y$ by $\dfrac{y}{\lambda}$ and $f$ by $\lambda f$, we then obtain that 
\[G(y,x) = y-f(x) \in k\{y,r^{-1}x\}\]
\[G= uP\]
where $u\in k\{y,r^{-1}x\}$ is a multiplicative unit,
$P \in k\{y\}[x]$ and $\|f\|<1$ has a radius of convergence exactly $r<1$. 
This implies that if we set 
\[S:=\{ (x,y) \in \B^2 \ \big| \ |x|\leq r \ \text{and} \ y=f(x) \} \]
then 
\[ S=\{ (x,y) \in \B^2 \ \big| \ |x|\leq r \ \text{and} \ P(x,y)=0 \} \]
so $S$ would be  semianalytic in $\B^2$, but in section \ref{section2}, 
we exploited many times that this is not the case.
\end{proof}

\begin{lemme}[Local algebraization of a function in a family of rings]
\label{algebraization}
Let $n$ be an integer and let us consider $a_0,\ldots, a_n$ 
some elements of $\{x\in k \st |x|\leq 1\}$ and  $r_0,\ldots,r_n$ some positive real numbers. 
Let $Y \subseteq \mathcal{M} ( k\{T\} ) =\B$  be the Laurent domain defined by
\[Y= \{ y\in \mathcal{M} ( k\{T\} ) \ \big| \ |(T-a_0)(y)| \leq r_0 \ and 
\ |(T-a_i)(y)| \geq r_i , \ i=1\ldots n \},\]
and let $X=\affin{A}$ be a $k$-affinoid space.
Let 
\[f\in \mathcal{O}(X \times Y)\] 
and let 
\[z\in X\times Y \] 
such that 
$\pi_1(z) = x\in X(k)$ and let us set 
$y := \pi_2(z)$. 
Assume that $f_x \in \mathcal{H}(x) \otimes \mathcal{O}(Y) \simeq 
\mathcal{O}(Y)$ is non-zero\footnote{Here $\mathcal{H}(x) \simeq k$ because $x\in X(k)$.}.
Then there exists $V=\affin{B}$ an affinoid neighbourhood of $x$, 
and $Y' \subset Y$ defined by 
\[Y'=\{ y\in \mathcal{M} ( k\{T\} ) \ \big| \ |(T-b_0)(y)| \leq s_0 \ \text{and} 
\ |(T-b_i)(y)| \geq s_i, \ i=1\ldots m \} \] 
an affinoid neighbourhood of $y$ such 
that 
\[ f_{|V \times Y'} = (uP)_{|V\times Y'}\] where 
the $s_i$'s are positive real numbers, $b_i \in k^\circ$, 
$u$ is a multiplicative unit of 
$V\times Y'$ and 
$P\in \mathcal{B}[T,(T-b_1)^{-1},\ldots,(T-b_m)^{-1}]$.
\end{lemme}
\begin{rem}
let us mention that in the proof we distinguish two very different cases.
\begin{enumerate}
\item If $y$ is a rigid point then $Y'$ can in fact be chosen to be a closed ball, 
i.e. $m=0$.
\item Otherwise, if $y$ is not a rigid point, then in fact 
$s_0 = r_0$, that is to say, we do not have to decrease the radius of 
the ambient closed ball, but in counterpart, we possibly have to remove 
some open balls.
\end{enumerate}
\end{rem}

\begin{proof}
If $y$ is a rigid point, we can indeed find a closed disc $Y'$ which contains $y$ 
and the result follows from lemma \ref{lemmet2}. \par
If $y$ is not a rigid point, $f_x \in \mathcal{H}(x) \otimes \mathcal{O}(Y) 
\simeq \mathcal{O}(Y)$. 
Then according to classical results on factorization of functions on rational 
domains of the closed disc (cf \cite[2.2.9]{FvdP}),  
there exist $\alpha_1, \ldots , \alpha_N \in k$, 
$d_1, \ldots , d_N \in \mathbb{N}$, $g$ an invertible 
function of $\mathcal{O}(Y)$ such that 
\begin{equation}
\label{functcour}
f_x =\prod_{i=1}^N (T-\alpha_i)^{d_i}g.
\end{equation}
We then set $m=n+N$, $b_i =a_i$ and $s_i=r_i$ for $i=0\ldots n$, 
and $b_{n+j} = \alpha_j$ for $j=1\ldots N$ and we take 
$s_{n+j} $ small enough so that 
$\{z\in Y \st  |T-\alpha_j|(z) \geq s_{n+j} \}$ is a neighbourhood of $y$ 
(this is possible because $y$ is not a rigid point).
Then we define 
\[ Y':=\{ y\in \mathcal{M} ( k\{T\} ) \ \big| \ |(T-b_0)(y)| \leq s_0 \ and 
\ |(T-b_i)(y)| \geq s_i , i=1\ldots m \}.\]
Next, we set 
\[G = f \prod _{i=1}^N (T-\alpha_i)^{-d_i} \in \mathcal{O}(X \times Y').\] 
Then, according to \eqref{functcour} $G_x =g$ which does not vanish on $Y'_x$. 
So there exists an affinoid neighbourhood 
$V=\affin{B}$ of $x$ such that $G$ is invertible on $V \times Y'$ 
because the locus of points $x$ where $G_x$ is invertible is open.
\begin{comment}
indeed $x$ is rigid so we can find $X_1 \supseteq X_2 \supseteq \ldots $ 
a decreasing sequence of affinoid neighbourhood such that 
$\{x\} = \cap_i X_i$ and if $G$ is not invertible on $X_i\times Y'$, pick $z_i \in 
X_i \times Y'$ such that $G(z_i)=0$ then by compactness we would find a point 
$z_{\infty} \in Y'_x$ such that $G(z_{\infty}) =0$).
\end{comment}
Now using the explicit description of 
$\mathcal{O}(V \times Y')$, we can write
\begin{equation*}
G= \sum_{\nu = (\nu_0, \ldots , \nu_n) \in \mathbb{N}^{m+1} } 
b_{\nu} (T-b_0)^{\nu_0} (T-b_1)^{-\nu_1} \ldots (T-b_m)^{-\nu_m}.
\end{equation*}
Now for 
$M \geq 0$ set 
\begin{equation*}
G_M= \sum_{|\nu|\leq M }b_{\nu} (T-b_0)^{\nu_0} (T-b_1)^{-\nu_1} \ldots (T-b_m)^{-\nu_m}.
\end{equation*} 
By definition, 
$G_M \in \mathcal{B}[T,(T-b_1)^{-1}, \ldots,(T-b_m)^{-1} ]$. 
In addition, 
\[G_M \xrightarrow[M \to \infty]{} G,\] 
so $G_M$ is invertible for $M$ big enough. 
For such an $M$, 
\begin{equation}
\label{funcfami}
 G=G_M+(G-G_M)= G_M(1+ G_M^{-1}(G-G_M) ).
 \end{equation} 
Moreover, if we take $M$ again larger, we can assume that  
$\|G_M^{-1} \| = \|G^{-1} \|$, and as a consequence  
\[\|G_M^{-1} (G-G_M)\| \xrightarrow[M \to \infty]{} 0.\] 
Thus, for $M$ large enough, if we set   
\[u_M=1+G_M^{-1}(G-G_M)\] 
then $u_M$ is a multiplicative unit, and according to \eqref{funcfami}
\[f=G_M u_M \prod_{i=1}^N(T-\alpha_i)^{d_i}.\]
We then set 
$u:=u_M$ and $P:=G_M \prod_{i=1}^N(T-\alpha_i)^{d_i}$ to conclude.
\end{proof}

\subsection{Blowing up}

From now on, $X$ will be a quasi-smooth $k$-analytic space of dimension $2$. \par
We now make two simple remarks that we will use in the proof of theorem \ref{theo_dim2}.

\begin{lemme}
\label{lemme_divide}
Let $\mathcal{A}$ be a $k$-affinoid algebra, $X= \affin{A}$, $0<r<s$ some real numbers and 
$h\in \mathcal{A}$. 
\begin{enumerate}
\item Consider the Weierstrass domain of $X$:
\[V = \{ x\in X \ \big| \ |h(x)| \leq s \}\] 
and let $S$ be a locally semi-analytic subset 
of $V$ such that 
\[S \subseteq \{ x\in X \ \big| \ |h(x)| \leq r\}.\]
Then $S$ is a also a locally semianalytic subset of $X$.
\item 
Consider the Laurent domain of $X$:
\[V = \{ x\in X \ \big| \ |h(x)| \geq r \}\] 
and let $S$ be a locally semi-analytic subset 
of $V$ such that 
\[S \subseteq \{ x\in X \ \big| \ |h(x)| \geq s\}.\]
Then $S$ is a also a locally semianalytic subset of $X$.
\end{enumerate}
\end{lemme}
\begin{proof} Choose a real number $t$ such that 
$r<t<s$.
\noindent
\begin{enumerate}
\item Let us set $W =\{ x\in X \ \big| \ t \leq |h(x)|  \}$. Then $\{V,W\}$ is a wide covering of 
$X$, and $S\cap V$ is by hypothesis locally semianalytic in $V$, and by assumption, 
$S\cap W = \emptyset$ so is also locally semianalytic in $W$, hence $S$ is locally semianalytic in $X$.
\item Likewise, let us set  
$W = \{ x\in X \ \big| \ |h(x)| \leq t \}$. Then $\{V,W\}$ is a wide covering of $X$, 
$S\cap V$ is locally semianalytic in $V$ and $S\cap W = \emptyset$, so $S$ is locally 
semianalytic in $X$.
\end{enumerate}
\end{proof}
This lemma will be used jointly with the following remark:
\begin{rem}
\label{rem_divide}
Let us consider a $k$-affinoid space $X =\affin{A}$, $f,g \in \mathcal{A}$, 
$0<s<r$ and 
\[(Z,S) \xrightarrow{\varphi} X\] 
the elementary constructible datum given by 
$Z= \affin{B}$ where 
$\mathcal{B}=  \mathcal{A}\{r^{-1}t \} /(f-tg)$ and 
\[S = \{ z\in Z \ \big| \ |f(z)|\leq s|g(z)|\neq 0 \}. \]
Moreover, let 
\[(Y,U) \stackrel{\psi}{\dashrightarrow} (Z,S)\] 
be a constructible datum.
\begin{enumerate}[A.]
\item 
Let us assume that $g | f $. In other words, there exists $h\in \mathcal{A}$ such that 
$f=gh$. Let us then consider  
$\mathcal{C}= \mathcal{A}\{r^{-1}t\}/(h-t)$ and $V=\M(\mathcal{C})$. 
Note that $V$ is the Weierstrass domain of $X$ defined by
\[V=\{x\in X \ \big| \ |h(x)| \leq r \}.\]
Let us denote by $\beta$ the map of the immersion of the affinoid domain 
$V$ inside $X$, and let 
\[T = \{x\in V \ \big| \ |h(x)|\leq s \ \and \ g(x)\neq 0  \}.\]
Since $f-tg = g(h-t)$, 
$(h-t) | (f-tg)$, and there is a closed immersion 
$V \xrightarrow{\alpha} Z$. Moreover, 
$\alpha (T) = S$.\par
Indeed $\alpha(T) \subseteq S$, follows from their respective definitions. 
Conversely, if 
$z\in S$, $(f-tg)(z)=0=g(z) (h-t)(z)$ but since 
$g(z)\neq 0$, $(h-t)(z) = 0$ which implies that 
$z\in V$, and by the definition of $S$, it follows that 
$z\in \alpha(T)$.\par
Let us then consider the following cartesian diagram of $k$-germs:
\[ 
\xymatrix{
(Y,U)  \ar@{.>}[r]^{\psi}  & (Z,S) \ar[r]^{\varphi} & X \\
(Y',U') \ar[u]^{\alpha'} \ar@{.>}[r]^{\psi'} & (V,T) \ar[u]^{\alpha} \ar[ur]_{\beta} &
}
\]
Here, 
$(Y',U') \stackrel{\psi'}{\dashrightarrow} (V,T)$ is still a constructible datum according 
to corollary \ref{lemmeintersection}.
Since $\alpha(T)=S$, it follows that 
$\alpha(\psi' (U'))) = \psi(U)$, so 
\begin{equation}
\label{eq_div1}
\varphi(\psi(U)) = \varphi ( \alpha (\psi' (U'))) = \beta (\psi'(U')).  
\end{equation}
Roughly speaking, we were starting with the constructible datum 
\[(Y,U) \stackrel{\psi}{\dashrightarrow}(Z,S) \xrightarrow{\varphi} X\] 
such that the elementary constructible 
datum of $\varphi$ was defined with functions $f$ and $g$ such that 
$g|f$. And we have been able to replace $\varphi$ by the  constructible datum 
\[(Y',U') \stackrel{\psi'}{\dashrightarrow}(V,T) \xrightarrow{\beta} X\]
where $V$ is a Weierstrass domain. Note moreover that 
$T$ and so also $\psi'(U')$ satisfy the hypothesis of lemma \ref{lemme_divide} (1).

\item 
If $f | g $, there exists $h\in \mathcal{A}$ such that 
$g=fh$. Let then 
$\mathcal{C}= \mathcal{A}\{r^{-1}t\}/(1-th)$, $V=\M(\mathcal{C})$. 
Note that $V$ is the Laurent domain of $X$ defined by
\[V=\{x\in X \ \big| \ |h(x)|\geq \frac{1}{r} \}.\]
Let us denote by $\beta$ the map of the immersion of the Laurent domain 
$V$ inside $X$, and let 
\[T = \{x\in V \ \big| \ |h(x)|\geq \frac{1}{s} \ \and \ g(x)\neq 0 \}.\]
Since 
$(1-th) | (f-tg)$, there is a closed immersion 
$V \xrightarrow{\alpha} Z$. Moreover, 
$\alpha (T) = S$.\par
\begin{comment}
Indeed $\alpha(T) \subseteq S$, follows from their respective definitions. 
Conversely, if 
$z\in S$, $(f-tg)(z)=0=f(z) (1-th)(z)$ but since 
$f(z)\neq 0$ (because $g(z)\neq 0$), $(1-th)(z) \neq 0$ which implies that 
$z\in V$, and by the definition of $S$, it follows that 
$z\in \alpha(T)$.\par
\end{comment}
We then consider the following cartesian diagram of $k$-germs:
\[ 
\xymatrix{
(Y,U)  \ar@{.>}[r]^{\psi}  & (Z,S) \ar[r]^{\varphi} & X \\
(Y',U') \ar[u]^{\alpha'} \ar@{.>}[r]^{\psi'} & (V,T) \ar[u]^{\alpha} \ar[ur]_{\beta} &
}
\]
Here, 
$(Y',U') \stackrel{\psi'}{\dashrightarrow} (V,T)$ is still a constructible datum.
Since $\alpha(T)=S$, it follows that 
$\alpha(\psi' (U'))) = \psi(U)$, so 
\begin{equation}
\label{eq_div}
\varphi(\psi(U)) = \varphi ( \alpha (\psi' (U'))) = \beta (\psi'(U')).  
\end{equation}
In that case, we were starting with the constructible datum 
$(Y,U) \stackrel{\psi}{\dashrightarrow}(Z,S) \xrightarrow{\varphi} X$ 
such that $f|g$, and we have been able to replace it by the following constructible datum 
$(Y',U') \stackrel{\psi'}{\dashrightarrow}(V,T) \xrightarrow{\beta} X$
where $V$ is a Laurent domain of $X$. Note moreover that 
$T$ and so also $\psi'(U')$ satisfies the hypothesis of lemma \ref{lemme_divide} (2).
\end{enumerate}
\end{rem}

\begin{rem}
\label{rem_bu}
We are going to use some blowing up of $k$-analytic spaces in the following context: 
$X$ will be a quasi-smooth $k$-analytic space of dimension $2$, and we will blow up a rigid point 
$p$ of $X$. In particular, the resulting blowing up $\tilde{X}$ will be still 
quasi-smooth. To give a precise description of the situation, 
since $k$ is algebraically closed, we can assume that $X = \B^2$ and $p$ 
is the origin. 
The blowing up can then be described with two charts as follows. We consider
\[
\begin{array}{rrcl}
  X_1 = \mathcal{M}(k\{x,t_1\})& \xrightarrow{\pi_1}       &\B^2=\mathcal{M}(k\{x,y\})  \\
       (x,t_1)                     &\mapsto   & (x,t_1 x)    
\end{array}
\hspace{5pt}
 \left|
 \hspace{5pt} 
\begin{array}{rrcl}
  X_2 = \mathcal{M}(k\{y,t_2\}) & \xrightarrow{\pi_2}  &\B^2=\mathcal{M}(k\{x,y\})  \\
        (y,t_2)                 & \mapsto              & (t_2y,y)    
\end{array}
\right.
\] 
Then $\tilde{\B^2}$ is obtained by gluing $X_1$ and $X_2$ along the domains 
$U_1 = \{z\in X_1 \ \big| \ t_1(z) \neq 0 \}$ and  
$U_2 = \{z\in X_2 \ \big| \ t_2(z) \neq 0 \}$ via the isomorphism 
\[
\begin{array}{ccc}
U_1 & \to & U_2 \\
(x,t_1) & \mapsto & (xt_1,t_1^{-1}).
\end{array}
\]
\end{rem}

\begin{prop}
\label{blowup}
Let $X=\affin{A}$ be a quasi-smooth $k$-affinoid space 
of dimension $2$ and let $f,g \in \mathcal{A}$. Then there exists a succession 
of blowing up of rigid points 
$\pi : \tilde{X} \to X$ such that for all 
$x\in \tilde{X}$, $f_x |g_x$ or 
$g_x|f_x$. Remark that $\tilde{X}$ is still quasi-smooth.
\end{prop}
\begin{proof}
We may assume that $X$ is irreducible. 
If $f=0$ or $g=0$, there is nothing to prove, so we may assume that $f\neq 0$ and 
$g\neq 0$. 
Likewise, if $f=g$, there is nothing to do, so we may also assume 
that $f-g\neq 0$.\par
Let $h=fg(f-g)$. Hence, $h\neq 0$. 
We can find a succession 
of blowing up of rigid points 
$\pi : \tilde{X} \to X$ such that 
$\pi^* (h)$ is a normal crossing divisor. Indeed, the classical proof 
(see \cite[1.8]{Kol}) that 
works in the algebraic case, or the complex analytic  case, can be translated \emph{verbatim} 
in our context, and since we are dealing with a compact space, 
the local procedure  of \cite[1.8]{Kol} needs only to 
be applied to a finite number of points. Let then $x\in \tilde{X}$. 
\par If $x$ is not a rigid point, 
$\mathcal{O}_{\tilde{X},x}$ is a field or a discrete valuation ring and the result is clear. \par
Otherwise, if $x$ is a rigid point, its local ring is a regular local ring of dimension 2. By assumption, $h =fg(f-g)$ is a normal crossing divisor, thus can be written in $\mathcal{O}_{\tilde{X},x}$ as 
\begin{equation}
\label{NCD}
(fg(f-g))_x=u\xi_1^n\xi_2^m
\end{equation} 
where $\xi_1,\xi_2$ is a system of local parameters around $x$ and $u$ 
is a unit in $\mathcal{O}_{\tilde{X},x}$. 
Dividing by the common divisor of $f_x$ and $g_x$ in $\mathcal{O}_{\tilde{X},x}$, 
we can assume for instance 
that $f_x=v\xi_1^p$ and $g_x=w\xi_2^q$ and $f_x-g_x=z\xi_1^a\xi_2^b$ where $v$, $w$ and $z$ are units of $\mathcal{O}_{\tilde{X},x}$. \par 
If $p > 0$ then modulo $\xi_1$ we obtain $f=0$, so  
$f_x-g_x = w\xi_2^q$ modulo $\xi_1$. This implies that $a=0$ and that $b=q$.
So $f_x=(f_x-g_x)+g_x$ is divisible by $\xi_2^q$, and this implies that $q=0$. 
So $g_x$ is invertible and, $g_x|f_x$.\par 
And if $p=0$, then $f_x$ is invertible, so $f_x| g_x$. 
\end{proof}

\begin{lemme}
\label{blowup2}
Let $X$ be a good quasi-smooth strictly $k$-analytic space of dimension $2$.
\begin{enumerate}
 \item Let $q\in X_{\text{rig}}$ and 
$\pi : \tilde{X} \to X$ the blowing-up of $X$ at $q$, and 
let $S \subseteq \tilde{X}$ be a locally semianalytic 
subset. Then $\pi(S)$ is locally semianalytic.
\item If $\pi : \tilde{X} \to X$ is a succession of blowing-up of rigid points, 
and $S \subseteq \tilde{X}$ is locally semianalytic, then 
$\pi(S)$ is also locally semianalytic.
\end{enumerate}
\end{lemme}

\begin{proof}
(2) is a consequence of (1) so we only have to show (1). \par
The problem is local on $X$, and since outside $q$, $\pi$ is a local isomorphism, 
we can restrict to an affinoid neighbourhood of $q$, and since 
$X$ is regular at $q$, we can assume that 
$X=\B^2$ and $q$ is the origin. \par
Then $\pi : \tilde{X} \to X$ can be described with two charts, one of them being 
\[
\begin{array}{rrcl}
\pi_1 :&  X_1 = \mathcal{M}(k\{x,t\})& \to       &X=\mathcal{M}(k\{x,y\})  \\
        &   (x,t)                     &\mapsto   & (x,tx)    
\end{array}
\] 
The other chart being analogous we only consider $\pi_1$.
Now, changing 
$S$ to $S \cap X_1$, we can restrict to show that 
if $S$ is locally semianalytic in $X_1$, so is $\pi_1(S)$. 
Since 
$\pi_1$ induces an isomorphism between $X_1 \setminus V(x)$ and 
$\{p\in \B^2 \ \big| \ |y(p)| \leq |x(p)| \neq 0 \}$, 
we only have to show that 
$\pi_1(S)$ is semianalytic around $q$, the origin of $\B^2$.\par
Now if for each $p\in \mathbf{E} := V(x) \subseteq X_1$ we can find $V_p$ an affinoid neighbourhood of $p$, 
and $\varepsilon_p>0$ such that 
$\pi_1(V_p \cap S) \cap \B_{\varepsilon_p}^2 $ is 
semianalytic in $\B_{\varepsilon_p}^2 \subseteq X$, then by compactness of $\mathbf{E}$, we can extract 
$V_1 , \ldots, V_n$ a finite covering of $\mathbf{E}$ and $\varepsilon >0$ such that 
\[\cup_{i=1}^n (\pi_1(V_i\cap S)) \cap \B_{\varepsilon}^2 = \pi_1(S) \cap \B_{\varepsilon}^2\] is 
semianalytic in $\B_{\varepsilon}^2$. 
So we fix $p\in \mathbf{E} = V(x)$ and try to find $V$ an affinoid neighbourhood of $p$, 
and $\varepsilon>0$ such that 
$\pi_1(V \cap S) \cap \B_{\varepsilon}^2 $ is 
semianalytic in $\B_{\varepsilon}^2$. \par
Since $S$ is locally semianalytic in $X_1$, we can find $V$ an affinoid neighbourhood 
of $p$ such that $V\cap S$ is  semianalytic in $V$. 
According to corollary \ref{topol}, 
we can assume that\footnote{Here we use the explicit description of 
affinoid domains of $\B$.} $V = \B_{\varepsilon}\times W$ where 
\[W=\{ w\in \mathcal{M}(k\{t\} ) \ \big| \ 
|(t-a_0)(w)| \leq r_0 \ \text{and} \ |(t-a_i)(w)| \geq r_i, \ i=1\ldots n\}\]
for some $a_0, \ldots , a_n \in k^\circ$ and $r_0,\ldots,r_n \in \R_+$. \par
To simplify the notation, we can also assume that the  semianalytic subset $S$ of $V$ 
we are dealing with is of the following form:
\[S= \bigcap_{j=1}^m \{v \in V \ \big| \  |f_j(v)| \Diamond_j |g_j(v)| \}.\] 
Now recall that $V=\B_{\varepsilon} \times W $ with 
$\B_{\varepsilon} = \mathcal{M} ( k\{ \varepsilon^{-1}x \} )$.
So we can factor each $f_j$ and $g_j$ by the greatest power of $x$ which is a factor, hence introduce some integers 
$b_j, c_j$ such that 
 \[S= \bigcap_{j=1}^m \{v \in V \ \big| \  |x^{b_j}\tilde{f}_j(v)| \Diamond_j |x^{c_j}\tilde{g}_j(v)| \}\] 
where the series $\tilde{f}_j(0,t)$ and $\tilde{g}_j(0,t)$ are non zero,
 and $f_j = x^{b_j}\tilde{f}_j, g_j = x^{c_j}\tilde{g}_j$. 
But to simplify the notation, we will 
use $f_j$ (resp. $g_j$) instead of $\tilde{f}_j$ (resp. $\tilde{g}_j$), so that 
\[S= \bigcap_{j=1}^m \{v \in V \ \big| \  |x^{b_j}f_j(v)| \Diamond_j |x^{c_j}g_j(v)| \}\] 
where the series $f_j(0,t)$ and $g_j(0,t)$ are non zero.\par
Then according to lemma \ref{algebraization} 
we can decrease $\varepsilon $ and $W$ so that for each 
$f_j,g_j \in \{f_1 \ldots f_m , g_1, \ldots , g_m\}$,  
$f_j=u_jP_j$ (resp. $g_j=v_jQ_j$) where $u_j$ (resp. $v_j$) is a multiplicative unit, and 
$P_j$ (resp. $Q_j$) $\in k\{\varepsilon^{-1} x\}[t,(t-a_1)^{-1}, \ldots , (t-a_n)^{-1} ]$. \par
Said differently, and with different notation, there exists an integer $N$ such that 
$f_j=u_j.\frac{P_j}{\left( (t-a_1)\ldots (t-a_n) \right)^N}$ where 
$u_j$ is a multiplicative unit and $P_j \in   k\{\varepsilon^{-1} x\}[t]$ 
(and resp. $g_j=v_j.\frac{Q_j}{\left( (t-a_1)\ldots (t-a_n) \right)^N}$). Hence 

\[ |f_j(v)| \Diamond_j |g_j(v)| \Leftrightarrow 
\left| u_j(v) \frac{P_j(v)}{( (t-a_1)\ldots (t-a_n))^N(v) } \right| \Diamond_j 
\left| v_j(v) \frac{Q_j(v)}{ ((t-a_1)\ldots (t-a_n))^N(v)} \right| 
\Leftrightarrow |P_j(v)| \Diamond_j \lambda_j |Q_j(v)|\] 
where 
\[ \lambda_j = \frac{\|v_j\|}{\|u_j\|} \in |k^{\times}| . \]
Moreover,
\[S\cap V = (S\cap \{v\in V \ \big| \ x(v)=0\}) \cup (S\cap \{v\in V \ \big| \ x(v)\neq 0\} )\] and 
$\pi_1 ( \{v\in V \ \big| \ x(v)=0\}) = q$, the origin of $\B^2$.\par
So, adding if necessary the origin to 
$\pi_1 (S \cap \{v\in V \ \big| \ v(x)\neq 0\} )$ 
(which will not change the fact that it is  semianalytic), 
we can restrict to show that $\pi_1( S \cap \{ v\in V \ \big| \ v(x) \neq 0\})$ is 
 semianalytic around the origin. Moreover since on 
$\{ v\in V \ \big| \ v(x) \neq 0\}$, $\pi_1$ is bijective, the following holds:
\[ \pi_1\left( \bigcap_{j=1}^m \{v\in V \ | \ |x^{b_j}f_j(v)| \Diamond_j |x^{c_j}g_j(v)| \} \cap \{v\in V \ \big| \ x(v)\neq 0 \} \right) \]
\[= 
\bigcap_{j=1}^m \pi_1\left( \{ v\in V \ \big| \ |x^{b_j}f_j(v)| \Diamond_j |x^{c_j}g_j(v)| \} \right) \cap \{v\in V \ \big| \ x(v) \neq 0\}.\]
Now since $y=tx$ and $P_j \in k\{\varepsilon^{-1}x\}[t]$ there exists an integer $M \geq 0$ such that 
$x^M P_j(x,t) \in k\{ \varepsilon^{-1}x\}[tx] = k\{ \varepsilon^{-1}x\}[y]$, i.e. 
$x^MP_j(x,t) = \pi^* (\tilde{P}_j (x,y) )$ 
for some 
$\tilde{P}_j(x,y) \in k\{ \varepsilon^{-1}x\}[y]$ and such that 
$x^M Q_j(x,t) \in  k\{ \varepsilon^{-1}x\}[y]$, i.e. 
$x^MQ_j(x,t) = \pi^* (\tilde{Q}_j (x,y) )$ 
for some 
$\tilde{Q}_j(x,y) \in k\{ \varepsilon^{-1}x\}[y]$.\par
Now on 
$\{v\in V \ \big| \ v(x)\neq 0\}$,
\[ 
|x^{b_j}f_j(v)|\Diamond_j |x^{c_j}g_j(v)| \Leftrightarrow |x^{M+b_j}f_j(v)|\Diamond_j |x^{M+c_j}g_j(v)| 
\Leftrightarrow | x^{b_j}\tilde{P}_j(\pi_1(v))| \Diamond_j \lambda_j |x^{c_j}\tilde{Q}_j(\pi_1(v))|.\]
From that we conclude that 
\[z\in \pi_1\left( \bigcap_{j=1}^m \{v\in V \ \big| \ |x^{b_j}f_j(v)| \Diamond_j |x^{c_j}g_j(v)| \} \cap \{v\in V \ \big| \ x(v) \neq 0\} \right) \]

\[\Leftrightarrow z\in \bigcap_{j=1}^m \{z\in \pi_1 ( V ) \ \big| \ |x^{b_j}\tilde{P}_j(z) | 
\Diamond_j | x^{c_j}\tilde{Q}_j(z)| \} \cap \{z\in X \ \big| \ x(z) \neq 0 \}. \]
Since $\pi_1 (\B_{ \varepsilon}^2) \subseteq \B_{\varepsilon}^2$ and is 
 semianalytic in $\B_{\varepsilon}^2$, we conclude that 
$\pi_1( S \cap \{ v\in V \ \big| \ v(x) \neq 0\})$ is  semianalytic 
in $\B_{\varepsilon }^2$, 
which ends the proof.
\end{proof}

\begin{theo}
\label{theo_dim2}
 Let $X$ be a good quasi-smooth strictly $k$-analytic space of dimension $2$ with $k$ algebraically closed, and 
$S \subseteq X$. Then $S$ is overconvergent subanalytic subset if and only if 
$S$ is locally semianalytic.
\end{theo}

\begin{proof}
 Since the problem is local, we can assume that $X$ is affinoid and that 
$S=\varphi(U)$ where 
$(Y,U) \stackrel{\varphi}{\dashrightarrow} X$ is a constructible datum, and just check that 
$S$ is locally semianalytic. We do it by induction on the complexity of $\varphi$.
So let 
$(Y,U) \stackrel{\varphi}{\dashrightarrow} X$ be a constructible datum, that 
we decompose as 
\[ \varphi = (Y,U) \stackrel{\psi}{\dashrightarrow} Z \xrightarrow{\chi} X \]
where $\chi$ is an elementary constructible datum, and $\psi$ a 
constructible datum whose complexity is one less than $\varphi$. So we can introduce 
$f,g \in \mathcal{A}$, $0<s<r$ such that 
$Z = \eclatement$.
According to proposition \ref{blowup}, we can find a succession 
of blowing-up of rigid points 
$\pi : \tilde{X} \to X$ such that for all $x\in \tilde{X}$,
$f_x | g_x$ or $g_x | f_x$. 
According to remark \ref{rem_bu}, 
$\tilde{X}$ is still quasi-smooth. 
This gives us the following cartesian diagram: 
\[\xymatrix{
(Y,U) \ar@{.>}[r]^{\varphi} & X \\
(Y',U') \ar@{.>}[r]^{\varphi'} \ar[u]^{\pi'} & \tilde{X} \ar[u]_{\pi}
} \]
Then $\varphi(U) = \pi (\varphi'(U'))$. 
Moreover, since $\tilde{X}$ is compact, 
we can then find a finite wide covering 
$\{X_i\}_{i=1}^n$ of $\tilde{X}$ by affinoid domains 
such that for all $i$,  
$f_{|X_i}|g_{|X_i}$ or $g_{|X_i}|f_{|X_i}$. 
We denote by 
$\pi_i : X_i \to X$ the composition of the embedding of 
the affinoid domain $X_i \to \tilde{X}$ with $\pi : \tilde{X} \to X$. 
This gives the following cartesian diagrams:
\[\xymatrix{
(Y,U) \ar@{.>}[r]^{\psi} & Z \ar[r]^{\chi} & X \\
(Y_i,U_i) \ar@{.>}[r]^{\psi_i} \ar[u]^{\pi_i''}  & 
Z_i \ar[u]^{\pi_i'}  \ar[r]^{\chi_i} & X_i \ar[u]_{\pi_i} 
}\]
Then 
\[\varphi(U) = \pi ( \varphi'(U') ) =
\pi \left(\bigcup_{i=1}^n \chi_i (\psi_i (U_i) ) \right)\]
But 
$(Y_i,U_i) \stackrel{\psi_i}{\dashrightarrow} Z_i$ 
is a constructible datum of lower complexity than $\varphi$, so 
that we would like to use our induction hypothesis, and claim that 
$\psi_i(U_i)$ is locally semianalytic. However, $Z_i$ is not necessarily still 
quasi-smooth so we cannot do that.
However, since 
$f_{|X_i}|g_{|X_i}$ or $g_{|X_i}|f_{|X_i}$, 
according to remark \ref{rem_divide}, 
we can in fact replace $Z_i$ by a 
Weierstrass (or a Laurent) domain of 
$X_i$, and hence assume that $Z_i$ is quasi-smooth. 
Thus by induction hypothesis  
$\psi_i(U_i)$ is locally semianalytic in $Z_i$. 
\par 
Next we use 
lemma \ref{lemme_divide} to assert that 
$\chi_i(\psi_i(U_i))$ is locally semianalytic in $X_i$.
So 
\[\displaystyle \varphi'(U') = \bigcup_{i=1}^n (\chi_i ( \psi_i(U_i) ) \] 
is 
locally semianalytic in $\tilde{X}$, since $\{X_i\}$ was a wide covering of 
$\tilde{X}$. 
Finally, according to  
lemma \ref{blowup2},  
$\pi (\varphi'(U'))=S$ is also locally semianalytic.  
\end{proof}

\bibliographystyle{alpha}
\bibliography{bibli}

\end{document}

%% file: macros.tex
\newcommand{\A}{\mathcal{A}}

\newcommand{\B}{\mathbb{B}}

\newcommand{\Br}{\mathbb{B}_r}
\newcommand{\Bs}{\mathbb{B}_s}
\newcommand{\Bur}{\mathbb{B}_{\underline{r}}}
\newcommand{\Bus}{\mathbb{B}_{\underline{s}}}
\newcommand{\C}{\mathbb{C}}

\newcommand{\M}{\mathcal{M}}
\newcommand{\N}{\mathbb{N}}

\newcommand{\Ql}{\ensuremath{\mathbb{Q}_\ell}}
\newcommand{\Qp}{\mathbb{Q}_p}

\newcommand{\R}{\mathbb{R}}

\newcommand{\Zp}{\mathbb{Z}_p}

\renewcommand{\and}{\text{and}}
\newcommand{\im}{\text{im}}

\newcommand{\st}{ \ \big| \ }
\newcommand{\affin}[1]{\ensuremath{ \mathcal{M} (\mathcal{#1})  }}

\newcommand{\Berko}{\text{Berko}}

\newcommand{\cha}{\text{char}}

\newcommand{\rig}{\text{rig}}

\newcommand{\val}{\ensuremath{\sqrt{|k^{\times} |}}}
\newcommand{\valeur}{\ensuremath{ \sqrt{|k^*|}  }}

\newcommand{\ADS}{\ensuremath{\mathcal{A}\langle\langle D \rangle\rangle }}
\newcommand{\Brho}{\mathcal{B}\{ \underline{\rho}^{-1} T \}}

\newcommand{\eclatement}{\ensuremath{ \mathcal{M}( \mathcal{A} \{ r^{-1}t \} / (f-tg) ) }}
\newcommand{\dce}{\ensuremath{ (Y,T) \xrightarrow{\varphi} (X,S)  }}
\newcommand{\dc}{\ensuremath{ (Y,T) \stackrel{\varphi}{\dashrightarrow} (X,S)   }}
\newcommand{\recouv}{\ensuremath{ (X_i,S_i) \stackrel{\varphi_i}{\dashrightarrow} (X,S) }}

\newcommand{\GSA}{semianalytic subset}
\newcommand{\ECD}{elementary constructible datum}
\newcommand{\OC}{overconvergent constructible subset}
\newcommand{\OS}{overconvergent subanalytic subset}
\newcommand{\Er}{\ensuremath{   \mathbb{B}_{\underline{r}} }}
\newcommand{\Es}{\ensuremath{   \mathbb{B}_{\underline{s}} }}

\newcommand{\E}{\ensuremath{  \mathbb{B}^n}}

\newcommand{\Ar}{\ensuremath{ A\{r^{-1}T\} }}
\newcommand{\Arn}{\ensuremath{ \mathcal{A}\{r_1^{-1}T_1, \ldots , r_n^{-1}T_n\} }}
\newcommand{ \ara }{\ensuremath{ \mathcal{A} \{\underline{r}^{-1} T \} }}

%% file: OverconvergentSubanalytic.bbl
\begin{thebibliography}{{Duc}11}

\bibitem[Ber90]{Berko90}
V.G. Berkovich.
\newblock {\em {Spectral theory and analytic geometry over non-Archimedean
  fields}}.
\newblock Amer Mathematical Society, 1990.

\bibitem[Ber93]{Berko93}
V.G. Berkovich.
\newblock {Etale cohomology for non-Archimedean analytic spaces}.
\newblock {\em Publications Math{\'e}matiques de l'IH{\'E}S}, 78(1):5--161,
  1993.

\bibitem[BGR84]{BGR}
S.~Bosch, U.~G{\"u}ntzer, and R.~Remmert.
\newblock {\em Non-{A}rchimedean analysis}, volume 261 of {\em Grundlehren der
  Mathematischen Wissenschaften [Fundamental Principles of Mathematical
  Sciences]}.
\newblock Springer-Verlag, Berlin, 1984.
\newblock A systematic approach to rigid analytic geometry.

\bibitem[BL93]{BL1}
S.~Bosch and W.~Lutkebohmert.
\newblock {Formal and rigid geometry. I : Rigid Spaces}.
\newblock {\em Mathematische Annalen}, 295(3):291--317, 1993.

\bibitem[BM00]{Bie01}
Edward Bierstone and Pierre~D. Milman.
\newblock Subanalytic geometry.
\newblock In {\em Model theory, algebra, and geometry}, volume~39 of {\em Math.
  Sci. Res. Inst. Publ.}, pages 151--172. Cambridge Univ. Press, Cambridge,
  2000.

\bibitem[CL11]{Clu_Lip}
R.~Cluckers and L.~Lipshitz.
\newblock Fields with analytic structure.
\newblock {\em J. Eur. Math. Soc. (JEMS)}, 13(4):1147--1223, 2011.

\bibitem[Duc03]{Duc_sa}
Antoine Ducros.
\newblock Parties semi-alg\'ebriques d'une vari\'et\'e alg\'ebrique
  {$p$}-adique.
\newblock {\em Manuscripta Math.}, 111(4):513--528, 2003.

\bibitem[{Duc}11]{Duc_flatness}
A.~{Ducros}.
\newblock {Families of Berkovich spaces}.
\newblock {\em ArXiv e-prints}, July 2011.

\bibitem[DvdD88]{DVdd}
J.~Denef and L.~van~den Dries.
\newblock {$p$}-adic and real subanalytic sets.
\newblock {\em Ann. of Math. (2)}, 128(1):79--138, 1988.

\bibitem[FvdP04]{FvdP}
Jean Fresnel and Marius van~der Put.
\newblock {\em Rigid analytic geometry and its applications}, volume 218 of
  {\em Progress in Mathematics}.
\newblock Birkh\"auser Boston Inc., Boston, MA, 2004.

\bibitem[Hub96]{HU96}
Roland Huber.
\newblock {\em \'{E}tale cohomology of rigid analytic varieties and adic
  spaces}.
\newblock Aspects of Mathematics, E30. Friedr. Vieweg \& Sohn, Braunschweig,
  1996.

\bibitem[Kol07]{Kol}
J{\'a}nos Koll{\'a}r.
\newblock {\em Lectures on resolution of singularities}, volume 166 of {\em
  Annals of Mathematics Studies}.
\newblock Princeton University Press, Princeton, NJ, 2007.

\bibitem[Lan02]{Lang}
Serge Lang.
\newblock {\em Algebra}, volume 211 of {\em Graduate Texts in Mathematics}.
\newblock Springer-Verlag, New York, third edition, 2002.

\bibitem[Lip88]{Lip_iso}
L.~Lipshitz.
\newblock Isolated points on fibers of affinoid varieties.
\newblock {\em J. Reine Angew. Math.}, 384:208--220, 1988.

\bibitem[Lip93]{LR93}
L.~Lipshitz.
\newblock Rigid subanalytic sets.
\newblock {\em Amer. J. Math.}, 115(1):77--108, 1993.

\bibitem[LR96]{LR_plane}
Leonard Lipshitz and Zachary Robinson.
\newblock Rigid subanalytic subsets of the line and the plane.
\newblock {\em Amer. J. Math.}, 118(3):493--527, 1996.

\bibitem[LR99]{LRSurf}
Leonard Lipshitz and Zachary Robinson.
\newblock Rigid subanalytic subsets of curves and surfaces.
\newblock {\em J. London Math. Soc. (2)}, 59(3):895--921, 1999.

\bibitem[LR00a]{LR_mod}
Leonard Lipshitz and Zachary Robinson.
\newblock Model completeness and subanalytic sets.
\newblock {\em Ast\'erisque}, (264):109--126, 2000.

\bibitem[LR00b]{LR_asterisque}
Leonard Lipshitz and Zachary Robinson.
\newblock Rings of separated power series and quasi-affinoid geometry.
\newblock {\em Ast\'erisque}, (264):vi+171, 2000.

\bibitem[Mar13]{Mar_these}
Florent Martin.
\newblock {\em Constructibility in Berkovich spaces}.
\newblock PhD thesis, Universit\'{e} Pierre et Marie Curie, 2013.

\bibitem[Osg16]{Osgood}
William~F. Osgood.
\newblock {On functions of several complex variables}.
\newblock {\em Trans. Amer. Math. Soc.}, 17(1):1--8, 1916.

\bibitem[Sch94a]{Sch_sub}
Hans Schoutens.
\newblock Rigid subanalytic sets.
\newblock {\em Compositio Math.}, 94(3):269--295, 1994.

\bibitem[Sch94b]{Sch_plane}
Hans Schoutens.
\newblock Rigid subanalytic sets in the plane.
\newblock {\em J. Algebra}, 170(1):266--276, 1994.

\bibitem[Sch94c]{Sch_unif}
Hans Schoutens.
\newblock Uniformization of rigid subanalytic sets.
\newblock {\em Compositio Math.}, 94(3):227--245, 1994.

\bibitem[Tat71]{Ta}
John Tate.
\newblock Rigid analytic spaces.
\newblock {\em Invent. Math.}, 12:257--289, 1971.

\end{thebibliography}
